\documentclass[english]{amsart}

\usepackage{amsmath}

\usepackage{amssymb}

\usepackage{amscd} \usepackage{tabularx}
\usepackage{enumerate}
\usepackage[all,cmtip,line]{xy}

\newcommand{\F}{\mathbb{F}} 

\def\RCS$#1: #2 ${\expandafter\def\csname RCS#1\endcsname{#2}}
\RCS$Revision: 185 $
\RCS$Date: 2010-06-02 14:43:31 -0400 (Wed, 02 Jun 2010) $

\DeclareMathOperator{\Frob}{Frob}

\DeclareMathOperator{\sgn}{sgn}
 
 \DeclareMathOperator{\ab}{ab}
 \newcommand{\eps}{\epsilon}
\newcommand{\epsbar}{\overline{\epsilon}}
\newcommand{\notdiv}{\nmid}

\newcommand{\wtilde}{\tilde{\omega}}
\newcommand{\To}{\longrightarrow}\newcommand{\into}{\hookrightarrow}

\newcommand{\psibar}{\overline{\psi}}
\newcommand{\epsilonbar}{\overline{\epsilon{}}}
\newcommand{\onto}{\twoheadrightarrow}
\newcommand{\isoto}{\stackrel{\sim}{\To}}

\newcommand{\bigO}{\mathcal{O}} 
\newcommand{\G}{\mathcal{G}}

\newcommand{\Z}{\mathbb{Z}} \newcommand{\A}{\mathbb{A}}
\newcommand{\Q}{\mathbb{Q}}

\newcommand{\fq}{\mathfrak{q}} \newcommand{\fQ}{\mathfrak{Q}}
\newcommand{\C}{\mathbb{C}}

\newcommand{\Favoid}{F^{(\mathrm{avoid})}} 
\newcommand{\tensor}{\otimes} 

\newcommand{\thetabar}{\bar{\theta}}

\newcommand{\bb}{\mathbb} 
\newcommand{\mc}{\mathcal}
\newcommand{\wt}{\widetilde} 
\newcommand{\mf}{\mathfrak}

\DeclareMathOperator{\Iw}{Iw}
  
\DeclareMathOperator{\Art}{Art}
\DeclareMathOperator{\univ}{univ}
\DeclareMathOperator{\red}{red}

\DeclareMathOperator{\hcf}{hcf}

\newcommand{\rhobar}{\overline{\rho}} 
\newcommand{\chibar}{\overline{\chi}} 
\newcommand{\rbar}{\bar{r}}
\newcommand{\Rbar}{\bar{R}}
\newcommand{\rbarwtv}{\rbar|_{G_{F_{\wt{v}}}}}

\newcommand{\Gal}{\operatorname{Gal}}
\newcommand{\GL}{\operatorname{GL}}
\newcommand{\SO}{\operatorname{SO}}
\newcommand{\SU}{\operatorname{SU}}

\newcommand{\gl}{\operatorname{\mathfrak{gl}}}
\newcommand{\PGL}{\operatorname{PGL}}
\newcommand{\HT}{\operatorname{HT}}

 \newcommand{\Qbar}{\overline{\Q}}
\newcommand{\Zbar}{\overline{\Z}}
 \newcommand{\Qp}{\Q_p}
\newcommand{\Ql}{{\Q_l}} 
\newcommand{\Qpbar}{\overline{\Q}_p}
\newcommand{\Qlbar}{\overline{\Q}_{l}}
\newcommand{\Zl}{{\Z_l}}
\newcommand{\Zlbar}{\overline{\Z}_{l}}

\newcommand{\Fl}{{\F_l}}
\newcommand{\Flbar}{\overline{\F}_l} 
\newcommand{\tw}{{\tilde{w}}}
\newcommand{\tv}{{\tilde{v}}}

\newcommand{\gr}{\operatorname{gr}}

\newcommand{\Spec}{\operatorname{Spec}}
\newcommand{\Ind}{\operatorname{Ind}}
\newcommand{\SL}{\operatorname{SL}}
\newcommand{\PSL}{\operatorname{PSL}}
\newcommand{\ad}{\operatorname{ad}}

\newcommand{\tr}{\operatorname{tr}}

\newcommand{\Hom}{\operatorname{Hom}}

\newcommand{\Fil}{\operatorname{Fil}}

\newcommand{\diag}{\operatorname{diag}}

\DeclareMathOperator{\Def}{Def}

\usepackage{amsthm}

\newtheorem{thm}{Theorem}[subsection]

 \newtheorem{lemma}[thm]{Lemma}
\newtheorem{lem}[thm]{Lemma} \newtheorem{prop}[thm]{Proposition}
 \theoremstyle{definition}
 \theoremstyle{definition}
\newtheorem{defn}[thm]{Definition} \theoremstyle{remark}

\numberwithin{equation}{subsection}

\theoremstyle{definition}

\begin{document}
\title[Ordinary lifts of Hilbert
  modular forms]  {Congruences between Hilbert modular forms: constructing ordinary lifts}

\author{Thomas Barnet-Lamb}\email{tbl@brandeis.edu}\address{Department of Mathematics, Brandeis University}
\author{Toby Gee} \email{tgee@math.harvard.edu} \address{Department of
  Mathematics, Harvard University} \author{David Geraghty}
\email{geraghty@math.harvard.edu}\address{Department of Mathematics,
  Harvard University} \thanks{The second author was partially supported
  by NSF grant DMS-0841491.}  
\begin{abstract}Under mild hypotheses, we prove that if $F$ is a
  totally real field, and $\rhobar:G_F\to\GL_2(\Flbar)$ is irreducible
  and modular, then there is a finite solvable totally real extension
  $F'/F$ such that $\rhobar|_{G_{F'}}$ has a modular lift which is
  ordinary at each place dividing $l$. We deduce a similar result for
  $\rhobar$ itself, under the assumption that at places $v|l$ the
  representation $\rhobar|_{G_{F_v}}$ is reducible. This allows us to deduce improvements to results
  in the literature on modularity lifting theorems for potentially
  Barsotti-Tate representations and the Buzzard-Diamond-Jarvis
  conjecture. The proof makes use of a novel lifting technique, going
  via rank 4 unitary groups.
\end{abstract}
\maketitle
\tableofcontents
\section{Introduction.}\label{sec:intro}

\subsection{} If $l$ is a prime and $f$ is a
cuspidal newform of weight $2\le k\le l+1$ and level prime to $l$, with
associated mod $l$ Galois representation
$\rhobar_f:G_\Q\to\GL_2(\Fl)$, then there are two possibilities for
the local representation $\rhobar_f|_{G_\Ql}$; either it is reducible
or irreducible. Furthermore, there is a simple criterion
distinguishing these two cases: $\rhobar_f|_{G_\Ql}$ is reducible if
and only if $f$ is ordinary at $l$, in the sense that $a_l$, the
eigenvalue of the Hecke operator $T_l$, is an $l$-adic unit. It is
relatively straightforward to deduce from this result that if
$\rhobar:G_\Q\to\GL_2(\Fl)$ is continuous, irreducible and modular (or
equivalently odd, by Serre's conjecture) and $\rhobar|_{G_\Ql}$ is
reducible, then $\rhobar$ has an ordinary modular lift
$\rho:G_\Q\to\GL_2(\Qlbar)$. This can be generalised to prove the
analogous statement for modular representations of $G_{F^+}$, where ${F^+}$ is
a totally real field in which $l$ splits completely (see Lemma 2.14 of
\cite{MR2392362} together with Lemma \ref{lem:existence of local pot BT lift}
of this paper).

If $l$ does not split completely in ${F^+}$ the situation is rather more
complicated. It is, for example, quite possible for the reduction mod
$l$ of a non-ordinary Galois representation to be ordinary; one easy
way to see this is that by solvable base change one may even find
non-ordinary representations $\rho:G_{F^+}\to\GL_2(\Qlbar)$ such that
$\rhobar|_{G_{{F^+,v}}}$ is trivial for each $v|l$, simply by taking any
non-ordinary modular representation and making a sufficiently large
base change. Moreover, as far as we know nothing is known about the
existence of global ordinary lifts other than Lemma 2.14 of
\cite{MR2392362}, even in the case that $l$ is unramified in ${F^+}$
(despite the claims made in section 3 of \cite{MR2280776}). There are,
however, several reasons to want to prove such results in greater
generality: in particular, the very general modularity lifting
theorems for two-dimensional potentially Barsotti-Tate representations
proved in \cite{kis04} and \cite{MR2280776} are subject to a
hypothesis on the existence of ordinary lifts which limits their
applicability.

In the present paper we resolve this situation by proving that
ordinary lifts exist in considerable generality; all that we require
is that $l\ge 5$, and a mild hypothesis on the image
$\rhobar(G_{F^+})$. The techniques that we employ are somewhat novel, and
are inspired by our work on the Sato-Tate conjecture (\cite{BLGG}); in
particular, they make use of automorphy lifting theorems for rank 4
unitary groups.

As in our earlier work on related questions (e.g. \cite{gee061} and
\cite{GG}) we construct the lifts that we seek by considering a
universal deformation ring, and proving that it has points in
characteristic zero. This gives a candidate Galois representation
lifting $\rhobar$, and one can then hope to prove that this
representation is modular by applying modularity lifting
theorems. In fact, the modularity of the representation usually
follows from the proof that the universal deformation ring has points
in characteristic zero (although not always; there is also a method
due to Ramakrishna (\cite{MR1935843}) of a purely Galois-cohomological
nature which gives no insight into modularity).

To fix ideas, let $R_{F^+}^{univ,ord}$ be the universal deformation ring
for lifts of $\rhobar$ which are crystalline and ordinary of fixed
Hodge-Tate weights 0 and 1, and are unramified outside of a fixed set of
places. It is enough to prove that this ring has a $\Qlbar$-point
which corresponds to a modular representation. By standard Galois
cohomological techniques, one can prove that $\dim R_{F^+}^{univ,ord}\ge
1$, so that in order to prove that it has a $\Qlbar$-point, it is
enough to prove that it is finite over $\Z_l$. In previous work,
following an idea of Khare-Wintenberger and Kisin, we have established
such results by comparing universal deformation rings $R_{F^+}^{univ,ord}$
to related rings for finite Galois extensions $F_1^+/F^+$. We now take
a short detour to explain this, and to explain why this technique is
insufficient for the problem at hand.

Consider, for example, the situation examined in
\cite{gee061}. Continue to assume that
$\rhobar:G_{F^+}\to\GL_2(\Flbar)$ is modular and irreducible, and let
$R_{F^+}^{univ,\tau}$ be the universal deformation ring for lifts of
$\rhobar$ which are potentially Barsotti-Tate and ordinary, of
prescribed Galois type $\tau_v$ at each place $v$ dividing $l$, and which are
unramified outside a fixed finite set of places. Assume that at each
place $v|l$ the local representation $\rhobar_{G_{F^+_v}}$ has a
potentially Barsotti-Tate lift of type $\tau_v$, so that there is no
trivial obstruction to the existence of a global lift. Then one may
again prove that $\dim R^{univ,\tau}_{F^+}\ge 1$, and one again wishes
to show that this ring is finite over $\Z_l$. In order to do so, one
considers a finite solvable extension $F^+_1/F^+$ of totally real
fields, with the property that each of the types $\tau_v$ become
trivial when restricted to the Weil groups of the places above $l$ of
$F^+_1$, and each of the representations $\rhobar|_{G_{F_v}}$ also
become trivial. The extension $F^+_1/F^+$ must also satisfy a number of other
properties which are unimportant to us here. Then one considers the
universal deformation ring $R_{F^+_1}^{univ,BT}$ for lifts of
$\rhobar|_{G_{F^+_1}}$ which are Barsotti-Tate at all places $v|l$,
and unramified outside of a finite set of places, containing all
places lying over the corresponding set for $F^+$. 

Restricting the universal deformation
$G_{F^+}\to\GL_2(R_{F^+}^{univ,\tau})$ to $G_{F^+_1}$ gives (by the
universal properties of universal deformations) $R^{univ,\tau}_{F^+}$
the structure of an $R^{univ,BT}_{F^+_1}$-algebra. It is
straightforward to check (using the fact that the universal
deformation rings are generated by the traces of the images of the
elements of the Galois groups) that in fact $R_{F^+}^{univ,\tau}$ is a finite
$R_{F^+_1}^{univ,BT}$-algebra. Thus to prove that
$R_{F^+}^{univ,\tau}$ is a finite $\Zl$-algebra, it is sufficient to
prove that $R_{F^+_1}^{univ,BT}$ is a finite $\Zl$-algebra.

At this point, we may apply the main theorems of \cite{kis04} and
\cite{MR2280776}, which assert, under certain hypotheses, that
$(R_{F^+_1}^{univ,BT})^{\red}$ is isomorphic to a Hecke algebra. This is
enough to prove that $R_{F^+_1}^{univ,BT}$ is a finite $\Zl$-algebra,
and we are done. The only difficulty is that the most stringent
hypothesis of the main theorems of \cite{kis04} and \cite{MR2280776}
(a hypothesis that we remove in this paper) is that, in order to prove
that $(R_{F^+_1}^{univ,BT})^{\red}$ is isomorphic to a Hecke algebra, one
needs to know that $\rhobar|_{G_{F^+_1}}$ has a modular lift which is
ordinary at all places lying over $l$. This is because the
Taylor-Wiles-Kisin method reduces modularity to the question of
whether one can produce modular Galois representations which give rise
to points on specified components of local (framed) deformation
rings. It is proved in \cite{kis04} and
\cite{MR2280776} that the local Barsotti-Tate deformation rings at hand have
precisely two components, one corresponding to ordinary Galois
representations, and one to non-ordinary representations. One can
always demonstrate the existence of non-ordinary modular lifts, at
least after a solvable base change, by a simple trick using cuspidal
types (cf. Corollary 3.1.6 of \cite{kis04}); but no such argument is
known for ordinary lifts.

In practice, one can sometimes get around this difficulty by only
considering types $\tau_v$ such that any potentially Barsotti-Tate
lift of $\rhobar|_{G_{F_v}}$ of type $\tau_v$ is automatically
non-ordinary, and replacing $R_{F^+_1}^{univ,BT}$ with the universal
deformation ring for non-ordinary Barsotti-Tate deformations. For
example, this approach was taken in \cite{gee08serrewts} in order to
prove unconditional results about the Buzzard-Diamond-Jarvis
conjecture for non-ordinary weights (and again, the
present paper allows us to extend these results to the
case of ordinary weights).

One might hope that an argument along these lines would enable one to
demonstrate the existence of ordinary modular lifts; there is no
problem in proving that the corresponding universal deformation ring
(for ordinary lifts of some fixed regular Hodge-Tate weights) is
positive-dimensional, so one just needs to prove the finiteness of
some universal deformation ring corresponding to the same deformation
problem after a base change, which one could do by proving a
modularity lifting theorem. This, however, appears to be difficult,
because the restriction of an ordinary lifting of $\rhobar|_{G_{F_v}}$
to an open subgroup of $G_{F_v}$ is still ordinary, and components of
local crystalline deformation rings contain either only ordinary
points or only non-ordinary points, so no advantage is gained by
making a base change.

Instead, we proceed by a method related to our work on the Sato-Tate
conjecture (\cite{BLGG}). Rather than employ base change, we instead
use a different functoriality, that of automorphic induction from
$\GL_2$ to $\GL_4$ (in fact, we also make various uses of base change,
but it is the use of automorphic induction, and the induction of
Galois representations, that is the key new ingredient in our
arguments). 

We now sketch the main argument. From the discussion above, we see
that it is enough to prove that there is a finite solvable extension
$F^+_1/F^+$ of totally real fields such that $\rhobar|_{G_{F^+_1}}$
has an ordinary modular Barsotti-Tate lift. In fact, we prove this
theorem without any assumptions on the restrictions of $\rhobar$ to
decomposition groups at places above $l$, and then deduce the main
result using the methods of \cite{gee061}, which we explained
above. We begin by choosing $F^+_1/F^+$ a finite solvable extension of
totally real fields such that $\rhobar|_{G_{F^+_{1,v}}}$ is trivial
for each place $v|l$ of $F^+_1$, and such that $\rhobar$ has a modular
Barsotti-Tate lift $\rho^{BT}:G_{F^+_1}\to\GL_2(\Qlbar)$ which is unramified at all places not dividing $l$,
and which is non-ordinary at all places dividing $l$. This is a
straightforward exercise; it uses the analogue of the statement that
all cuspidal newforms are congruent to forms of weight $2$, together
with the Skinner-Wiles trick for level lowering, and the trick using
cuspidal types that we mentioned above in order to obtain
non-ordinarity.

We now choose a CM quadratic extension $M/F^+_1$ in which every place
dividing $l$ splits, and an algebraic character
$\theta:G_M\to\Qlbar^\times$ such that $\Ind_{G_M}^{G_{F^+_1}}\theta$
is Barsotti-Tate and non-ordinary at
each place dividing $l$. We also construct an algebraic character
$\theta':G_M\to\Qlbar^\times$ such that
$\Ind_{G_M}^{G_{F^+_1}}\theta'$ is crystalline and ordinary with Hodge-Tate weights $0$ and $3$ at
each place dividing $l$, with $\theta$ and $\theta'$ congruent mod
$l$. We now consider 2 deformation problems: firstly, we let
$R_2^{univ,ord}$ denote the universal deformation ring for deformations of
$\rhobar$ which are unramified at all places not dividing $l$ and
which are crystalline and ordinary at all places dividing $l$ with
Hodge-Tate weights $0$ and $3$. Secondly, we let $R_4^{univ}$ denote
the universal deformation ring for certain self-dual lifts of the
$4$-dimensional representation
$\rhobar\otimes\Ind_{G_M}^{G_{F_1^+}}\thetabar$, which are unramified
at all places not dividing $l$, and which are crystalline at all
places dividing $l$ with Hodge-Tate weights $0$, $1$, $3$ and $4$, and
which furthermore satisfy an additional condition: for each place $v|l$, the global
deformations give rise to local liftings on the same components of
the appropriate local crystalline lifting ring as the
representation \[((\Ind_{G_M}^{G_{F^+_1}}\theta)\otimes(\Ind_{G_M}^{G_{F^+_1}}\theta'))|_{G_{F^+_{1,v}}}.\] 

Now, note that (because the local crystalline ordinary lifting rings
are connected) a putative $\Qlbar$-point of $R_2^{univ,ord}$ would
correspond to a representation $\rho:G_{F^+_1}\to\GL_2(\Qlbar)$ such
that the representation $\rho\otimes(\Ind_{G_M}^{G_{F^+_1}}\theta)$
gives a $\Qpbar$-point of $R_4^{univ}$. Thus we obtain (from the
universal properties defining $R_4^{univ}$ and $R_2^{univ,ord}$) a
natural map $R^{univ}_4\to R_2^{univ,ord}$. We show that this map is
finite, again making use of the fact that universal deformation rings
are generated by traces.

Furthermore, we show that $R^{univ}_4$ is a finite $\Zl$-algebra, by
identifying its reduced quotient with a Hecke algebra, using the automorphy lifting
theorem proved in \cite{BLGG}, and the automorphy of
$\rho^{BT}\otimes\Ind_{G_M}^{G_{F^+_1}}\theta$ (which follows from
standard properties of automorphic induction). Thus $R_2^{univ,ord}$
is a finite $\Zl$-algebra. Standard Galois cohomology calculations
show that it has dimension at least 1, so it has $\Zl$-rank at least
one, and thus has $\Qlbar$-points. Let
$\rho^{ord}:G_{F^+_1}\to\GL_2(\Qlbar)$ be the lift of
$\rhobar|_{G_{F^+_1}}$ corresponding to such a point. By construction,
this is ordinary at each place $v|l$, and since $(R^{univ}_4)^{red}$
has been identified with a Hecke algebra, we see that
$\rho^{ord}\otimes\Ind_{G_M}^{G_{F^+_1}}\theta$ is automorphic. The
modularity of $\rho^{ord}$ then follows from an idea of Harris
(\cite{harristrick}, although the version we use is based on that
employed in \cite{BLGHT}). Finally, a basic argument with Hida
families allows us to replace $\rho^{ord}$ with a modular ordinary
Barsotti-Tate representation.

In fact, we have glossed over several technical difficulties in the
above argument. It is more convenient to work over a quadratic CM
extension of $F^+_1$ for much of the argument, only descending to $F^+_1$
at the end. It is also convenient to make several further solvable
base changes during the argument; for example, we prefer to work in
situations where all representations considered are unramified outside
of $l$, in order to avoid complications due to non-smooth points on
local lifting rings. We also need to ensure that the various
hypotheses of the automorphy lifting theorems used are satisfied; this
requires us to assume that $l>4$ for most of the paper (because we use
automorphy lifting theorems for $\GL_4$), and to assume
various ``big image'' hypotheses. 

We now explain the structure of the paper. In section
\ref{sec:character} we construct the characters $\theta$ and
$\theta'$. The arguments of this section are very similar to the
corresponding arguments in \cite{BLGHT} and \cite{BLGG}. Section
\ref{sec:lifting} contains our main results on global deformation
rings, including the finiteness results discussed above, and a
discussion of oddness. We take care to work in considerable
generality, with a view to further applications; in particular, in
future work we will apply the machinery of this section to
generalisations of the weight part of Serre's conjecture. In section
\ref{sec:big image GL2} we examine the condition of a subgroup of
$\GL_2(\Flbar)$ being 2-big in some detail; this allows us to obtain
concrete conditions under which our main theorems hold. Section
\ref{sec:untwisting CM} contains a slight variant on the trick of
Harris mentioned above, and is very similar to the corresponding
section of \cite{BLGG}. The main theorems are obtained in section
\ref{sec:main thm}. In addition to the argument explained above, we
give a more elementary argument in the case of residually dihedral
representations, which allows us to handle some cases with $l=3$.

We would like to thank Brian Conrad for his assistance with the proof
of Lemma \ref{lem:existence of local pot BT lift}. This paper was
written simultaneously with \cite{BLGGT}, and we would like to thank
Richard Taylor for a number of helpful discussions.

\subsection{Notation}
If $M$ is a field, we let $G_M$ denote its absolute Galois group.  We
write all matrix transposes on the left; so ${}^tA$ is the transpose
of $A$. Let $\epsilon$ denote the $l$-adic cyclotomic character, and
$\bar{\epsilon}$ or $\omega$ the mod $l$ cyclotomic character. If $M$ is
a finite extension of $\bb{Q}_p$ for some $p$, we write $I_M$ for the
inertia subgroup of $G_M$. If $R$ is a local ring we write
$\mf{m}_{R}$ for the maximal ideal of $R$.

We fix an algebraic closure $\Qbar$ of $\Q$, and regard all algebraic
extensions of $\Q$ as subfields of $\Qbar$. For each prime $p$ we fix
an algebraic closure $\Qpbar$ of $\Qp$, and we fix an embedding
$\Qbar\into\Qpbar$. In this way, if $v$ is a finite place of a number
field $F$, we have a homomorphism $G_{F_v}\into G_F$.

We normalise the definition of Hodge-Tate weights so that all the
Hodge-Tate weights of the $l$-adic cyclotomic character are $-1$. We
refer to a crystalline representation with all Hodge-Tate weights
equal to $0$ or $1$ as a Barsotti-Tate representation (this is
somewhat non-standard terminology, but it will be convenient).

We will use some of the notation and definitions of \cite{cht} without
comment. In particular, we will use the notions of RACSDC and RAESDC
automorphic representations, for which see sections 4.2 and 4.3 of
\cite{cht}. We will also use the notion of a RAECSDC automorphic
representation, for which see section 1 of \cite{BLGHT}. If $\pi$ is a RAESDC automorphic representation of $\GL_n(\A_F)$, $F$
a totally real field, and $\iota:\Qlbar\isoto\C$, then we let
$r_{l,\iota}(\pi):G_F\to\GL_n(\Qlbar)$ denote the corresponding Galois
representation. Similarly, if $\pi$ is a RAECSDC or RACSDC automorphic representation of $\GL_n(\A_F)$, $F$
a CM field (in this paper, all CM fields are totally imaginary), and $\iota:\Qlbar\isoto\C$, then we let
$r_{l,\iota}(\pi):G_F\to\GL_n(\Qlbar)$ denote the corresponding Galois
representation. For the properties of $r_{l,\iota}(\pi)$, see Theorems
1.1 and 1.2 of \cite{BLGHT}. 

\section{Another character-building exercise.}\label{sec:character}
\subsection{} In this section, we will complete a trivial but rather technical exercise
in class field theory, which will allow us to construct characters with prescribed 
properties which will be useful to us throughout our arguments. Similar calculations
appeared in \cite{BLGHT} and \cite{BLGG}; we apologize that our earlier efforts
were not carried out in sufficient generality for us to be able to to
simply cite them.

We first recall the following definition.

\begin{defn}
\label{defn: m-big}
  Let $k/\F_l$ be algebraic and let $m$ and $n$ be positive integers. We say that a
  subgroup $H$ of $\GL_n(k)$ is $m$-\emph{big} if the
  following conditions are satisfied.
  \begin{itemize}
  \item $H$ has no $l$-power order quotient.
  \item $H^0(H,\mathfrak{sl}_n(k))=(0)$.
  \item $H^1(H,\mathfrak{sl}_n(k))=(0)$.
  \item For all irreducible $k[H]$-submodules $W$ of
    $\mathfrak{gl}_n(k)$ we can find $h\in H$ and $\alpha\in k$ such
    that:
    \begin{itemize}
    \item $\alpha$ is a simple root of the characteristic polynomial
      of $h$, and if $\beta$ is any other root then
      $\alpha^m\ne\beta^m$.
    \item Let $\pi_{h,\alpha}$ (respectively $i_{h,\alpha}$) denote
      the $h$-equivariant projection from $k^n$ to the
      $\alpha$-eigenspace of $h$ (respectively the $h$-equivariant
      injection from the $\alpha$-eigenspace of $h$ to $k^n$). Then
      $\pi_{h,\alpha}\circ W\circ i_{h,\alpha}\ne 0$.
    \end{itemize}
  \end{itemize}
Also, we say that $H$ is \emph{big} if it is 1-big.
\end{defn}

We now turn to the main result of this section.

\begin{lem}\label{lem:char building}
Suppose that:
\begin{itemize}
\item $F^+$ is a totally real field,
\item $n$ is a positive integer,
\item  $l$ is a rational prime, with $l>2$ and
  $l>2n-2$,
\item $m$ is a positive even integer, with $l\notdiv m$,
\item $\Favoid$ is a finite extension of $F^+$,
\item $T$ is a finite set of primes of $F^+$, not containing places above $l$,
\item $\eta:G_{F^+} \to \Zbar_l^\times$ is a finite order character, such that $\eta$ takes some fixed
value $\gamma\in\{\pm1\}$
on every complex conjugation, and is unramified at each prime of $T$, and
\item $\eta':G_{F^+} \to \Zbar_l^\times$ is another finite order
  character, congruent to $\eta$ (mod $l$), and is unramified at each prime of $T$.
(Note that $\eta'$ therefore also takes the
value $\gamma$ on every complex conjugation.)
\end{itemize}

Suppose further that for each embedding $\tau$ of $F^+$ into $\Qbar_l$, and each 
integer $i, 1\leq i\leq m$, we are given an integer $h_{i,\tau}$ and another integer 
$h'_{i,\tau}$, Finally, suppose that there are integers $w, w'$ such that for each $i$ and $\tau$, 
the integers $h_{i,\tau}$, $h'_{i,\tau}$ satisfy:
$$h_{i,\tau} + h_{m+1-i,\tau}=w\quad\text{and}\quad h'_{i,\tau} + h'_{m+1-i,\tau}=w'.$$

Then we can find a cyclic, degree $m$ CM extension $M$ of $F^+$ which
is linearly disjoint from $\Favoid$ over $F^+$, and continuous
characters
$$\theta,\theta':G_M\to\Zbar_l^\times$$
such that $\theta$ and
$\theta'$ are de Rham at all primes above $l$, and furthermore:
\begin{enumerate}
\item $\theta,\theta'$ are congruent (mod $l$).
\item For any $\rbar:G_{F^+}\to\GL_n(\Flbar)$, a continuous Galois representation ramified only
at primes in $T$ and above $l$, which satisfies $\overline{F^+}^{\ker \bar{r}}(\zeta_l)\subset\Favoid$:
    \begin{itemize}
    \item  \label{lem:char building - big image}
    If $\rbar(G_{F^+(\zeta_l)})$ is $m$-big, then
    $(\rbar\tensor\Ind_{G_M}^{G_{F^+}} \overline{\theta})(G_{F^+(\zeta_l)})=(\rbar\tensor\Ind_{G_M}^{G_{F^+}}
    \overline{\theta}')(G_{F^+(\zeta_l)})$ is big.
    \item If $[\overline{F^+}^{\ker\ad \bar{r}}(\zeta_l):\overline{F^+}^{\ker\ad \bar{r}}]>m$ then the 
    fixed field of the kernel of
    $\ad(\bar{r}\tensor\Ind_{G_M}^{G_{F^+}}
    \bar{\theta})=\ad(\bar{r}\tensor\Ind_{G_M}^{G_{F^+}}\bar{\theta}')$
    will not contain
    $\zeta_l$.
    \end{itemize}
\item    \label{lem:char building - ind theta pairing}
      We can put a perfect pairing on $\Ind_{G_M}^{G_{F^+}} \theta$ satisfying
     \begin{enumerate}
     \item $\langle v_1,v_2\rangle=(-1)^w\gamma\langle v_2,v_1\rangle$.
     \item For $\sigma \in G_{F^+}$, we have 
        $$\langle \sigma v_1, \sigma v_2\rangle 
             = \eps(\sigma)^{-w}\eta(\sigma) \langle  v_1, v_2\rangle.$$
     \end{enumerate}
     Thus, in particular, 
     $$(\Ind_{G_M}^{G_{F^+}} \theta)\cong (\Ind_{G_M}^{G_{F^+}} \theta)^\vee
     \tensor \eps^{-w}\eta.$$
     (Note that
     the character on the right hand side takes the value $(-1)^w\gamma$ on
     complex conjugations.)  
\item   \label{lem:char building - ind theta prime pairing}
 We can put a perfect pairing on $\Ind_{G_M}^{G_{F^+}} \theta'$ satisfying
     \begin{enumerate}
     \item $\langle v_1,v_2\rangle=(-1)^{w}\gamma \langle v_2,v_1\rangle$.
     \item For $\sigma \in G_{F^+}$, we have 
        $$\langle \sigma v_1, \sigma v_2\rangle 
             = \eps(\sigma)^{-w'}
             \wtilde(\sigma)^{w'-w} \eta'(\sigma)\langle  v_1, v_2\rangle$$
             where $\wtilde$ is the Teichm\"{u}ller lift of the mod
             $l$ cyclotomic character.
     \end{enumerate}
     Thus, in particular, 
     $$(\Ind_{G_M}^{G_{F^+}} \theta')\cong (\Ind_{G_M}^{G_{F^+}}
     \theta')^\vee\tensor \eps^{-w'}\wtilde^{w'-w}\eta'.$$ 
     (Note that
     the character on the right hand side takes the value $(-1)^{w}\gamma$ on
     complex conjugations.)       
\item For each $v$ above $l$, the representation $(\Ind_{G_M}^{G_{F^+}} \theta)|_{G_{F^+_v}}$
is conjugate to a representation which breaks up as a direct sum of characters:

$$(\Ind_{G_M}^{G_{F^+}} \theta)|_{G_{F^+_v}} \cong 
               \chi^{(v)}_1\oplus\chi^{(v)}_2\oplus\dots\oplus \chi^{(v)}_{m}$$
where, for each $i$, $1\leq i\leq m$, and each embedding $\tau:F^+_v\to\Qbar_l$ we have that:
$$\HT_\tau(\chi^{(v)}_i) = h_{i,\tau}.$$ Similarly, the representation $(\Ind_{G_M}^{G_{F^+}} \theta')|_{G_{F^+_v}}$
is conjugate to a representation which breaks up as a direct sum of characters:
$$(\Ind_{G_M}^{G_{F^+}} \theta')|_{G_{F^+_v}} \cong 
               \chi'^{(v)}_1\oplus\chi'^{(v)}_2\oplus\dots\oplus \chi'^{(v)}_{m}$$
where, for each $i$, $1\leq i\leq m$, and each embedding $\tau:F^+_v\to\Qbar_l$ we have that:
$$\HT_\tau(\chi'^{(v)}_i) = h'_{i,\tau}.$$

\item $M/F^+$ is unramified at each prime of $T$; and $\theta, \theta'$ are unramified above each
prime of $T$. Thus $\Ind_{G_M}^{G_{F^+}} \theta$ and $\Ind_{G_M}^{G_{F^+}} \theta'$ are
unramified at each prime of $T$. Each prime of $F^+$ above $l$ splits completely in $M$.
\item If $\eta$ is unramified at each place dividing $l$, then
  $\theta$ is crystalline.
\end{enumerate}
\end{lem}
\begin{proof} Throughout this proof, we will use $\bar{F}$ as a shorthand for $\overline{F^+}$.

 {\sl Step 1: Finding a suitable field $M$.} We claim that there
 exists a surjective character $\chi: \Gal(\bar{F}/F^+)\to
 \mu_{m}$ (where $\mu_{m}$ is the group of $m$-th
 roots of unity in $\Qbar^{\times}$) such that
 \begin{itemize}
 \item $\chi$ is unramified with $\chi(\Frob_v)=1$ at all places $v$ of $F^+$ above $l$.
 \item $\chi(c_v)=-1$ for each infinite place $v$ (where $c_v$ denotes
   a complex conjugation at $v$).
 \item $\bar{F}^{\ker \chi}$ is linearly disjoint from $\Favoid$ over $F^+$.
 \item $\chi$ is unramified at all primes of $T$.
 \end{itemize}
 We construct the character $\chi$ as follows. First, we find using
 weak approximation a totally negative element $\alpha\in (F^+)^\times$
 which is a $v$-adic unit and a quadratic residue mod $v$ for each
 prime $v$ of $F^+$ above $l$, and which is a $w$-adic unit for each 
 prime $w$ in $T$. Let $\chi_0$ be the quadratic 
 character associated to the extension we get by adjoining the square
 root of this element. Then:
\begin{itemize}
 \item $\chi_0$ is unramified with $\chi_0(\Frob_v)=1$ at all places $v$ of $F^+$ above $l$.
 \item $\chi_0(c_v)=-1$ for each infinite place $v$ (where $c_v$ denotes
   a complex conjugation at $v$).
 \item $\chi_0$ is unramified at all places in $T$.
\end{itemize}
Now choose a cyclic totally real extension $M_1/\Q$ of degree $m$
such that:
\begin{itemize}
\item $M_1/\Q$ is unramified at all the rational primes where
  $\bar{F}^{\ker \chi_0}\Favoid/\Q$ is ramified, and at all rational
  primes which lie below a prime of $T$.
\item $l$ splits completely in $M_1$.
\end{itemize}
Since $\bar{F}^{\ker \chi_0}\Favoid/\Q$ and $M_1/\Q$ ramify at disjoint sets of
primes, they are linearly disjoint, and we can find a rational prime
$p$ which splits completely in $\Favoid\bar{F}^{\ker
  \chi_0}$ but such that $\Frob_p$ generates $\Gal(M_1/\Q)$. Since
$M_1/\Q$ is cyclic, we may pick an isomorphism between $\Gal(M_1/\Q)$
and $\mu_{m}$, and we can think of $M_1$ as determining a character
$\chi_1:G_\Q \to \mu_{m}$ such that:
\begin{itemize}
\item $\chi_1$ is trivial on $G_{\Ql}$.
\item $\chi_1$ is trivial on complex conjugation.
\item $\chi_1(\Frob_p) = \zeta_{m}$, a primitive $m$th root of
  unity.
\item $\chi_1$ is unramified at all rational primes lying below primes of $T$.
\end{itemize}
Then, set $\chi=(\chi_1|_{G_{F^+}})\chi_0$. Note that this maps onto
$\mu_{m}$, even when we restrict to $G_{\Favoid}$ (since
$p$ splits completely in $\Favoid$ and if $\wp$ is a place
of $\Favoid$ over $p$, we have $\chi_0(\Frob_\wp)=1$ while
$\chi_1(\Frob_\wp)=\zeta_{m}$).
The remaining properties are clear.

 Having shown $\chi$ exists, we set $M=\bar{F}^{\ker \chi}$; note
 that this is a CM field, and a cyclic extension of $F^+$ of degree $m$. Write $\sigma_{M/F^+}$
 for a generator of $\Gal(M/F^+)$. Write $M^+$ for the maximal totally
 real subfield of $M$. Note also our choices have ensured that almost 
 all the points in conclusion 6 of the Lemma currently being proved will hold; all that
 remains from there to be checked is that $\theta$ and $\theta'$ are unramified
 at all primes above prime of $T$.

 \medskip{\sl Step 2: An auxiliary prime $q$}. Choose a rational prime
 $q$ such that
\begin{itemize}
\item no prime of $T$ lies above $q$,
\item $q \ne l$,
\item $q$ splits completely in $M$, 
\item $q$ is unramified in $\Favoid$, 
\item $q-1>2n$, and
\item $\eta$ and $\eta'$ are both unramified at all primes above $q$.
\end{itemize}
Also choose a prime $\fq$ of $F^+$ above $q$, and a prime $\fQ$ of M
above $\fq$.

\medskip{\sl Step 3: Defining certain algebraic characters $\phi$,
  $\phi'$}.  For each prime $v$ of $F^+$, choose a prime $\tilde{v}$
of $M$ lying above it. Then for each embedding $\tau$ of $F^+$ into
$\Qbar_l$, select an embedding $\tilde{\tau}$ of $M$ into $\Qbar_l$
extending it, in such a way that $\tilde{\tau}$ corresponds to the
prime $\tilde{v}$ if $\tau$ corresponds to $v$.
    
Note we now have a convenient notation for all the embeddings
extending $\tau$; in particular, they can be written as
$\tilde{\tau}\circ\sigma_{M/F^+}^j$ for $j=0,\dots, m-1$.

We are now forced into a slight notational ugliness. Write $\tilde{M}_0$
for the Galois closure of $M$ over $\Q$,
and $\tilde{M}$ for $\tilde{M}_0$ with the $\#\eta(G_{F^+})$th roots
of unity adjoined, so $\tilde{M}=\tilde{M}_0(\mu_{\#\eta(G_{F^+})})$.  
(Thus $\Gal(\tilde{M}/\Q)$
is in bijection with embeddings $\tilde{M}\to \Qbar$.) Let us fix
$\iota^*$, an embedding of $\tilde{M}$ into $\Qbar_l$, and write $v^*$ for the prime of $M$ below this.\footnote{The choice
  of this $\iota^*$ will affect the choice of the algebraic characters
  $\phi,\phi'$ below, but will be cancelled out---at least concerning
  the properties we care about---when we pass to the $l$-adic
  characters $\theta,\theta'$ below.}  Using $\iota^*$, we can and will abuse notation by 
  thinking of $\eta$ as being valued in $\tilde{M}$.
  Given any embedding $\iota'$ of $M$
into $\Qbar_l$, we can choose an element $\sigma_{\iota^*\leadsto\iota'}$ in
$\Gal(\tilde{M}/\Q)$ such that $\iota'= (\iota^*\circ\sigma_{\iota^*\leadsto\iota'})|_M$.

 We claim that there exists an extension $M'$ of $\tilde{M}$, and a character
 $\phi:\A_{M}^\times\to (M')^\times$ with open kernel such that:
 \begin{itemize}
 \item For $\alpha\in M^\times$,
   \[
    \phi(\alpha) = \prod_{\tau\in \Hom(F^+,\Qbar_l)} 
    \prod^{(m/2)-1}_{j=0} 
         (\sigma_{\iota^*\leadsto\tilde{\tau}\circ\sigma^{-j}_{M/F^+}}(\alpha))^{h_{j+1,\tau}}
         (\sigma_{\iota^*\leadsto\tilde{\tau}\circ\sigma_{M/F^+}^{-j-(m/2)}}(\alpha))^{h_{m-j,\tau}}.
   \]
\item For $\alpha\in(\A_{M^+})^\times$, we have 
  $$\phi(\alpha)=(\prod_{v\nmid\infty}
  |\alpha_v|\prod_{v|\infty}\sgn_v(\alpha_v))^{-w}\delta_{M/M^+}(
  \Art_{M^+}(\alpha))^{-w+(\gamma-1)/2}\eta|_{G_{M^+}}(\Art_{M^+}(\alpha)),$$
where $\delta_{M/M^+}$ is the quadratic character of $G_{M^+}$
  associated to $M$.
 (Note that, in the right hand side, we
  really think of $\alpha$ as an element of $\A_{M^+}$, not
  just as an element of $\A_{M}$ which happens to lie in
  $\A_{M^+}$; so for instance $v$ runs over places of $M^+$,
  and the local norms are appropriately normalized to reflect us
  thinking of them as places of $M^+$.)
\item If $\eta$ is unramified at $l$, then $\phi$ is unramified at $l$.
\item $\phi$ is unramified at all primes above primes of $T$.
\item $q|\#\phi(\mathcal{O}^\times_{M,\fQ})$, but $\phi$ is
  unramified at primes above $\fq$ other than $\fQ$ and $\fQ^c$
\end{itemize}
This is an immediate consequence of Lemma 2.2 of \cite{hsbt}, as
follows.
We must
verify 
that the formula for $\phi(\alpha)$ for $\alpha\in\A_{M^+}^\times$ in the
second bullet point is trivial on $-1\in M^+_v$ for each infinite place $v$
and that the conditions in the various bullet points are compatible.
The former is immediate, as may be verified by a direct calculation.
The only difficult part in checking that the bullet points are compatible
is comparing the first and second, which we may 
verify as follows. Suppose that $\alpha\in (M^+)^\times$; then we must
check that the expression for $\phi(\alpha)$ from the first bullet point 
($\phi_0(\alpha)$, say), matches that from the second ($\phi_1(\alpha)$, say).
It will suffice to show that $\iota^*(\phi_0(\alpha))=\iota^*(\phi_1(\alpha))$.

We have:
\begin{align*}
\iota^*(\phi_0(\alpha)) &= \prod_{\tau\in \Hom(F^+,\Qbar_l)} 
    \prod^{(m/2)-1}_{j=0} (\tilde{\tau}(\sigma^{-j}_{M/F^+}(\alpha)))^{h_{j+1,\tau}}
         (\tilde{\tau}(\sigma_{M/F^+}^{-j-(m/2)}(\alpha)))^{h_{m-j,\tau}}\\
\intertext{and then using the facts that $\sigma^{m/2}_{F^+/M}$ fixes $M^+$ and $h_{i,\tau} + h_{m+1-i,\tau}=w$, this}
         &= \prod_{\tau\in \Hom(F^+,\Qbar_l)} 
    \prod^{(m/2)-1}_{j=0} (\tilde{\tau}(\sigma^{-j}_{M/F^+}(\alpha)))^{h_{j+1,\tau}}
         (\tilde{\tau}(\sigma^{-j}_{M/F^+}(\alpha)))^{h_{m-j,\tau}} \\
         &= \prod_{\tau\in \Hom(F^+,\Qbar_l)} 
    \prod^{(m/2)-1}_{j=0} (\tilde{\tau}(\sigma^{-j}_{M/F^+}(\alpha)))^{w} \\ &= \prod_{\tau_M\in \Hom(M^+,\Qbar_l)} \tau_M(\alpha)^w\\
 \displaybreak[0]
 &=i_{\Q\to\Qlbar} (N_{M^+/\Q}(\alpha)^w)\\
    &=i_{\Q\to\Qlbar}\left(\prod_{\substack{v\text{ a finite}\\\text{place of $\Q$}}} |N_{M^+/\Q}(\alpha)|_v^{w} \sgn(N_{M^+/\Q}(\alpha))^{w}\right)\\
    &=i_{\Q\to\Qlbar}((\prod_{v\nmid\infty}
  |\alpha_v|\prod_{v|\infty}\sgn_v(\alpha_v))^{-w})
       \quad\quad\text{($v$ ranges over places of $M^+$)}\\
    &=i_{\Q\to\Qlbar}(\prod_{v\nmid\infty}
  |\alpha_v|\prod_{v|\infty}\sgn_v(\alpha_v))^{-w}\delta_{M/M^+}(
  \Art_{M^+}(\alpha))^{-w+(\gamma-1)/2}\eta(\Art_{M^+}(\alpha)),
\end{align*}
(where $i_{\Q\to\Qlbar}$ is the inclusion $\Q\into\Qlbar$) since
$\Art(\alpha)$ is trivial; this is as required.

Similarly, we construct a character $\phi':(\A_{M})^\times\to
(M')^\times$ (enlarging $M'$ if necessary) with open kernel such that:
 \begin{itemize}
 \item For $\alpha\in M^\times$,
   \[
    \phi'(\alpha) = \prod_{\tau\in \Hom(F^+,\Qbar_l)} 
    \prod^{(m/2)-1}_{j=0} 
         (\sigma_{\iota^*\leadsto\tilde{\tau}\circ\sigma^{-j}_{M/F^+}}(\alpha))^{h'_{j+1,\tau}}
         (\sigma_{\iota^*\leadsto\tilde{\tau}\circ\sigma^{-j-(m/2)}}(\alpha))^{h'_{m-j,\tau}}.
   \]
\item For $\alpha\in(\A_{M^+})^\times$, we have 
  $$\phi'(\alpha)=(\prod_{v\nmid\infty}
  |\alpha_v|\prod_{v|\infty}\sgn_v(\alpha_v))^{-w'}\delta_{M/M^+}(
  \Art_{M^+}(\alpha))^{-w'+(\gamma-1)/2}\eta'|_{G_{M^+}}(\Art_{M^+}(\alpha)).$$
  (Again, we think of $\alpha$ in the right
  hand side as a bona fide member of $\A_{M^+}$.)
\item $\phi'$ is unramified at all primes above primes of $T$.
\end{itemize}
Again, this follows from Lemma 2.2 of \cite{hsbt}.

 \medskip{\sl Step 4: Defining the characters
   $\theta,\theta'$}. Write $\tilde{M}'$ for the Galois
 closure of $M'$ over $\Q$, and extend $\iota^*:\tilde{M}\to\Qbar_l$
 to an embedding $\iota^{**}:\tilde{M}'\to\Qbar_l$.  Define $l$-adic
 characters $\theta,\theta_0':\Gal(\overline{M}/M)\to
 \overline{\Z}_l^\times$ by the following expressions (here $\alpha\in \A_M$,
 and if $\tilde{\tau}$ is a map $M\into\Qbar_l$, corresponding to a place $v(\tilde{\tau})$ of $M$,
 then $\tilde{\tau}(\alpha)$ is a shorthand for $\alpha_{v(\tilde{\tau})}$, mapped into $\Qbar_l$ via 
 the unique continuous extension of $\tilde{\tau}$ to a map $M_{v(\tilde{\tau})}\into \Qbar_l$):
 \begin{align*}
   \theta(\Art \alpha) &= \iota^{**}(\phi(\alpha)) \prod_{\tau\in\Hom(F^+,\Qbar_l)}
   \prod^{(m/2)-1}_{j=0} 
   (\tilde{\tau}\sigma^{-j}_{M/F^+})(\alpha)^{-h_{j+1,\tau}} 
   (\tilde{\tau}\sigma^{-j-m/2}_{M/F^+})(\alpha)^{-h_{m-j,\tau}} 
	\\	
   \theta'_0(\Art \alpha) &= \iota^{**}(\phi'(\alpha)) \prod_{\tau\in\Hom(F^+,\Qbar_l)}
   \prod^{(m/2)-1}_{j=0} 
   (\tilde{\tau}\sigma^{-j}_{M/F^+})(\alpha)^{-h'_{j+1,\tau}} 
   (\tilde{\tau}\sigma^{-j-m/2}_{M/F^+})(\alpha)^{-h'_{m-j,\tau}} 
 \end{align*}
where $v$ runs over places of $F$ dividing $l$. (It is easy to check that the expressions on the right hand sides are
 unaffected when $\alpha$ is multiplied by an element of $M^\times$.)
 Observe then that they enjoy the following properties:
 \begin{itemize}
 \item $\theta \circ V_{M/M^+} = \eps^{-w}\delta_{M/M^+}^{-w+(\gamma-1)/2}\eta|_{G_M^+}$ where $V_{M/M^+}$
     is the transfer map $G_{M^+}^{\ab} \rightarrow G_{M}^{\ab}$. In
     particular, $\theta\theta^c=\eps^{-w}\eta|_{G_M}$.
 \item $\theta'_0 \circ V_{M/M^+} = \eps^{-w'}\delta_{M/M^+}^{-w'+(\gamma-1)/2}\eta'|_{G_M^+}$
   and hence, $\theta'_0\theta'_0{}^c=\eps^{-w'} \eta'|_{G_M}$.
\item For $\tau$ an embedding of $F^+$ into $\Qbar_l$ and $0\le j\le (m/2)-1$, the Hodge-Tate
 weight of $\theta$ at the embedding $\tilde{\tau}\sigma^{-j}_{M/F^+}$ is $h_{j+1,\tau}$
 and the Hodge-Tate
 weight of $\theta$ at the embedding $\tilde{\tau}\sigma^{-j-m/2}_{M/F^+}$ is $h_{m-j,\tau}$.
\item For $\tau$ an embedding of $F^+$ into $\Qbar_l$ and $0\le j\le (m/2)-1$, the Hodge-Tate
 weight of $\theta'_0$ at the embedding $\tilde{\tau}\sigma^{-j}_{M/F^+}$ is $h'_{j+1,\tau}$
 and the Hodge-Tate
 weight of $\theta'_0$ at the embedding $\tilde{\tau}\sigma^{-j-m/2}_{M/F^+}$ is $h'_{m-j,\tau}$.
 \item $q|\#\theta(I_\fQ)$, but $\theta$ is unramified at all primes
   above $\fq$ except $\fQ,\fQ^c$.
 \end{itemize}
We now define $\theta'=\theta'_0(\tilde{\theta}/
\tilde{\theta'_0})$---where $\tilde{\theta}$ (resp $\tilde{\theta_0'}$)
denotes the Teichm\"uller lift of the reduction mod $l$ of $\theta$
(resp $\theta_0'$)---and observe that:
 \begin{itemize}
 \item $\theta'$ (mod $l$) = $\theta$ (mod $l$).
 \item $\theta'(\theta')^c=\eps^{-w'}\wtilde^{w'-w}\eta'|_{G_M}$.
\end{itemize}

\medskip{\sl Step 5: Properties of $\Ind_{G_M}^{G_{F^+}}\theta$ and $\Ind_{G_M}^{G_{F^+}}\theta'$.} 
We begin by addressing point 3. We define a pairing on $\Ind_{G_M}^{G_{F^+}}\theta$ 
by the formula
$$\langle \lambda,\lambda'\rangle = \sum_{\sigma\in\Gal(\overline{M}/M)\backslash 
                      \Gal(\overline{M}/F)} 
               \eps_l(\sigma)^{w} \eta(\sigma)^{-1}
                \lambda(\sigma)\lambda'(c\sigma)$$
where $c$ is any complex conjugation. One easily checks that this is
well defined and perfect, and that the properties (a) and (b) hold.

We can address point 4 in a similar manner, defining a pairing on $\Ind_{G_M}^{G_{F^+}}\theta'$
via:
$$\langle \lambda,\lambda'\rangle = \sum_{\sigma\in\Gal(\overline{M}/M)\backslash 
                      \Gal(\overline{M}/F)} 
               \eps_l(\sigma)^{w'} \wtilde(\sigma)^{w-w'} \eta'(\sigma)^{-1}
                \lambda(\sigma)\lambda'(c\sigma)$$
and checking the required properties.

Next, we address point 5. We will give the argument for $\Ind^{G_{F^+}}_{G_M} \theta$;
the argument for $\Ind^{G_{F^+}}_{G_M} \theta'$ is similar.
Since the primes of $F^+$ above $l$ split in $M$,
we immediately have that for each such prime $v$, 
$(\Ind^{G_{F^+}}_{G_M} \theta)|_{G_{F^+_v}}$ will split as required as a direct 
sum of characters $\chi^{(v)}_i$, with each character $\chi^{(v)}_i$ corresponding
to a prime of $M$ above $v$. We recall that we chose above
a prime $\tilde{v}$ of $M$ above $v$ for each such prime $v$;
we without loss of generality assume that the $\chi^{(v)}_i$ are
numbered so that for $0\leq i \leq (m/2)-1$, we have that
 $\chi^{(v)}_{i+1}$ corresponds to the prime $\sigma_{M/F^+}^{i}\tilde{v}$
and $\chi^{(v)}_{m-i}$ corresponds to the prime $\sigma_{M/F^+}^{i+m/2}\tilde{v}$.

Then for $\tau$ an embedding of $F^+_v\to \Qbar_l$, we have that
\begin{align*}
\HT_\tau(\chi^{(v)}_{i+1})&=\HT_{\tilde{\tau}\sigma_{M/F^+}^{-i}}(\theta)=h_{i+1,\tau}\\
\HT_\tau(\chi^{(v)}_{m-i})&=\HT_{\tilde{\tau}\sigma_{M/F^+}^{-i-m/2}}(\theta)=h_{m-i,\tau}
\end{align*}

 \medskip{\sl Step 6: Establishing the big image/avoid $\zeta_l$ properties}. All that
 remains is to prove the big image and avoiding $\zeta_l$ properties; that is, point (2). 
 We will just show the stated properties concerning
$\Ind_{G_M}^{G_{F^+}}\theta$; the statement for $\Ind_{G_M}^{G_{F^+}}\theta'$
then follows since $\theta$ and $\theta'$ are congruent.

Let $\rbar : G_{F^+}\to \GL_n(\Flbar)$ be a
continuous Galois representation such that
the following properties hold:
\begin{itemize}
\item $\rbar$ is ramified only at primes of $T$ and above $l$, and 
\item we have 
$\bar{F}^{\ker \bar{r}}(\zeta_l)\subset \Favoid$.
\end{itemize} We 
may now apply Lemma 4.1.2 of \cite{BLGG} (with $F$ in that lemma equal
to our current $F^+$). (To be completely precise, Lemma 4.1.2 as written
only applies to a characteristic 0 representation $r$ rather than the
characteristic $l$ representation $\rbar$ as we have here. But the proof 
goes through exactly the same if one starts with a characteristic $l$ representation.) 
 
If we assume that $\rbar(G_{F^+(\zeta_l)})$ is $m$-big, then applying part 2 of that lemma will give that 
$(\rbar \otimes\Ind_{G_M}^{G_{F^+}}\thetabar)|_{G_{F^+(\zeta_l)}}$ has big image,
(the first part of point (2) to be proved). Similarly, applying part 1 will give the fact
that we avoid $\zeta_l$ (the second part of point (2)).
 All that remains is to check the hypotheses
of Lemma 4.1.2 of \cite{BLGG}.

The fact that $M$ is linearly disjoint from $\overline{F}^{\ker \rbar}(\zeta_l)$
(common to both parts) comes from the fact that 
$\overline{F}^{\ker \rbar}(\zeta_l)\subset \Favoid$ and $M$ was
chosen to be linearly disjoint from $\Favoid$. The fact that every
place of $F^+$ above $l$ is unramified in $M$ follows from the
construction of $M$ (in fact, they all split completely).

We turn now to the particular hypotheses of the second part. That $\rbar|_{G_{F(\zeta_l)}}$ has $m$-big
image is by assumption.
The properties we require of $\fq$ follow directly from the bullet
points established in Step 2, the properties of $\rbar$ just above,
and the first and last bullet points 
(concerning $\theta\theta^c$ and $\#\theta(I_\fQ)$ respectively)
in the list of properties of $\theta$
given immediately after $\theta$ is introduced in step 4.
The fact that $\theta\theta^c$ can be extended to $G_{F^+}$ comes from the first bullet
point in the list of properties of $\theta$ in step 4, the fact that $\eta$ is given as
a character of $G_{F^+}$, and the fact that the cyclotomic character obviously extends
in this way. 
 \end{proof}

\section{Lifting a Galois representation.}\label{sec:lifting}

\subsection{Notation} 

\subsubsection{The group $\mc{G}_n$}
\label{subsubsec: curly G_n}
Let $n$ be a positive integer and let $\mc{G}_n$ be the group scheme
over $\Spec \bb{Z}$ defined in section 2.1 of \cite{cht}, that is, the
semi-direct product of $\GL_n \times \GL_1$ by the group $\{1,j\}$
acting on $\GL_n \times \GL_1$ by $j(g,\mu)j^{-1} = (\mu
{}^tg^{-1},\mu)$. Let $\nu : \mc{G}_n \rightarrow \GL_1$ be the
homomorphism which sends $(g,\mu)$ to $\mu$ and $j$ to $-1$.

Let $F$ be an imaginary CM field with totally real subfield $F^+$ and
let $n$ be a positive integer. Fix a complex conjugation $c \in
G_{F^+}$ and let $\delta_{F/F^+}: G_{F^+}\to \{\pm 1\}$ be the
quadratic character associated to the extension $F/F^+$. If $R$ is a topological ring, then by Lemma 2.1.1 of
\cite{cht} there is a natural bijection between the set of continuous
homomorphisms $\rho : G_{F^+} \to \mc{G}_{n}(R)$ with
$\rho^{-1}(\GL_n(R)\times\GL_1(R))=G_F$ and the set of triples
$(r,\chi,\langle\ ,\ \rangle)$ where $r:G_F \to \GL_n(R)$ and $\chi :
G_{F^+} \to R^{\times}$ are continuous homomorphisms and $\langle\ ,\
\rangle : R^n \times R^n \to R$ is a perfect $R$-bilinear pairing such
that
\begin{itemize}
\item $\langle x,y\rangle = -\chi(c) \langle y,x\rangle$ for all $x,y
  \in R^n$, and
\item $\langle r(\sigma)x,r(c\sigma c)y\rangle = \chi(\sigma)\langle
  x,y\rangle$ for all $\sigma \in G_F$ and $x,y \in R^n$.
\end{itemize}
If $\rho$ corresponds to the triple $(r,\chi,\langle\ ,\ \rangle)$,
then $\chi = \nu \circ \rho$,  $\rho|_{G_F}=(r,\chi|_{G_F})$ and if we write
$\rho(c)=(A,-\chi(c))j$, then $\langle x,y\rangle = {}^{t}xA^{-1}y$ for
all $x,y \in R^n$. We say that $\rho$ extends $r$.

If $R$ is a ring and $\rho : G_{F^+} \to \mc{G}_n(R)$ a homomorphism
with $G_F = \rho^{-1}(\GL_n(R)\times \GL_1(R)$ and $(r,\chi,\langle\
,\ \rangle)$ is the corresponding triple, we will often abuse
notation and write $\rho|_{G_F}$ for $r : G_F \to \GL_n(R)$.

\subsubsection{Oddness (CM case)} 
\label{subsubsec: oddness CM}
Let $k$ be a topological field and
let $r : G_F \to \GL_n(k)$ be a continuous representation.  We say
that $r$ is \emph{essentially conjugate-self-dual} (ECSD), if there
exists a continuous character $\mu : G_{F} \to k^\times$ such that
$r^c \cong r^\vee \mu$ and such that $\mu$ can be extended to a
character $\chi : G_{F^+} \to k^\times$ with $\chi(c_v)$ independent
of $v|\infty$ (where $c_v$ denotes a complex conjugation at a place
$v|\infty$). Note that if such an extension $\chi$ exists then there
is one other, namely $\chi \delta_{F/F^+}$, and hence there is one totally odd extension and one totally even extension.

Assume that $r: G_F \to \GL_n(k)$ is ECSD and let $\mu: G_{F} \rightarrow k^\times$ be as above. Let
$\langle\ ,\ \rangle : k^n \times k^n \to k$ be a perfect bilinear
pairing giving rise to the isomorphism $r^c \cong r^\vee \mu$ in the
sense that $\langle r(\sigma)x,r(c\sigma c)y\rangle =
\mu(\sigma)\langle x,y\rangle$ for all $\sigma \in G_F$ and $x,y \in
k^n$. We say that the triple $(r, \mu, \langle\ ,\ \rangle)$ is
\emph{odd} if $\langle\ ,\ \rangle$ is symmetric. If $r$ is absolutely
irreducible, then $\langle\ ,\ \rangle$ is unique up to scaling by
elements of $k^\times$. In this case we say that $(r,\mu)$ is odd if
$(r,\mu,\langle\ ,\ \rangle)$ is odd for one (hence any) choice of
$\langle\ ,\ \rangle$. If $\mu$ is clear from the
context, we will sometimes just say that $r$ is odd if $(r,\mu)$ is
odd.

Let $r$ and $\mu$ be as in the previous paragraph and assume that $r$
is absolutely irreducible. Then there is a bilinear pairing $\langle\
,\ \rangle: k^n \times k^n \to k$, unique up to scaling, giving rise
to the isomorphism $r^c \cong r^\vee \mu$. This pairing satisfies
$\langle x,y \rangle = (-1)^a \langle y, x \rangle$ for some $a \in
\bb{Z}/2\bb{Z}$. Let $\chi : G_{F^+} \to k^\times$ be the unique
extension of $\mu$ to $G_{F^+}$ with $\chi(c_v)=(-1)^{a+1}$ for all
$v|\infty$. Then the triple $(r,\chi,\langle\ ,\ \rangle)$ corresponds
to a continuous homomorphism $ G_{F^+} \to \mc{G}_n(k)$
extending $r$. Moreover, any other extension of $r$ to a homomorphism
$G_{F^+} \to \mc{G}_n(k)$ corresponds to a triple
$(r,\chi,\alpha\langle\ ,\ \rangle)$ for some $\alpha \in
k^\times$. In particular, $r$ is odd if and only if for one (hence
any) extension $\rho : G_{F^+} \to \mc{G}_n(k)$ of $r$, the character
$\chi = \nu \circ \rho$ is odd.

\subsubsection{Oddness (totally real case)} 
\label{subsubsec: oddness totally real}
Let $k$ be a topological
field and let $r : G_{F^+} \to \GL_n(k)$ be a continuous
representation. We say that $r$ is \emph{essentially-self-dual} (ESD),
if there exists a continuous character $\mu : G_{F^+} \to k^\times$
such that $r \cong r^\vee \mu$ and such that $\mu(c_v)$ is independent
of $v|\infty$.

Assume that $r: G_{F^+} \to \GL_n(k)$ is ESD and let $\mu: G_{F^+}
\rightarrow k^\times$ be as above. Assume also that the isomorphism $r
\cong  r^\vee \mu$ is realised by a perfect bilinear pairing $(\ ,\
):k^n \times k^n \to k$ satisfying $(r(\sigma)x,r(\sigma)y) =
\mu(\sigma)(x,y)$ for all $\sigma \in G_{F^+}$ and $x,y \in k^n$. We
say that the triple $(r,\mu,(\ ,\ ))$ is \emph{odd} if either (i) $(\
,\ )$ is symmetric and $\mu$ is even, or (ii) $(\ ,\ )$ is alternating
and $\mu$ is odd.

We now explain the relation between oddness in the totally real case and oddness in the CM case. Let $(r, \mu, (\ ,\ ))$ be a triple as in the previous paragraph (not
assumed to be odd). Then $r|_{G_F}$ satisfies $(r|_{G_F})^c \cong
r|_{G_F} \cong (r|_{G_F})^\vee \mu|_{G_F}$. Moreover, let $J\in
M_n(k)$ be the matrix with $(x,y)= {}^{t}x J y$ and define a new
pairing $\langle\ ,\ \rangle : k^n \times k^n \to k$ by $\langle
x,y\rangle = {}^tx J r(c)y$. Then $\langle r(\sigma)x,r(c\sigma
c)y\rangle = \mu(\sigma)\langle x,y \rangle$ for all $\sigma \in
G_{F^+}$ and $x,y \in k^n$ and hence $\langle\ ,\ \rangle$ realises
the isomorphism $(r|_{G_F})^c \cong (r|_{G_F})^\vee \mu|_{G_F}$. Note
that $\langle\ ,\ \rangle$ is symmetric if and only if ${}^tJ=\mu(c)J$
which occurs if and only (i) or (ii) above hold. In particular, we see
that $(r,\mu,(\ ,\ ))$ is odd if and only if
$(r|_{G_F},\mu|_{G_F},\langle\ ,\ \rangle)$ is odd.

Suppose $r$ and $\mu$ are as in the previous two paragraphs. If $r$ is
absolutely irreducible, then up to scaling there is a unique pairing
$(\ ,\ )$ giving rise to the isomorphism $r\cong r^\vee \mu$. This
pairing is either symmetric or alternating. We say that $(r,\mu)$ is
odd if $(r,\mu,(\ ,\ ))$ is odd. If $\mu$ is clear from the context we
will say that $r$ is odd if this holds.

Note that parts (3) and (4) of Lemma \ref{lem:char building} give rise
to odd triples.

\subsubsection{Standard basis for tensor products}
\label{subsubsec: standard basis for tensor products}
\begin{defn} \label{defn:inherited basis for otimes}
Let $R$ be a ring and let $V$ and $W$ be two finite free $R$-modules
of rank $n$ and $m$ respectively. Let $e_1,\ldots,e_n$ be an $R$-basis of
$V$ and $f_1,\ldots,f_m$ an $R$-basis of $W$. Let \emph{the basis of $V\otimes_R W$ inherited from
the bases $\{e_i\}$ and $\{f_j\}$} be the ordered $R$-basis of
$V \otimes_R W$ given by the vectors $e_i \otimes f_j$, ordered
lexicographically. Let $G$ be a group let
and $r : G \to \GL_n(R)$ and $s : G\to \GL_m(R)$ represent $R$-linear
actions of $G$ on $V$ and $W$ with respect to the bases $\{e_i\}$ and
$\{f_j\}$. Then we let $r\otimes s: G \to \GL_{nm}(R)$ denote the
action of $G$ on $V\otimes_R W$ with respect to the inherited basis of
$V\otimes_R W$.
\end{defn}

\subsubsection{Galois deformations}
\label{subsubsec: galois deformations}
Let $l$ be a prime number and fix an algebraic closure $\Qlbar$ of $\Ql$. Let $K$ be a finite extension of $\Ql$ inside $\Qlbar$ with ring of integers $\mc{O}$ and residue field $k$. Let $\mf{m}_{\mc{O}}$ denote the maximal ideal of $\mc{O}$. Let
$\mathcal{C}_\bigO$ be the category of complete local Noetherian $\bigO$-algebras
with residue field isomorphic to $k$ via the structural homomorphism. As in section 3 of \cite{BLGHT}, we consider an object $R$ of $\mc{C}_{\mc{O}}$ to be geometrically integral if for all finite extensions $K'/K$, the algebra $R\otimes_{\mc{O}}\mc{O}_{K'}$ is an integral domain.

\subsubsection*{Local lifting rings}
\label{subsubsec: local lifting rings}

Let $M$ be a finite extension of $\bb{Q}_p$ for some prime $p$
possibly equal to $l$ and let $\rhobar : G_M \rightarrow \GL_n(k)$ be
a continuous homomorphism.
Then the functor from $\mathcal{C}_\bigO$ to
$Sets$ which takes $A\in\mathcal{C}_\bigO$ to the set of continuous
liftings $\rho:G_M\to\GL_n(A)$ of $\rhobar$  is represented by
a complete local Noetherian $\bigO$-algebra $R^{\square}_{\rhobar}$. We call this
ring the universal $\mc{O}$-lifting ring of $\rhobar$.
 We write
$\rho^{\square}:G_M\to\GL_n(R^{\square}_{\rhobar})$ for the universal
lifting. 

Assume now that $p=l$ so that $M$ is a finite extension of
$\bb{Q}_l$. Assume also that $K$ contains the image of every embedding $M \into \overline{K}$.

\begin{defn} 
  Let $(\bb{Z}^{n}_{+})^{\Hom(M,K )}$ denote the subset of
  $(\bb{Z}^{n})^{\Hom(M,K )}$ consisting of elements $\lambda$ which
  satisfy
  \[ \lambda_{\tau,1} \geq \lambda_{\tau,2} \geq \ldots \geq
  \lambda_{\tau,n} \] for every embedding $\tau$.
\end{defn}

Let $a$ be an element of $(\bb{Z}^{n}_{+})^{\Hom(M,K)}$.  We
associate to $a$ an $l$-adic Hodge type $\mathbf{v}_{a}$
in the sense of section 2.6 of \cite{kisinpst} as follows. Let $D_{K}$
denote the vector space $K^{n}$. Let $D_{K,M}=D_{K} \otimes_{\bb{Q}_l}
M$. For each embedding $\tau : M \hookrightarrow K$, we let
$D_{K,\tau}=D_{K,M} \otimes_{K\otimes M,1\otimes \tau}K $ so that
$D_{K,M} = \oplus_{\tau} D_{K,\tau}$. For each $\tau$ choose a
decreasing filtration $\Fil^{i}D_{K,\tau}$ of $D_{K,\tau}$ so that
$\dim_{K} \gr^{i}D_{K,\tau} = 0 $ unless $ i =
(j-1)+a_{\tau,n-j+1}$ for some $j=1,\ldots,n$ in which case
$\dim_{K} \gr^{i}D_{K,\tau}=1$. We define a decreasing filtration of
$D_{K,M}$ by $K \otimes_{\bb{Q}_l} M$-submodules by setting
\[ \Fil^{i} D_{K,M} = \oplus_{\tau} \Fil^{i}D_{K,\tau}.\] Let
$\mathbf{v}_{a}= \{ D_{K}, \Fil^{i}D_{K,M} \}$. 

We now recall some results of Kisin. Let $a$ be an element of $(\bb{Z}^{n}_{+})^{\Hom(M,K)}$ and let
$\mathbf{v}_{a}= \{ D_{K}, \Fil^{i}D_{K,M} \}$ be the associated $l$-adic Hodge type.
\begin{defn}\label{defn: l-adic hodge type}
  If $B$ is a finite $K$-algebra and $V_{B}$ is a free $B$-module of
  rank $n$ with a continuous action of $G_{M}$ that makes $V_B$ into a
  de Rham representation, then we say that \emph{$V_B$ is of $l$-adic Hodge
  type $\mathbf{v}_{a}$} if for each $i$ there is an isomorphism
  of $B \otimes_{\bb{Q}_l} M$-modules
\[ \gr^{i}(V_{B} \otimes_{\bb{Q}_l} B_{dR})^{G_{M}}
\tilde{\rightarrow} (\gr^{i}D_{K,M}) \otimes_{K} B. \] 
\end{defn}

Corollary 2.7.7 of \cite{kisinpst} implies that there is a unique
$l$-torsion free quotient $R^{\mathbf{v}_{a},cr}_{\rhobar}$ of
$R^{\square}_{\rhobar}$ with the property that for any finite
$K$-algebra $B$, a homomorphism of $\mc{O}$-algebras $\zeta :
R^{\square}_{\rhobar}\rightarrow B$ factors through
$R^{\mathbf{v}_{a},cr}_{\rhobar}$ if and only if $\zeta \circ
\rho^{\square}$ is crystalline of $l$-adic Hodge type
$\mathbf{v}_{a}$. Moreover, Theorem 3.3.8 of \cite{kisinpst}
implies that $\Spec R^{\mathbf{v}_{a},cr}_{\rhobar}[1/l]$ is
formally smooth over $K$ and equidimensional of dimension $n^2 +
\frac{1}{2}n(n-1)[M:\bb{Q}_l]$.

\begin{defn}
\label{defn: sim}
Suppose that $\rho_1,\rho_2 : G_M\to\GL_n(\bigO)$ are two continuous lifts
of $\rhobar$. Then we say that
$\rho_1\sim \rho_2$ if the following hold.
\begin{enumerate}
\item There is an $a \in(\bb{Z}^{n}_{+})^{\Hom(M,K)}$
such that $\rho_1$ and $\rho_2$ both correspond to points of
$R^{\mathbf{v}_{a},cr}_{\rhobar}$ (that is,
$\rho_1\otimes_\bigO K$ and $\rho_2\otimes_\bigO K$ are both crystalline of $l$-adic Hodge
  type $\mathbf{v}_{a}$).
\item For every minimal prime ideal $\wp$ of
  $R^{\mathbf{v}_{a},cr}_{\rhobar}$, the quotient
  $R^{\mathbf{v}_{a},cr}_{\rhobar}/\wp$ is geometrically
  integral.
\item $\rho_1$ and $\rho_2$ give rise to closed points on a common irreducible
  component of  $\Spec R^{\mathbf{v}_{a},cr}_{\rhobar}[1/l]$.
\end{enumerate}
\end{defn}

In (3) above, note that the irreducible component is uniquely
determined by either of $\rho_1$, $\rho_2$ because $\Spec
R^{\mathbf{v}_{a},cr}_{\rhobar}[1/l]$ is formally smooth. Note
also that we can always ensure that (2) holds by replacing $\mc{O}$ with the
ring of integers in a finite extension of $K$.

\subsubsection*{Global deformation rings}
\label{subsubsec: global deformation rings}
Let $F/F^{+}$ be a totally imaginary quadratic extension of a totally
real field $F^+$. Let $c$ denote the non-trivial element of $\Gal(F/F^+)$. Assume that $K$ contains the image of every
embedding $F \into \Qlbar$ and that the prime $l$ is odd. Assume that
every place in $F^+$ dividing $l$ splits in $F$. Let $S$ denote a finite set of finite places of $F^+$ which split in
$F$, and assume that $S$ contains every place dividing $l$. Let $S_l$ denote the set of places of $F^+$ lying over $l$. Let $F(S)$
denote the maximal extension of $F$ unramified away from $S$. Let
$G_{F^{+},S}=\Gal(F(S)/F^{+})$ and $G_{F,S}=\Gal(F(S)/F)$. For each $v
\in S$ choose a place $\wt{v}$ of $F$ lying over $v$ and let $\wt{S}$
denote the set of $\wt{v}$ for $v \in S$. For each place $v|\infty$ of $F^+$ we let
$c_v$ denote a choice of a complex conjugation at $v$ in
$G_{F^+,S}$.  For each place $w$ of $F$
we have a $G_{F,S}$-conjugacy class of homomorphisms $G_{F_w}
\rightarrow G_{F,S}$. For $v \in S$ we fix a choice of homomorphism
$G_{F_{\wt{v}}} \rightarrow G_{F,S}$.

Fix a continuous homomorphism
\[ \rhobar : G_{F^+,S} \rightarrow \mc{G}_n(k) \] such that $G_{F,S} =
\rhobar^{-1}(\GL_n(k)\times \GL_1(k))$ and fix a continuous character
$\chi : G_{F^+,S}\rightarrow \mc{O}^{\times} $ such that $\nu \circ
\rhobar = \overline{\chi}$. Let $(\rbar,\chibar,\langle\ ,\ \rangle)$ be the triple corresponding to $\rhobar$ (see section \ref{subsubsec: curly G_n}).
Assume that $\rbar$ is absolutely
irreducible.
As in Definition 1.2.1 of \cite{cht}, we define
\begin{itemize}
\item a \emph{lifting} of $\rhobar$ to an object $A$ of
  $\mc{C}_{\mc{O}}$ to be a continuous homomorphism $\rho : G_{F^+,S}
  \rightarrow \mc{G}_n(A)$ lifting $\rhobar$ and with $\nu \circ \rho =
  \chi$;
\item two liftings $\rho$, $\rho^{\prime}$ of $\rhobar$ to $A$ to be
  \emph{equivalent} if they are conjugate by an element of
  $\ker(\GL_n(A)\rightarrow \GL_n(k))$;
\item a \emph{deformation} of $\rhobar$ to an object $A$ of
  $\mc{C}_{\mc{O}}$ to be an equivalence class of liftings.
\end{itemize}

For each place $v \in S$, let $R^{\square}_{\rbarwtv}$ denote
the universal $\mc{O}$-lifting ring of $\rbar|_{G_{F_{\wt{v}}}}$ and
let $R_{\wt{v}}$ denote a quotient of $R^{\square}_{\rbarwtv}$ which satisfies the
following property: 
\begin{itemize}
\item[(*)] let $A$ be an object of $\mc{C}_{\mc{O}}$ and let
$\zeta,\zeta^{\prime}: R^{\square}_{\rbarwtv}\rightarrow A$ be homomorphisms corresponding to
two lifts $r$ and $r^{\prime}$ of $\rbarwtv$ which are
conjugate by an element of $\ker(\GL_n(A)\rightarrow
\GL_n(k))$. Then $\zeta$ factors through $R_{\wt{v}}$ if and only if
$\zeta^{\prime}$ does.
\end{itemize}
We
consider the \emph{deformation problem}
\[ \mc{S} = (F/F^{+},S,\wt{S},\mc{O},\rhobar,\chi,\{ R_{\wt{v}}\}_{v \in
  S}) \] 
(see sections 2.2 and 2.3 of \cite{cht} for this terminology).
We say that a lifting $\rho : G_{F^+,S}\rightarrow \mc{G}_n(A)$ is
\emph{of type $\mc{S}$} if for each place $v \in S$, the homomorphism
$R^{\square}_{\rbarwtv} \rightarrow A$ corresponding to $\rho|_{G_{F_{\wt{v}}}}$ factors
through $R_{\wt{v}}$. We also define deformations of type $\mc{S}$ in the same way.
Let $\Def_{\mc{S}}$ be the functor $ \mc{C}_{\mc{O}}\rightarrow Sets$
which sends an algebra $A$ to the set of deformations of $\rbar$ to
$A$ of type $\mc{S}$. 
By Proposition 2.2.9 of
\cite{cht} this functor is represented by an object
$R_{\mc{S}}^{\univ}$ of $\mc{C}_{\mc{O}}$. The next lemma follows from Lemma 3.2.3 of \cite{GG}.

\begin{lem}
  \label{lem: property (*)}
  Let $M$ be a finite extension of $\bb{Q}_p$ for some prime $p$. Let
  $\rbar : G_M \rightarrow \GL_n(k)$ be a continuous
  homomorphism. If $p\neq l$, let $R$ be a quotient of the maximal $l$-torsion free
  quotient of $R^{\square}_{\rbar}$ corresponding to a union of
  irreducible components. If $p=l$,
  assume that $K$ contains the image of each embedding $M \into
  \Qlbar$ and
  let $R$ be a quotient of $R^{\mathbf{v}_{a},cr}_{\rbar}$, for some $a \in
  (\bb{Z}^n_+)^{\Hom(M,K)}$, corresponding to a union of irreducible components. Then $R$ satisfies property $(*)$ above.
\end{lem}

\subsection{Relative finiteness for deformation rings: restriction and tensor products}
\label{subsec: relative finiteness for deformation rings of tensor products}
Let $m$ and $n$ be positive integers, and let $l>mn$ be a rational
prime. Let $F \subset F'$ be imaginary CM fields with maximal totally
real subfields $F^+$ and $(F')^+$ respectively. Let $S$ be a finite
set of finite places of $F^+$ which split in $F$ and let $\tilde{S}$
be a set of places of $F$ consisting of exactly one place lying over
each place of $S$. Assume that each place of $F^+$ which lies over $l$
is in $S$ (and hence splits in $F$). Let $S'$ and $\tilde{S}'$ be the sets of places of $(F')^+$
and $F'$ which lie over $S$ and $\tilde{S}$ respectively. If $v$ is a
place in $S$ (respectively $S'$) we write $\tilde{v}$ for the unique
place of $\tilde{S}$ (respectively $\tilde{S}'$) lying over $v$. Fix a
choice of complex conjugation $c \in G_{(F')^{+}}$.

Let $K \subset \Qlbar$ be a finite extension of $\Ql$ with ring of
integers $\bigO$ and residue field $k$.  Assume that $K$ contains the
image of every embedding $F'\into\Qlbar$.

  Suppose that $\rbar:G_{F,S}\to\GL_n(k)$ and $r' : G_{F',S'} \to
  \GL_m(\mc{O})$ are continuous representations. Let $\rbar'' =
  \rbar|_{G_{F'}}\otimes_k \rbar' : G_{F',S'} \to \GL_{mn}(k)$ (we are
  using the conventions of Definition \ref{defn:inherited basis for
    otimes} to regard $\rbar''$ as a $\GL_{mn}(k)$-valued
  homomorphism). Suppose that $\chi : G_{F^+,S} \to \mc{O}^\times$,
  $\chi',\chi'':G_{(F')^+,S'}\to \mc{O}^\times$ are continuous characters,
  and assume that
\begin{enumerate}
  \item $\rbar$, $\rbar'$ and $\rbar''$ are absolutely irreducible.
  \item $\rbar^c\cong\rbar^\vee\chibar|_{G_F}$.
  \item $(r'\otimes_{\mc{O}}K)^c\cong
    (r'\otimes_{\mc{O}}K)^\vee\chi'|_{G_{F'}}$. 
 \item $\chi''|_{G_{F'}}=\chi|_{G_{F'}}\chi'|_{G_{F'}}$.
 \item $\chi,\chi'$ and $\chi''$ are each totally odd or totally even.
\end{enumerate}

The first and second assumptions imply that we can find a perfect 
bilinear pairing $\langle\ ,\ \rangle : k^n \times k^n \rightarrow k$,
unique up to scaling by elements of $k^\times$, such that $\langle
\rbar(\sigma)x,\rbar(c\sigma c)y\rangle =
\chibar(\sigma)\langle x,y\rangle$ for all $\sigma
\in G_F$, $x,y \in k^n$. We must have $\langle x,y\rangle = (-1)^a
\langle y,x\rangle$ for some $a \in \bb{Z}/2\bb{Z}$. 
Note that $r'\otimes_\bigO K$ is absolutely irreducible by the first
assumption. This and the third assumption then imply that we can find
a perfect symmetric bilinear pairing $\langle\ ,\ \rangle' : K^m
\times K^m \to K$, unique up to scaling by elements of $K^\times$,
with $\langle r'(\sigma)x,r'(c\sigma c)y\rangle' =
\chi'(\sigma)\langle x,y\rangle'$ for all $\sigma \in G_{F'}$, $x,y
\in K^m$. The first assumption implies that the dual lattice of
$\mc{O}^m$ under this pairing is $\lambda^a \mc{O}^m$ for some $a \in
\bb{Z}$. Replacing $\langle\ ,\ \rangle'$ by $\alpha \langle\ ,\
\rangle'$ where $\alpha \in K$ satisfies $\lambda^a = (\alpha)$, we
may assume that $\mc{O}^m$ is self-dual under $\langle\ ,\
\rangle'$. We necessarily have $\langle x,y\rangle' = (-1)^b \langle
y,x\rangle'$ for some $b\in \bb{Z}/2\bb{Z}$. We assume
\begin{enumerate}
\item[(6)] $\chi(c)=(-1)^{a+1}$, $\chi'(c)=(-1)^{b+1}$ and $\chi''(c)=(-1)^{a+b+1}$. 
\end{enumerate}

The discussion of section \ref{subsubsec: curly G_n} shows that the
triple $(\rbar,\chibar,\langle\ ,\ \rangle)$ corresponds to a
homomorphism $\rhobar:G_{F^+,S}\to\G_n(k)$ with
$\rhobar|_{G_{F}}=(\rbar,\chibar|_{G_F})$ and
$\nu\circ\rhobar=\chibar$. Similarly, the triple $(r,\chi',\langle\ ,\
\rangle')$ corresponds to a homomorphism $\rho':G_{(F')^+,S'} \to
\mc{G}_m(\mc{O})$ with $\rho'|_{G_{F'}}=(r,\chi'|_{G_{F'}})$ and $\nu
\circ \rho' = \chi'$.

Note that $(\rbar'')^c \cong
(\rbar'')^{\vee}\chibar|_{G_{F'}}\chibar'|_{G_{F'}}= (\rbar'')^\vee \chibar''|_{G_{F'}}$. The pairings $\langle\
,\ \rangle$ and $\langle\ ,\ \rangle'$ induce a perfect
bilinear pairing $\langle\ ,\ \rangle'': k^{mn}\times k^{mn} \to k$
with $\langle x\otimes y, w\otimes z\rangle '' = \langle x ,w\rangle
\langle y ,z \rangle'$ for all $x,w \in k^n$ and $y,z \in k^m$ (we are
identifying $k^n\otimes_k k^m$ with $k^{mn}$ as above). This pairing
satisfies $\langle \rbar''(\sigma)u ,\rbar''(c \sigma c)v\rangle'' =
\chibar''(\sigma)\langle u,v\rangle''$ for all
$\sigma \in G_{F'}$, $u,v \in k^{mn}$. Since $\langle
x,y\rangle''=(-1)^{a+b}\langle x,y\rangle''$ and
$\chi''(c)=(-1)^{a+b+1}$, we see that the triple
$(\rbar'',\chibar'',\langle\ ,\
\rangle'')$ corresponds to a homomorphism $\rhobar'': G_{(F')^+,S'} \to
\mc{G}_{mn}(k)$ with $\rhobar''|_{G_{F'}}=(\rbar'',\chibar|_{G_{F'}})$
and $\nu \circ \rhobar'' = \chibar''$.

We now turn to deformation theory. For each place $v \in S$
(respectively $w\in S'$), let $R^{\square}_{\rbar|_{G_{F_{\tv}}}}$
(respectively $R^{\square}_{\rbar''|_{G_{F'_{\tw}}}}$)
be the universal $\mc{O}$-lifting ring of $\rbar|_{G_{F_{\tv}}}$
(respectively $\rbar''|_{G_{F'_{\tw}}}$).
Consider the deformation problem
\[ \mc{S}_0 =
(F/F^{+},S,\wt{S},\mc{O},\rhobar,\chi,\{
R^{\square}_{\rbar|_{G_{F_{\tv}}}} \}_{v \in S}).\] 
See section
\ref{subsubsec: global deformation rings} for this notation.  Let
$R_{\mc{S}_0}^{\univ}$ be the $\mc{O}$-algebra representing the
corresponding deformation functor, and let
$\rho^{\univ}_{\mc{S}_0}:G_{F^+,S}\to\G_n(R^{\univ}_{\mc{S}_0})$ be a
representative of the universal deformation. Similarly, consider the deformation problem
\[ \mc{S}''_0 =
(F'/(F')^{+},S',\wt{S}',\mc{O},\rhobar'',\chi'',\{
R^{\square}_{\rbar''|_{G_{F'_{\tw}}}} \}_{w \in S'}). \] Let $R_{\mc{S}''_0}^{\univ}$ be the
$\mc{O}$-algebra representing the corresponding deformation functor,
and let
$\rho^{\univ}_{\mc{S}''_0}:G_{(F')^+,S'}\to\G_{mn}(R^{\univ}_{\mc{S}''_0})$
be a representative for the universal deformation.

Suppose we are given an object $A$ of $\mc{C}_{\mc{O}}$ and a
continuous homomorphism $\rho_A : G_{F^+,S}\to \mc{G}_n(A)$ lifting
$\rhobar$ and with $\nu \circ \rho_A =
\chi$. Let
$(r_A,\chi,\langle\ ,\ \rangle_A)$ be the
corresponding triple (see section \ref{subsubsec: curly
  G_n}). Identifying $A^n\otimes_A A^m$ with $A^{mn}$ as in Definition
\ref{defn:inherited basis for otimes}, we obtain a homomorphism
$r''_A :=(r_A|_{G_{F'}}) \otimes_A (r'\otimes_{\mc{O}}A) : G_{F'} \to
\GL_{mn}(A)$ and a perfect bilinear pairing $\langle\ ,\
\rangle_A'': A^{mn}\times A^{mn} \to A$ with $\langle x\otimes y,
w\otimes z\rangle_A '' = \langle x ,w\rangle_A \langle y ,z \rangle'$ for
all $x,w \in A^n$ and $y,z \in A^m$. This pairing satisfies $\langle
r''_A(\sigma)u ,r''_A(c \sigma c)v\rangle_A'' =
\chi''(\sigma)\langle u,v\rangle_A''$ for all $\sigma \in
G_{F'}$, $u,v \in A^{mn}$. The triple $(r''_A,\chi'',\langle\
,\ \rangle_A'')$ therefore gives rise to a homomorphism $\rho''_A:
G_{(F')^+,S'} \to \mc{G}_{mn}(A)$ lifting $\rhobar''$. We also denote
$\rho''_A$ by $(\rho_A|_{G_{(F')^+}})\otimes
(\rho'\otimes_{\mc{O}}A)$. Taking $A=R^{\univ}_{\mc{S}_0}$ and $\rho_A
= \rho^{\univ}_{\mc{S}_0}$, we obtain an
$\bigO$-algebra homomorphism \[\theta:R_{\mc{S}''_0}^{\univ}\to
R_{\mc{S}_0}^{\univ}\] with the property that the homomorphisms
$\theta\circ\rho^{\univ}_{\mc{S}''_0}$ and
$(\rho^{\univ}_{\mc{S}_0}|_{G_{(F')^+}})\otimes
(\rho'\otimes_{\bigO}R_{\mc{S}_0}^{\univ})$ are
$\ker(\GL_{mn}(R_{\mc{S}_0}^{\univ})\to\GL_{mn}(k))$-conjugate. 

\begin{lem}
  \label{lem:deformation ring is finite over one obtained by tensor}
Let 
\begin{itemize}
\item $\rbar : G_{F,S} \to \GL_n(k)$, $r' : G_{F',S'} \to
  \GL_m(\mc{O})$, $\rbar''=\rbar|_{G_{F'}} \otimes \rbar' : G_{F',S} \to \GL_{mn}(k)$
\item  $\chi : G_{F^+,S} \to \mc{O}^\times$,
  $\chi':G_{(F')^+,S'}\to \mc{O}^\times$, and $\chi'':G_{(F')^+,S'}\to \mc{O}^\times$
\end{itemize}
be continuous homomorphisms satisfying assumptions (1)-(6) above. Then
the homomorphism $\theta:R_{\mc{S}''_0}^{\univ}\to
R_{\mc{S}_0}^{\univ}$ constructed above is finite.
\end{lem}

\begin{proof}
  Let $\wp$ denote a prime ideal of
  $R_{\mc{S}_0}^{\univ}/\theta(\mf{m}_{R^{\univ}_{\mc{S}''_0}})$ and let
  $\Rbar =
  (R_{\mc{S}_0}^{\univ}/\theta(\mf{m}_{R^{\univ}_{\mc{S}''_0}}))/\wp$.  It
  suffices to show that $\Rbar$ is a finite $k$-algebra (for any
  choice of $\wp$), since if this holds then we see that
  $R_{\mc{S}_0}^{\univ}/\theta(\mf{m}_{R^{\univ}_{\mc{S}''_0}})$ is a
  Noetherian ring of dimension 0 and hence is Artinian. The result then
  follows from the topological version of Nakayama's lemma.

  Let $\rho_{\Rbar} =
  \rho^{\univ}_{\mc{S}_0}\otimes_{R_{\mc{S}_0}^{\univ}}\Rbar$ and
  $\rho''_{\Rbar}=(\rho_{\Rbar}|_{G_{(F')^+}}) \otimes (\rhobar'
  \otimes_{k} \Rbar)$. Note that $\rho_{\Rbar}''$ is
  $\ker(\GL_{mn}(\Rbar)\to \GL_{mn}(k))$ conjugate to $\rhobar''
  \otimes_k \Rbar$ and in particular has finite image. Let $r_{\Rbar}
  : G_F \to\GL_n(\Rbar)$ denote $\rho_{\Rbar}|_{G_{F}}$ composed with
  the projection $\GL_n(\Rbar)\times\GL_1(\Rbar) \to \GL_n(\Rbar)$.
  Then $r_{\Rbar}|_{G_{F'}} \otimes \rbar' : G_{F'} \to
  \GL_{mn}(\Rbar)$ has finite image. From this it follows easily that
  $r_{\Rbar}$ has finite image. However, Lemma 2.1.12 of \cite{cht}
  implies that $\Rbar$ is topologically generated as a $k$-algebra by
  $\tr(r_{\Rbar}(G_F))$. Since $r_{\Rbar}$ has finite image, we deduce
  that for each $\sigma \in G_F$, $\tr(r_{\Rbar}(\sigma))$ is a sum of
  roots of unity of bounded order. It follows that $\Rbar$ is a finite
  $k$-algebra, as required.
\end{proof}

We also consider refined deformation problems as follows. For each place $v \in S$
let $R_{\tv}$ be a quotient of $R^{\square}_{\rbar|_{G_{F_{\tv}}}}$
satisfying the condition (*) of section \ref{subsubsec: global
  deformation rings}. For each place $w \in S'$, let $R_{\tw}$ be a
quotient of $R^{\square}_{\rbar''|_{G_{F'_{\tw}}}}$ satisfying the
same property (*). Furthermore, if $v$ denotes the place of $S$ lying
under $w$, assume that given any object $A$ of $\mc{C}_{\mc{O}}$ and
any $\mc{O}$-algebra homomorphism $R_{\tv} \to A$ corresponding to a lift $r_A$
of $\rbar|_{G_{F_{\tv}}}$, the lift $r_A|_{G_{F'_{\tw}}} \otimes_A
(r'|_{G_{F'_{\tw}}}\otimes_{\mc{O}}A)$ (regarded as a homomorphism to
$\GL_{mn}(A)$ using the conventions of Definition \ref{subsubsec:
  global deformation rings}) of $\rbar''|_{G_{F'_{\tw}}}$ gives rise
to an $\mc{O}$-algebra homomorphism
$R^{\square}_{\rbar''|_{G_{F'_{\tw}}}}\to A$ which factors through
$R_{\tw}$. We let
\begin{eqnarray*}
 \mc{S} &=&
(F/F^{+},S,\wt{S},\mc{O},\rhobar,\chi,\{
R_{\tv} \}_{v \in S}) \\
  \mc{S}'' &=&
(F'/(F')^{+},S',\wt{S}',\mc{O},\rhobar'',\chi'',\{
R_{\tw} \}_{w \in S'})
\end{eqnarray*}
and let $R^{\univ}_{\mc{S}}$ and $R^{\univ}_{\mc{S}''}$ be 
objects representing the corresponding deformation problems.
The compatibility between the
rings $R_{\tv}$ and $R_{\tw}$ for $v \in S$, $w\in S'$ implies that
the map $\theta : R^{\univ}_{\mc{S}''_0} \to R^{\univ}_{\mc{S}_0}$
gives rise to a map \[ \overline{\theta}: R^{\univ}_{\mc{S}''}\to
R^{\univ}_{\mc{S}}.\] The following result follows immediately from \ref{lem:deformation ring is finite over
  one obtained by tensor}.

\begin{lem}
  \label{lem:refined deformation ring is finite over
  one obtained by tensor}
The map $ \overline{\theta}: R^{\univ}_{\mc{S}''}\to
R^{\univ}_{\mc{S}}$ is finite.
\end{lem}

\subsection{Relative finiteness for deformation rings: induction}
\label{subsec: relative finiteness for deformation rings: induction}

The results of this section are not needed in this paper, but are used
in \cite{BLGGT}.

Let $n$ be a positive integer. Let $F_1$ and $F_2$ be imaginary
CM fields with $F_1^+ \subset F_2^+$. Let
$m:=[F_2^+:F_1^+]$. Assume also that every prime of $F_i^+$ above $l$
splits in $F_i$ for $i=1,2$. For $i=1,2$, let $S_i$ be a finite set of
finite places of $F^+_i$ which split in $F_i$ with every place above
$l$ being contained in $S_i$. Let $\wt{S}_i$ be a set of places of
$F_i$ consisting of exactly one place of $F_i$ lying over each place
of $S_i$. If $v\in S_i$, we denote by the $\wt{v}$ the unique place of
$\wt{S}_i$ lying over $v$. Assume that $S_1$ contains the restriction
to $F_1^+$ of every place in $S_2$. Assume also that $S_1$ contains
every prime of $F_1^+$ that ramifies in $F_2$. Fix a choice of
complex conjugation $c \in G_{F_2^+}$.

Let $K \subset \Qlbar$ be a finite extension of $\Ql$ with ring of
integers $\bigO$ and residue field $k$. Assume that $K$ contains the
image of every embedding $F_2 \into \Qlbar$.

Let
\begin{eqnarray*} 
\rbar_2 : G_{F_2,S_2} \to \GL_{n}(k) \\
\chi : G_{F_1^+,S_1} \to \mc{O}^\times
\end{eqnarray*}
be continuous representations and let
\[ \rbar_1 = \Ind_{G_{F_2}}^{G_{F_1^+}} \rbar_2 : G_{F_1^+,S_1} \to
\GL_{2mn}(k).\]
Assume that
\begin{enumerate}
\item $\rbar_1|_{G_{F_1}}$ and $\rbar_2$ are absolutely irreducible.
\item $\rbar_2^c \cong \rbar_2^\vee \chibar|_{G_{F_2}}$ 
\item $\chi$ is totally odd or totally even.
\end{enumerate}
Choose a perfect bilinear pairing $\psibar_2 : k^n \otimes
k^n \to k$ such that
\[ \psibar_2 ( \rbar_2(\sigma)x, \rbar_2(c\sigma c)y ) = \chibar(\sigma) \psibar_2(x,y) \]
for all $x,y \in k^n$ and $\sigma \in G_{F_2}$.
(Such a pairing exists and is unique up to scaling as $\rbar_2$ is absolutely
irreducible.) Assume further that
\begin{itemize}
\item[(4)] $\psibar_2(x,y) = -\chibar(c)\psibar_2(y,x)$ for all $x,y \in k^n$.
\end{itemize}
Then the triple $(\rbar_2,\chi|_{G_{F_2^+}},\psibar_2)$ gives rise to
a homomorphism $\rhobar_2 : G_{F_2^+,S_2} \to \mc{G}_n(k)$ with
$G_{F_2,S_2}=\rhobar_2^{-1}(\GL_n(k)\times \GL_1(k))$.

For each $v \in S_2$, let
$R^\square_{\rbar_2|_{G_{F_{2,\wt{v}}}}}$ denote the universal
$\mc{O}$-lifting ring of $\rbar_2|_{G_{F_{2,\wt{v}}}}$. Let
$\mc{S}_2$ denote the deformation problem
\[ \mc{S}_2 =
(F_2/F_2^+,S_2,\wt{S}_2,\mc{O},\rhobar_2,\chi|_{G_{F_2^+}},\{
R^\square_{\rbar_2|_{G_{F_{2,v}}}}\}). \]
Let $\rho_2^{\univ}: G_{F_2^+} \to \mc{G}_n(R^{\univ}_{\mc{S}_2})$
represent the universal deformation of type $\mc{S}_2$ and let
$(r_2^{\univ},\chi|_{G_{F_2^+}},\psi_2^{\univ})$ be the corresponding
triple. Define a pairing $\psi_1$ on $r_1:=\Ind_{G_{F_2}}^{G_{F_1^+}}
r^{\univ}_2$ by setting
\[ \psi_1(f,g)=\sum_{\sigma \in G_{F_2}\backslash G_{F_1^+}}
\chibar(\sigma)^{-1} \psibar_2(f(\sigma),g(c\sigma)) \] for all $f,g
\in r_1$.
This pairing  
is perfect and alternating and
satisfies
\[ \psi_1(r_1(\tau) f, r_1(\tau) g) =\chi(\tau)\psi_1(f,g)\]
for all $\tau \in G_{F_1^+}$ and $f,g \in
r_1$. By Lemma 2.1.2 of \cite{cht},
there is a continuous homomorphism $ \rho_1 : G_{F_1^+,S_1} \to \mc{G}_{2mn}(R^{\univ}_{\mc{S}_2})$
with
\begin{itemize}
\item $G_{F_1,S_1}=\rho_1^{-1}(\GL_n(R^{\univ}_{\mc{S}_2})\times \GL_1(R^{\univ}_{\mc{S}_2}))$;
\item $\rho_1|_{G_{F_1}}=(r_1|_{G_{F_1}},\chi|_{G_{F_1}})$;
\item $\rho_1(c)=(r_1(c)J^{-1},-\chi(c))j$, where $J \in
  M_{2mn}(R^{\univ}_{\mc{S}_2})$ is the matrix with $\psi_1(x,y)={}^txJy$;
\item $\nu \circ \rho_1 = \chi$.
\end{itemize}
Let $\rhobar_1 = \rho_1 \mod \mf{m}_{R^{\univ}_{\mc{S}_2}}  :G_{F_1^+}
\to \mc{G}_{2mn}(k)$. (Note that $\rhobar_1$ corresponds to the triple
$(\rbar_1|_{G_{F_1}},\chibar,\psibar_1')$ where
$\psibar_1'$ is the pairing associated to the matrix
$\overline{J}\rbar_1(c)$.) 

For each place $v \in S_1$, let 
$R^\square_{\rbar_1|_{G_{F_{1,\wt{v}}}}}$ denote the universal
$\mc{O}$-lifting ring of $\rbar_2|_{G_{F_{2,\wt{v}}}}$. Let
$\mc{S}_1$ denote the deformation problem
\[ \mc{S}_1 =
(F_1/F_1^+,S_1,\wt{S}_1,\mc{O},\rhobar_1,\{R^\square_{\rbar_1|_{G_{F_{1,v}}}}\}). \]
The lift $\rho_1$ of $\rhobar_1$ is of type $\mc{S}_1$, and hence
gives rise to a map $\iota: R^{\univ}_{\mc{S}_1} \to R^{\univ}_{\mc{S}_2}$
(note that $R^{\univ}_{\mc{S}_1}$ exists as $\rbar_1|_{G_{F_1}}$ is
absolutely irreducible).

\begin{lem}
  \label{lem: induction induces finite map on deformation rings}
Let 
\begin{itemize}
\item $\rbar_2 : G_{F_2,S_2} \to \GL_n(k)$,
\item $\rbar_1 := \Ind_{G_{F_2}}^{G_{F_1^+}}\rbar_2 : G_{F_1^+,S_1}
  \to \GL_{2mn}(k)$, and
\item $\chibar : G_{F_1^+} \to \mc{O}^\times$
\end{itemize}
be continuous homomorphisms satisfying assumptions (1)-(4) above. 
Then the homomorphism $\iota: R^{\univ}_{\mc{S}_1} \to R^{\univ}_{\mc{S}_2}$
constructed above is finite.
\end{lem}

\begin{proof}
  Let $\wp$ denote a prime ideal of
  $R^{\univ}_{\mc{S}_2}/\iota(\mf{m}_{R^{\univ}_{\mc{S}_1}})$. As in
  the proof of Lemma \ref{lem:deformation ring is finite over one obtained by tensor}, it suffices to
  show that $\Rbar :=
  (R^{\univ}_{\mc{S}_2}/\iota(\mf{m}_{R^{\univ}_{\mc{S}_1}}))/\wp$ is a
  finite $k$-algebra. Let $r_{2,\Rbar}=r_2^{\univ}
  \otimes_{R^{\univ}_{\mc{S}_2}} \Rbar : G_{F_2} \to
  \GL_n(\Rbar)$. As in the proof of Lemma \ref{lem:deformation ring
    is finite over one obtained by tensor} again, it suffices to show that
  $r_{2,\Rbar}$ has finite image. However,
  $(\Ind_{G_{F_2}}^{G_{F_1^+}}r_{2,\Rbar})|_{G_{F_1}}: G_{F_1} \to
  \GL_{2mn}(\Rbar)$ is equivalent to $\rbar_1|_{G_{F_1}}\otimes_k
  \Rbar$ and hence has finite image. It follows easily that
  $r_{2,\Rbar}$ has finite image.
\end{proof}

\subsection{Finiteness of a deformation ring}
Let $F$ be a CM field with maximal totally real subfield $F^+$. Let
$n$ be a positive integer and $l>n$ an odd prime such that $\zeta_l \not
\in F$. Let $S$ be a set of places of $F^+$ which split in $F$, and
assume that $S$ contains all places of $F^+$ lying over $l$. Let
$\tilde{S}$ be a set of places of $F$ consisting of exactly one place
lying over each place of $S$. If $v$ is a place in $S$ we write
$\tilde{v}$ for the unique place of $\tilde{S}$ lying over $v$. Fix a
choice of complex conjugation $c \in G_{F^+}$.

Let $K\subset \Qlbar$ be a finite extension of $\Ql$ with ring of
integers $\bigO$ and residue field $k$.  Assume that $K$ contains the
image of every embedding $F\into\Qlbar$.

Suppose that $\rbar:G_F\to\GL_n(k)$ is a continuous absolutely irreducible
representation which is unramified at all places not lying over a
place in $S$. Suppose that $\chi:G_{F^+}\to \bigO^\times$ is a continuous totally odd crystalline character, unramified away from $S$, such that
 \[\rbar^c\cong\rbar^\vee\chibar|_{G_F}.\]
Assume also that $\rbar$ is odd in the sense of section
\ref{subsubsec: oddness CM}. Thus, there exists a non-degenerate
symmetric bilinear pairing $\langle\ ,\ \rangle : k^n \times k^n \to
k$ such that $\langle \rbar(\sigma)x,\rbar(c\sigma c)y\rangle =
\chibar(\sigma)\langle x,y\rangle$ for all $\sigma \in G_F$, $x,y \in
k^n$. The triple $(\rbar,\chibar,\langle\ ,\
\rangle)$ then corresponds to a continuous homomorphism $\rhobar :
G_{F^+} \to \mc{G}_n(k)$ (see section \ref{subsubsec: curly G_n}).

For each place $v\in S$ not dividing $l$, assume that for every minimal prime ideal $\wp$ of $R^\square_{\rbar|_{G_{F_\tv}}}$, the quotient $R^\square_{\rbar|_{G_{F_\tv}}}/\wp$ is geometrically integral. Note that this can be achieved by replacing $K$ by a finite extension. Let $R_{\tilde{v}}$ be one such irreducible component of $R^\square_{\rbar|_{G_{F_\tv}}}$ which is of characteristic 0. For each place $v\in \tilde{S}$ dividing $l$, fix an element
$a_{\tv}\in(\Z^n_+)^{\Hom(F_{\tv},\Qlbar)}$ with corresponding
$l$-adic Hodge type $\mathbf{v}_{a_\tv}$ and assume that the ring $R^{\mathbf{v}_{a_\tv},cr}_{\rbar|_{G_{F_\tv}}}$ is non-zero and moreover that for every minimal prime $\wp$, the irreducible component $R^{\mathbf{v}_{a_\tv},cr}_{\rbar|_{G_{F_\tv}}}/\wp$ is geometrically integral. Let $R_{\tilde{v}}$ be one of these irreducible components.
Consider the deformation problem
\[ \mc{S}=(F/F^+,S,\tilde{S},\mc{O},\rhobar,\chi,\{R_{\tilde{v}}\}_{v\in
  S})\]
(see section \ref{subsubsec: global deformation rings} for this
terminology) and let $R_{\mc{S}}^{\univ}$ represent the corresponding
deformation functor. Note that the rings $R_{\tilde{v}}$ satisfy
property (*) of section \ref{subsubsec: global deformation rings} by
Lemma \ref{lem: property (*)}.

\begin{lem}
\label{lem:deformation ring is finite over Zl because Hecke}
Suppose that
\begin{enumerate}[(i)]
\item $\rbar|_{G_{F(\zeta_l)}}$ has big image.
\item $\bar{F}^{\ker\ad\rbar}$ does not contain $\zeta_l$
\end{enumerate}
Suppose that there is a continuous representation
$r':G_F\to\GL_n(\bigO)$ lifting $\rbar$ such that:
\begin{enumerate}[(1)]
\item There is an isomorphism $\iota: \Qlbar \to \bb C$, a RAECSDC
  automorphic representation $\pi'$ of $\GL_n(\bb{A}_F)$ and a continuous
  character $\mu' : \bb{A}_{F^+}^\times/(F^+)^\times$ such that
  \begin{itemize}
  \item $r_{l,\iota}(\pi')\cong r'\otimes_\mc{O} \Qlbar$
  \item $(\pi')^c \cong (\pi')^\vee \otimes (\mu' \circ
    \mathbf{N}_{F/F^+}\circ \det)$
  \item $\chi':=\epsilon^{1-n}r_{l,\iota}(\mu')$ is a lift of
    $\chibar$. (Note that $(r')^c \cong (r')^{\vee}\chi'|_{G_F}$.)
  \item $\pi'$ is unramified above $l$ and outside the set of places
    of $F$ lying above $S$.
  \end{itemize}
\item For each place $v\in S$, $r'|_{G_{F_\tv}}$ corresponds to a
  closed point of $R_\tv[1/l]$. Moreover, if $v\nmid l$ then
  $r'|_{G_{F_{\tilde{v}}}}$ does not give rise to a closed point on any
  other irreducible component of $R^{\square}_{\rbar|_{G_{F_{\tilde{v}}}}}[1/l]$.
\end{enumerate}
Then the universal deformation ring $R_{\mc{S}}^{\univ}$ defined above
is finite over $\bigO$. Moreover $R^{\univ}_{\mc{S}}[1/l]$ is non-zero
and any $\Qlbar$-point of this ring
gives rise to a representation of $G_F$ which is automorphic of level
prime to $l$.   
\end{lem}

\begin{proof}
When $\mu'$ is trivial, the result follows from
Proposition 3.6.3 of \cite{BLGG}. We essentially reduce to this case by a 
twisting argument:

{\sl Step 1: Twisting.} Choose a non-degenerate symmetric bilinear pairing $\langle\ ,\
\rangle : \mc{O}^n \times \mc{O}^n \to \mc{O}$ lifting the pairing
also denoted $\langle\ ,\ \rangle$ on $k^n$ and such that $\langle
r'(\sigma)x,r'(c\sigma c)y\rangle = \chi'(\sigma)\langle x ,y\rangle$
for all $\sigma \in G_F$ and $x,y \in \mc{O}^n$.

By Lemma 4.1.5 of \cite{cht}, extending $K$ if necessary, we can and
do choose a continuous algebraic character $\psi : G_F \to
\mc{O}^\times$ such that (i) $\psi \psi^c  =
(\chi^{-1}\epsilon^{1-n})|_{G_F}$, and (ii) $\psi$ is crystalline above
$l$. By Lemma 4.1.6 of
\cite{cht}, again extending $K$ if necessary, we can and do choose a
continuous algebraic character $\psi': G_F \to \mc{O}^\times$ such
that (i) $\psi'(\psi')^c = ((\chi')^{-1}\epsilon^{1-n})|_{G_F}$, (ii)
$\psi'$ is crystalline above $l$, (iii)
$\psi'|_{I_{F_{\tilde{v}}}}=\psi|_{I_{F_{\tv}}}$ for each 
$\tv\in\tilde{S}$, and (iv) $\psibar'=\psibar$.

Let $r'_{\psi'}= r'\otimes\psi'$.  Note that $\langle
r'_{\psi'}(\sigma)x,r'_{\psi'}(c\sigma c)y\rangle =
(\psi'(\psi')^c\chi')(\sigma)\langle x,y \rangle =
\epsilon^{1-n}(\sigma)\langle x,y\rangle$ for all $\sigma\in G_F$ and
$x,y \in \mc{O}^n$. Let $\rbar_{\psibar}=\rbar \otimes \psibar = \rbar'_{\psibar'}$. The triple
$(\rbar_{\psibar},\epsbar^{1-n}\delta_{F/F^+}^n,\langle\ ,\ \rangle)$
corresponds to a continuous homomorphism $\rhobar_{\psibar}: G_{F^+}
\to \mc{G}_{n}(k)$. The triple
$(r'_{\psi'},\eps^{1-n}\delta_{F/F^+}^n,\langle\ ,\ \rangle)$ gives
rise to a homomorphism $\rho'_{\psi'} : G_{F^+} \to  \mc{G}_n(\mc{O})$
lifting $\rhobar_{\psibar}$.

For each $v\in S$, twisting by $\psi|_{G_{F_{\tilde{v}}}}$ defines an isomorphism between
the lifting problems for $\rbar|_{G_{F_{\tilde{v}}}}$ and
$\rbar_{\psibar}|_{G_{F_{\tilde{v}}}}$. This gives
rise to isomorphisms
$R^{\square}_{\rbar_{\psibar}|_{G_{F_{\tilde{v}}}}} \isoto
R^{\square}_{\rbar|_{G_{F_{\tilde{v}}}}}$. For each $v\in S$, let $R_{\tilde{v},\psi}$ be
the irreducible quotient of
$R^{\square}_{\rbar_{\psibar}|_{G_{F_{\tilde{v}}}}}$ corresponding to
$R_{\tilde{v}}$ under this isomorphism. Note that
$r'_{\psi'}|_{G_{F_{\tv}}}$ gives rise to a closed point of
$\Spec R^{\square}_{\rbar_{\psibar}|_{G_{F_\tv}}}[1/l]$ contained in
$\Spec R_{\tv,\psi}[1/l]$ and no other irreducible component. (This
follows easily from the fact that $(\psi'/\psi)|_{G_{F_\tv}}$ is
unramified and has trivial reduction.)

Choose a finite solvable totally real extension $L^+/F^+$ such that,
if we set $L=FL^+$, then the following hold:
\begin{itemize}
\item $\psi|_{G_L}$ and $\psi'|_{G_L}$ are unramified at all places not lying above $S$.
\item Each place $v \in S$ splits completely in $L^+$.
\item $L$ is linearly disjoint from $\overline{F}^{\ker \rbar}(\zeta_l)$ over $F$.
\end{itemize}
Let $S_L$ and $\tilde{S}_L$ denote the set of places of $L^+$ and $L$
above $S$ and $\tilde{S}$ respectively. For each $w \in S_L$, let
$\tilde{w}$ denote the unique place of $\tilde{S}_L$ lying above
$w$. For $w \in S_L$ lying over $v \in S$, let $R_{\tilde{w},\psi}$
denote $R_{\tilde{v},\psi}$ considered as a quotient of
$R^{\square}_{\rbar_{\psibar}|_{G_{L_{\tilde{w}}}}}$.  Let
\[
\mc{S}_{L,\psibar}=(L/L^+,S_L,\tilde{S}_L,\mc{O},\rhobar_{\psibar}|_{G_{L^+}},\epsilon^{1-n}\delta_{L/L^+}^n,\{R_{\tilde{w},\psi}\}_{w\in
  S_L})\] and let $R^{\univ}_{\mc{S}_{L,\psibar}}$ represent the
corresponding deformation functor. For $w \in S_L$, let
$R_{\tilde{w}}$ denote $R_{\tilde{v}}$ regarded as a quotient of
$R^{\square}_{\rbar|_{G_{L_{\tilde{w}}}}}$. Let $\mc{S}_L$ denote the
restriction of the deformation problem $\mc{S}$ to $L/L^+$:
\[ \mc{S}_L = (L/L^+,S_L,\tilde{S}_L,\mc{O},\rhobar|_{G_{L^+}},\chi|_{G_{L^+}},\{R_{\tilde{w}}\}_{w\in S_L}).\]
Twisting by $\psi|_{G_L}$ defines an isomorphism between
the deformation problems associated to $\mc{S}_{L}$ and $\mc{S}_{L,\psibar}$
and hence we have an isomorphism $R^{\univ}_{\mc{S}_{L,\psibar}} \isoto
R^{\univ}_{\mc{S}_L}$.

 {\sl Step 2: Finishing off.} Proposition 3.6.3 of \cite{BLGG} implies that
$R^{\univ}_{\mc{S}_{L,\psibar}}$ is a finite $\mc{O}$-algebra and that
any $\Qlbar$-point of $R^{\univ}_{\mc{S}_{L,\psibar}}$ gives rise to a
representation which is automorphic of level prime to $l$. By Lemma \ref{lem:deformation ring is finite over
  one obtained by tensor}, the natural map
$R^{\univ}_{\mc{S}_{L}} \to R^{\univ}_{\mc{S}}$ (coming from
restriction to $G_{L^+}$) is finite.
It follows that
$R^{\univ}_{\mc{S}}$ is finite over $\mc{O}$ and that any
$\Qlbar$-point of $R^{\univ}_{\mc{S}}$ gives rise to a representation
$r: G_{F} \to \GL_n(\Zlbar)$ whose restriction to $G_L$ is automorphic
of level prime to $l$. Since $L/F$ is solvable, Lemma 1.4 of
\cite{BLGHT} implies that any such $r$ is automorphic. Moreover, since
each place of $F$ lying above $l$ splits completely in $L$, we see
that such an $r$ must be automorphic of level prime to $l$. Finally,
Lemma 3.2.4 of \cite{GG} shows that the Krull dimension of
$R^{\univ}_{\mc{S}}$ is at least 1. Since $R^{\univ}_{\mc{S}}$ is
finite over $\mc{O}$, we deduce that $R^{\univ}_{\mc{S}}[1/l] \neq 0$.
\end{proof}

\subsection{A lifting result}
\label{subsec: a lifting result}
Let $m$ and $n$ be positive integers and let $l>mn$ be an odd prime. Let $F$ be a CM field with
maximal totally real subfield $F^+$ 
Let $S_l$ be the set of places of $F^+$ which lie over $l$,
and assume that they all split in $F$. Let $\tilde{S}_l$ be a set of
places of $F$ consisting of exactly one place lying over each place of
$S_l$. Let $(F')^+/F^+$ be a finite  
extension of totally real
fields, let $F'=(F')^+F$, and let $S'_l$ (respectively $\tilde{S}'_l$)
be the set of places of $(F')^+$ (respectively $F'$) lying over places
in $S_l$ (respectively $\tilde{S}_l$). If $v$ is a place in $S_l$
(respectively $S'_l$) we write $\tilde{v}$ for the unique place of
$\tilde{S}_l$ (respectively $\tilde{S}'_l$) lying over $v$. Finally, assume
 that $\zeta_l \not \in F'$.

Let $K\subset \Qlbar$ be a
finite extension of $\Ql$ with ring of integers $\bigO$ and residue
field $k$. Assume
that $K$ contains the image of every embedding $F'\into\Qlbar$. 

Fix a continuous absolutely irreducible representation
$\rbar:G_F\to\GL_n(k)$.  For each place $v\in \tilde{S}_l$, fix an
element $a_{\tv}\in(\Z^n_+)^{\Hom(F_{\tv},\Qlbar)}$ with corresponding
$l$-adic Hodge type $\mathbf{v}_{a_\tv}$. Assume that the ring
$R^{\mathbf{v}_{a_\tv},cr}_{\rbar|_{G_{F_\tv}}}$ is non-zero and moreover that for every minimal prime $\wp$, the irreducible component $R^{\mathbf{v}_{a_\tv},cr}_{\rbar|_{G_{F_\tv}}}/\wp$ is geometrically integral (this can always be arranged by extending $K$).
Fix an irreducible component $R_{\tv}$ of
$R^{\mathbf{v}_{a_\tv},cr}_{\rbar|_{G_{F_\tv}}}$.

\begin{thm}
  \label{thm:existence of a lift with specified properties}
  Use the assumptions and notation established above. Suppose that
  there exist continuous crystalline totally odd characters $\chi :
  G_{F^+} \to \bigO^\times$, $\chi',\chi'' : G_{(F')^+} \to
  \bigO^\times$ and continuous representations $r':G_{F'} \to
  \GL_m(\bigO)$ and $r'' : G_{F'} \to \GL_{nm}(\bigO)$ such that:
  \begin{enumerate}
  \item $\rbar^c\cong\rbar^\vee\chibar|_{G_F}$ and $\rbar$ is odd.
  \item $\rbar,r',r'',\chi,\chi'$ and $\chi''$ are all unramified away
    from $l$.
  \item $(r')^c\cong (r')^\vee\chi'|_{G_{F'}}$.
  \item $r'$ is crystalline at all places of $F'$ lying over $l$.
  \item $r''$ is automorphic of level prime to $l$ and
    $(r'')^c\cong(r'')^\vee\chi''|_{G_{F'}}$. Note that $r''$ is
    necessarily crystalline at all places of $F'$ dividing
    $l$. 
  \item $\rbar''=\rbar|_{G_{F'}}\otimes_k \rbar'$ and
    $\chi''|_{G_{F'}}=\chi|_{G_{F'}}\chi'|_{G_{F'}}$.
  \item $\rbar''|_{G_{F'(\zeta_l)}}$ has big image.
  \item $\bar{F}^{\ker\ad\rbar''}$ does not contain $\zeta_l$.
  \item For all places $w\in S'_l$, lying over a place $v$ of $F^+$,
    if $\rho_v:G_{F_\tv}\to\GL_n(\bigO)$ is a lifting of
    $\rbar|_{G_{F_\tv}}$ corresponding to a closed point on $\Spec
    R_\tv[1/l]$, then $(\rho_v|_{G_{F'_{\tw}}})\otimes_{\mc{O}}
    r'|_{G_{F'_{\tw}}}\sim r''|_{G_{F'_\tw}}$ (see Definition \ref{defn: sim}, noting that we are using Definition \ref{defn:inherited basis for otimes} to regard the tensor product as a lift of $\rbar''$). Note that if this is
    true for one such $\rho_v$ then it is true for all choices of
    $\rho_v$.
  \end{enumerate}
  Then, after possibly extending $\mc{O}$, there is a continuous lifting $r:G_F\to\GL_n(\bigO)$ of $\rbar$
  such that:
  \begin{enumerate}[(a)]
  \item $r^c\cong r^\vee\chi|_{G_F}$.
  \item $r$ is unramified at all places of $F$ not lying over $l$.
  \item For each place $v\in\tilde{S}_l$, $r|_{G_{F_\tv}}$ corresponds
    to a closed point on $\Spec R_\tv[1/l]$.
  \item $r|_{G_{F'}}\otimes_{\mc{O}} r'$ is automorphic of level prime to $l$.
  \end{enumerate}
\end{thm}

\begin{proof}
  Choose a complex conjugation $c \in G_{(F')^+}$. We can choose a symmetric bilinear pairing $\langle\ ,\ \rangle :
  k^n \times k^n \to k$, unique up to scaling by elements of
  $k^\times$ such that $\langle \rbar(\sigma)x,\rbar(c\sigma c)\rangle
  = \chibar(\sigma)\langle x,y\rangle$ for all $\sigma \in G_F$, $x,y
  \in k^n$. The triple $(\rbar,\chibar,\langle\ ,\ \rangle)$ gives rise
  to a homomorphism $\rhobar : G_{F^+} \to \mc{G}_n(k)$. Let $\mc{S}$
  be the deformation problem
  \[  \mc{S}=(F/F^+,S_l,\tilde{S}_l,\mc{O},\rhobar,\chi,\{R_{\tv}\}_{v\in S_l})\]
  and let $R^{\univ}_{\mc{S}}$ be an object of $\mc{C}_{\mc{O}}$ representing the
  corresponding deformation functor.

  By a similar argument we can extend $r'$ to a continuous
  homomorphism $\rho':G_{(F')^+} \to \mc{G}_m(\mc{O})$ corresponding
  to a triple $(r',\chi',\langle\ ,\ \rangle')$. The pairings
  $\langle\ ,\ \rangle$ and $\langle\ ,\ \rangle'$ give rise to a
  non-degenerate symmetric pairing $\langle\ ,\ \rangle''$ on $k^{mn}$
  and the triple $(\rbar'',\chibar'',\langle\ ,\ \rangle'')$
  corresponds to a homomorphism $\rhobar'': G_{(F')^+} \to \mc{G}_{mn}(k)$.

  Let $a \in (\bb{Z}^n_+)^{\Hom(F',\bb{C})}_c$ be the weight of the
  automorphic representation $\pi$. For $w \in S'_l$, let $a_{\tw}$ be
  the element of $(\bb{Z}^n_+)^{\Hom(F'_{\tw},\Qlbar)}$ with
  $(a_{\tw})_{\tau}=a_{\iota\circ \tau|_{F'}}$ for each $\tau : F'_w
  \into \Qlbar$. Then $r''|_{G_{F'_{\tw}}}$ is crystalline of $l$-adic
  Hodge type $\mathbf{v}_{a_{\tw}}$.  Let $R_{\tw}$ be the irreducible
  component of $R^{\mathbf{v}_{a_{\tw}},cr}_{\rbar''|_{G_{F'_{\tw}}}}$
  determined by $r''|_{G_{F'_{\tw}}}$. Let
  \[ \mc{S}'' =(F'/(F')^+,S'_l,\tilde{S}'_l,\mc{O},\rhobar'',\chi'',\{R_{\tw}\}_{w\in S_l})\]
 and let $R^{\univ}_{\mc{S}''}$ be an object representing the
 deformation functor associated to $\mc{S}''$. 

 As in section \ref{subsec: relative finiteness for deformation rings
   of tensor products}, `restricting to $G_{(F')^+}$ and tensoring with $\rho'$' gives rise to a
 homomorphism $R^{\univ}_{\mc{S}''}\to R^{\univ}_{\mc{S}}$ which is
 finite by Lemma \ref{lem:refined deformation ring is finite over
  one obtained by tensor}. Lemma \ref{lem:deformation ring is finite over Zl because Hecke}
(applied to $\rbar''$ and
  $R^{\univ}_{\mc{S}''}$) implies that $R^{\univ}_{\mc{S}''}$ is
  finite over $\mc{O}$ and any $\Qlbar$-point of this ring gives rise
  to a lift $G_{F'} \to \GL_{mn}(\Zlbar)$ of $\rbar''$ which
  automorphic of level prime to $l$. It follows that
  $R^{\univ}_{\mc{S}}$ is finite over $\mc{O}$ and moreover that any
  $\Qlbar$-point of this ring gives rise to a lift $r : G_{F} \to
  \GL_n(\Zlbar)$ satisfying (a)-(d) of the statement of the
  theorem being proved. The existence of such a $\Qlbar$-point follows from
  Lemma 3.2.4 of \cite{GG} (which shows that $R^{\univ}_{\mc{S}}$ has
  Krull dimension at least 1) and the fact that $R^{\univ}_{\mc{S}}$
  is finite over $\mc{O}$.
\end{proof}

\section{Big image for $\GL_2$.}\label{sec:big image GL2}\subsection{}

In this section we elucidate the meaning of the condition ``2-big'' in
the special case of subgroups of $\GL_2(\Flbar)$ (see Definition
\ref{defn: m-big}).
We will give concrete
criteria under which the 2-big condition may be realised. We will prove
results both for subgroups containing $\SL_2(k')$ for some field, and
results for subgroups whose projective image is tetrahedral,
octahedral or icosahedral.

Before we do so, let us first state a few results which will be useful to us in the
sequel. These results are due to Snowden and Wiles for bigness (see 
\cite[\S2]{snowden2009bigness}), and the generalization to $m$-bigness are due to White
(see \cite[\S2, Propositions 2.1 and 2.2]{pjw}). 
\begin{prop}[Snowden-Wiles,White] \label{normal subgroups and bigness}
Let $k/\F_l$ be algebraic, and $m$ be an integer. 
If $G\subset\GL_2(k)$ has no $l$-power order quotients, and has a normal
subgroup $H$ which is $m$-big, then $G$ is $m$-big.
\end{prop}

\begin{prop}[Snowden-Wiles,White] \label{bigness and scalars}
Let $k/\F_l$ be algebraic, $m$ be an integer, and $G\subset\GL_2(k)$. Then $G$ is
$m$-big if and only if $k^\times G$ is.
\end{prop}

\subsection{Subgroups containing $\SL_2(k')$.}
\begin{lem}
  \label{lem:SL2 and contained in SL2 implies 2-big}Let $l>3$ be prime, and let $k$ be an
  algebraic extension of $\Fl$. Let $k'$ be a finite subfield of $k$,
  and suppose that $H$ is a
  subgroup of $\GL_2(k)$ with
  \[\SL_2(k')\subset H\subset k^\times\GL_2(k').\]If the cardinality
  of $k'$ is greater than 5 then $H$ is 2-big.
\end{lem}
\begin{proof}All but the last condition in the definition follow at
  once from Lemma 2.5.6 of \cite{cht}. To check the last condition,
  take $h=\diag(\alpha,\alpha^{-1})$, where $\alpha\in k'$ satisfies
  $\alpha^4\ne 1$ (such an $\alpha$ exists because the cardinality of
  $k'$ is greater than 5). Then the roots of the characteristic
  polynomial of $h$ are $\alpha$ and $\alpha^{-1}$, and $\alpha^2\ne
  \alpha^{-2}$. Furthermore $\pi_{h,\alpha}\circ W\circ
  i_{h,\alpha}\ne 0$ for any irreducible $k[H]$-submodule $W$ of
  $\mathfrak{gl}_2(k)$, just as in the proof of Lemma 2.5.6 of
  \cite{cht}.\end{proof}

\begin{lem}\label{lem:Containing SL2 implies 2-big}Let $l>3$ be prime, and let $k$ be an
  algebraic extension of $\Fl$. Let $k'$ be a finite subfield of $k$,
  and suppose that $H$ is a finite
  subgroup of $\GL_2(k)$ with
  \[\SL_2(k')\subset H.\]If the cardinality
  of $k'$ is greater than 5 then $H$ is 2-big.
  \end{lem}
\begin{proof}
  Let $\bar{H}$ be the image of $H$ in $\PGL_2(k)$. By Theorem 2.47(b)
  of \cite{MR1605752}, $\bar{H}$ is either conjugate to a subgroup of
  the upper triangular matrices, or is conjugate to $\PSL_2(k'')$ or
  $\PGL_2(k'')$ for some finite extension $k''$ of $k$, or is
  isomorphic to $A_4$, $S_4$, $A_5$, or a dihedral group. Since
  $\PSL_2(k')$ is a simple group, and $k'$ has cardinality greater
  than 5, we see that $\bar{H}$ must be conjugate to $\PSL_2(k'')$ or
  $\PGL_2(k'')$ for some finite extension $k''$ of
  $k'$. Then \[\SL_2(k'')\subset H\subset k^\times\GL_2(k''),\] and
  the result follows from Lemma \ref{lem:SL2 and contained in SL2 implies 2-big}.
\end{proof}
We will be able to give a strengthening of this result once we have considered
the icosahedral, octahedral and tetrahedral cases in the next subsection.

\subsection{Icosahedral, octahedral and tetrahedral cases.}

  In this section, let
us fix at the outset a subgroup $H\subset \GL_2(k)$ whose image $Hk^\times/k^\times$
in $\PGL_2(k)$
is isomorphic to $A_5$, $S_4$ or $A_4$. Let us also suppose
$l>5$. Under the assumption that $k$ contains a primitive cube root of
unity, we will show that $H$ is 2-big. Since these groups clearly have no $l$-power
order quotients, to show $H$ is 2-big, it will suffice (by proposition
\ref{normal subgroups and bigness}) to show that the commutator subgroup $[H,H]$
is 2-big, since this is a normal subgroup of $H$. If $Hk^\times/k^\times=S_4$, the image of this commutator
subgroup in $\PGL_2(k)$ will be isomorphic to $[S_4,S_4]=A_4$. Thus, by replacing 
$H$ with $[H,H]$ in this case, we may suppress the 
possibility that $Hk^\times/k^\times\cong S_4$. Thus $H/Z(H)=A_n$, for $n=4$ or $5$.

Now, since both $Hk^\times/k^\times$ and $k^\times\cap H$ have order prime to $l$
(in the latter case, as $k^\times$ does), $H$ itself has order prime to $l$. Thus, if we
think of the inclusion $H\into \GL_2(k)$ as a representation of $H$, this representation
will have a lift to characteristic zero, say $r_l:H\into \GL_2(\Zlbar)$. Given an element 
$g\in A_n=H/Z(H)$, we can lift to an element $\tilde{g}\in H$, and map this to
an element $r_l(\tilde{g})$. Making a series of such choices (and choosing $1\in H$
as the lift of $1\in H/Z(H)$), we get a map $P:A_n\to\GL_2(\Zlbar)$ sending $g\mapsto
r_l(\tilde{g})$. This will be a projective representation in the sense of the (unnumbered)
definition given at the beginning of chapter 1 of \cite{hoffman1992projective}.

Recall that $n=4$ or $5$.
By the last sentence of chapter 2 of \cite{hoffman1992projective} (on p23, just after
the unnumbered remark after Theorem 2.12) we see the construction of a group,
called there $\tilde{A}_n$, which is a representation group for
$A_n$. 
(A \emph{representation group} of $G$, in the terminology
of \cite{hoffman1992projective}, is a stem extension $G^*$ of $G$ such that the 
newly adjoined central elements are isomorphic to the Schur multiplier $M(G)$
of $G$.) This is defined as a certain subgroup of a certain group $\tilde{S}_n$,
which is given a presentation just before Theorem 2.8 of \emph{loc.~cit.}, on p18.
Comparing this presentation to the discussion in \S2.7.2 of \cite{wilsonfinitesimple},
we see that $\tilde{A}_n$ is the same group as the group called $2.A_n$ in \cite{wilsonfinitesimple}
(since this seems to be the more standard name for this group, we shall call this
group $2.A_n$ from now on). Examining the discussion in \S5.6.8 and \S5.6.2
of \cite{wilsonfinitesimple}, we see that $2.A_5$ and $2.A_4$ are respectively the 
binary icosahedral and tetrahedral groups.
(Thus if we consider $A_5$ as the group of symmetries of a 
icosahedron, a subgroup of $\SO(3)$, then $2.A_5$ is the inverse image of $A_5$ under
the natural 2 to 1 map $\SU(2)\to\SO(3)$; and similarly for $A_4$ and the group of
symmetries of the tetrahedron.) 

Following the discussion at the top of p7 of \cite{hoffman1992projective}, we choose a 
map $r:A_n\to 2.A_n$ (\emph{not} a homomorphism) sending each element of $A_n$
to a lift in the representation group, and sending $1$ to $1$. 

Then by Theorem 1.3 of \cite{hoffman1992projective}, applied to the projective representation
$P$ above, we can find a representation
$$R: 2.A_n \to \GL_2(\Zlbar)$$
such that for each $a\in A_n$, we have $R(r(a))=P(a)\beta(a)$, where 
$\beta(a)\in \Zlbar$ is a scalar. 
(The statement of Theorem 1.3 is for representations over $\C$, but the
proof generalizes to our case.)
We will write $\bar{R}$ for the reduction of $R$ mod $l$.
We then claim that $k^\times H=k^\times \bar{R}(2.A_n)$. To see this, note
that if $h\in H$, then 
$h$ mod $Z(H)$ will be some element $x$ of $A_n$; then we have 
$h=\alpha \bar{P}(x) = \alpha \beta \bar{R}(r(x))$,
where $\alpha$ and $\beta$ are scalars and $\bar{P}(a)$ is the reduction of $P(a)$ mod $l$.
Thus the image of $\bar{R}(2.A_n)$ in $\PGL_2(k^\times)$ is $A_n$.

Now, note
also that by Proposition \ref{bigness and scalars}, we have that 
$$H \text{ 2-big}\Leftrightarrow k^\times H \text{ 2-big}\Leftrightarrow k^\times \bar{R}(2.A_n)
\text{ 2-big}\Leftrightarrow \bar{R}(2.A_n)\text{ 2-big}$$
and we may therefore replace $H$ by $\bar{R}(2.A_n)$, and therefore 
assume that $H$ is of the form $\bar{R}(2.A_n)$ 
where $\bar{R}$ is some
characteristic $l$ representation of $2.A_n$, whose kernel is contained
in the center of $2.A_n$ (that is, the representation is `faithful on the
$A_n$ quotient'; recall that $2.A_n/Z(2.A_n)=A_n$).
Indeed, $\bar{R}$ must either be faithful or factor through $A_n$, and we can think
of $\bar{R}$ as a faithful mod $l$ representation \emph{either} of
$A_n$ \emph{or} $2.A_n$.
Recall that $n=4$ or $5$.

Since $\bar{R}$ comes from characteristic zero, it
will be a direct sum of irreducible representations. Since 
$A_4$, $A_5$, $2.A_4$
and $2.A_5$ are non-abelian, no direct sum of one dimensional
representations can ever be faithful,
so we can in fact see that $\bar{R}$ is a a faithful irreducible 2 dimensional
representation of $A_4$, $A_5$, $2.A_4$
or $2.A_5$.
It will thus suffice to show that  the images of all such
representations are $2$-big. 

Thus, in order to establish the following:
\begin{prop} \label{prop: big image, hedral cases}
Suppose $l>5$ is a prime, and $k$ is an algebraic extension
of $\F_l$ which contains a primitive cube root of 1.
Suppose $H$ is a subgroup of $\GL_2(k)$, and the image of $H$ in 
$\PGL_2(k)$ is isomorphic to $A_5$, $S_4$ or $A_4$. Then $H$ is 2-big.
\end{prop}
It will suffice to establish the following:
\begin{prop} \label{prop: big image, hedral cases, simplified}
Suppose $l>5$ is a prime, and $k$ is an algebraic extension
of $\F_l$ which contains a primitive cube root of 1.
Suppose $G$ is one of $A_5$, $2.A_5$,
$A_4$ or $2.A_4$. Suppose $r:G\to\GL_2(k)$ is a faithful 
irreducible representation. Then $r(G)$ is 2-big.
\end{prop}

Before we do so, the following lemma will be useful.
\begin{lem} \label{big image convenient lemma}
Suppose $G$ is a group, and $l$ is a prime such that that $\hcf(l,|G|)=1$.
Suppose that $r:G\to\GL_2(k)$ is a faithful irreducible representation. 
Suppose further that the following hold.
\begin{itemize}
\item $\ad^0 r$ is irreducible.
\item There exists an odd prime $p$ such that:
\begin{itemize} \item $G$ has some non-central element of order $p$, and
\item $k$ has a primitive $p$-th root of unity.
\end{itemize} 
\end{itemize}
Then $r(G)$ is 2-big.
\end{lem}
\begin{proof}
  The fact that $\hcf(l,|G|)=1$ implies that $r(G)$ has no quotients
  of $l$-power order and also that $H^1(r(G),\mathfrak{sl}_2(k))=(0)$ (the latter
  claim follows immediately from Corollary 1 of section VIII.2 of
  \cite{MR554237}).
Since $\ad^0 r$ is irreducible and has dimension 3, it cannot contain any copy
of the trivial representation and hence $H^0(G,\ad^0 r)=(0)$; that is, 
$H^0(r(G),\mathfrak{sl}_2(k))=(0)$.

Finally, let $g$ be a non-central element of $G$ of order $p$. Note that $r(g)$ has order $p$
and must be non-central (since $r$ is faithful). If the two roots of the characteristic polynomial
were equal, then $r(g)$ would be a scalar matrix and hence central. ($r(g)$ must be
diagonalizable, as we assume that $\hcf(l,|G|)=1$ and 
non-semisimple elements of $\GL_2(k)$ have order
divisible by $l$.)
Thus the roots of the characteristic polynomial of $r(g)$ are distinct $p$-th roots of unity,
$\alpha$, and $\alpha'$, with at least one of them 
($\alpha$, say) primitive. 
Since $p$ is odd, we have that $\alpha^2\ne(\alpha')^{2}$.
Since $k$ contains 
the $p$-th roots of unity, $\alpha\in k$.

Now using the fact that $\ad^0 r$ is irreducible again, we see that the only irreducible 
$k[r(G)]$ submodules of $\gl_2(k)$ are $\mathfrak{sl}_2(k)$ and $\langle I\rangle$. Checking that
$\pi_{r(g),\alpha}\circ W\circ i_{r(g),\alpha}\neq 0$ for each of these is trivial.
\end{proof}

\begin{proof}[Proof of Proposition \ref{prop: big image, hedral cases, simplified}]
Since $l>5$, $\hcf(|r(G)|,l)=1$. Then given Lemma \ref{big image convenient lemma}, using 
our assumption that $k$ has a 3rd root of unity, and noting that $A_4$ and $A_5$
have elements of order 3, we see that it suffices to check that $\ad^0 r$ is irreducible.

Since $\hcf(|r(G)|,l)=1$, $r$ must be the reduction of a characteristic zero representation,
and  we can replace $r$
with a characteristic zero lift, and prove that $\ad^0 r$ is irreducible for this new $r$.

Let us first deal with the case $G=2.A_5$.
The character table of $2.A_5$ can be found on page 5 of
\cite{conway-atlas},
and we can immediately see that there are precisely two irreducible representations
of dimension 2, corresponding to the characters $\chi_6$ and $\chi_7$ there.
(The first of these corresponds to $\rho_{\text{nat},2.A_5}$, the natural representation we get
by thinking of $2.A_5$ as the binary icosahedral group, and hence a subgroup of $\SU(2)$;
the second is $\rho^{(12)}_{\text{nat},2.A_5}$)
Since the character $\chi_6$ is real, the dual representation of 
$\rho_{\text{nat}}$ has the same character
and $\ad^0\rho_{\text{nat},2.A_5}$ has character $\chi_6^2-1$.
We recognize this character as $\chi_2$ from the table. Thus $\ad^0\rho_{\text{nat}}$
is irreducible. Similarly $\ad^0\rho_{\text{nat},2.A_5}^{(12)}$ has character $\chi_3$, which
is irreducible. 

We can also deal with the case $G=A_5$; since $A_5$ is a quotient of
$2.A_5$,
every representation of 
$A_5$ will occur in the character table for $2.A_5$; but in fact the only two 
dimensional representations we saw do not factor through the center of $2.A_5$.

The character table
of  $2.A_4$ is standard, but
since we have been unable to locate a convenient published reference 
for them, we reproduce it here.
(See Figure \ref{bin a4 char table}.  In 
calculating this table, James Montaldi's web page, which provides
a real character table for $2.A_4$, was very helpful.)
Here are some notes on the construction of the table, which should provide enough
detail that the reader can straightforwardly check its accuracy:
\begin{itemize}
\item We have labelled conjugacy classes by thinking of $2.A_4$ as
the binary tetrahedral group, and hence as a
subgroup of $\SU(2)$; we have written elements of $\SU(2)$ as unit quaternions.
\item $1$ stands for the trivial representation.
\item $\omega$ stands for the character $2.A_4\to A_4 \to A_4/K_4\isoto \Z/3\Z\to \langle e^{2\pi/3}\rangle$.
\item $\rho_{2,2.A_4}$ stands for the natural 2D representation we get by thinking of 
$2.A_4$ as a subgroup of $\SU(2)$.
\item $\rho_{3,2.A_4}$ stands for the natural 3D representation we get by thinking of 
$2.A_4$ as a subgroup of $\SU(2)$ then mapping $\SU(2)/\{\pm1\}\to\SO(3)$.
\end{itemize}

\begin{figure}
\begin{tabular}{l|ccccccc}
\emph{class} &$e$ & [-1] &$[i]$ & $[\frac{1+i+j+k}{2}]$ &$[\frac{1+i+j-k}{2}]$  &$[\frac{-1+i+j+k}{2}]$ &$[\frac{-1+i+j-k}{2}]$  \\  
\emph{size} & 1 & 1 & 6 &4&4&4&4 \\
\hline
1 					& 1 & 1 & 1 & 1 & 1&1&1		\\
$\omega$				& 1 & 1 & 1 & $e^{2\pi/3}$& $e^{-2\pi/3}$& $e^{-2\pi/3}$& $e^{2\pi/3}$\\
$\omega^{\otimes2}$	& 1 & 1 & 1 & $e^{-2\pi/3}$& $e^{2\pi/3}$& $e^{2\pi/3}$& $e^{-2\pi/3}$\\
$\rho_{2,2.A_4}$				& 2 & -2 & 0 & 1& 1& -1& -1\\
$\rho_{2,2.A_4}\otimes\omega$	& 2 & -2 & 0 &$e^{2\pi/3}$& $e^{-2\pi/3}$& $-e^{-2\pi/3}$& $-e^{2\pi/3}$\\
$\rho_{2,2.A_4}\otimes\omega^{\otimes2}$	& 2 & -2 & 0 &$e^{-2\pi/3}$& $e^{2\pi/3}$& $-e^{2\pi/3}$& $-e^{-2\pi/3}$\\
$\rho_{3,2.A_4}$				& 3 & 3 & -1 &0&0&0&0\\
\end{tabular}
\begin{caption}{Character table for $2.A_4$. \label{bin a4 char table}}
\end{caption}
\end{figure}

We note that $A_4$ has no irreducible 2D representations (all 2D irreducible representations
of $2.A_4$ fail to factor through $A_4$) and the only irreducible 2D representations
of $2.A_4$ are $\rho_{2,2.A_4}$, $\rho_{2,2.A_4}\otimes \omega$ and $\rho_{2,2.A_4}\otimes \omega^{\otimes2}$;
tensoring any of these with its dual gives $\rho_{2,2.A_4}^{\otimes 2}$, which is $1\oplus\rho_{3,2.A_4}$.
Thus, for each of these representations, the $\ad^0$ is $\rho_{3,2.A_4}$ and hence irreducible.
\end{proof}

\subsection{A synoptic result} 

We now combine the results of the previous two sections to give a somewhat
explicit characterization of big subgroups of $\GL_2(k)$. Our first result 
is an extension of Lemma \ref{lem:Containing SL2 implies 2-big}.

\begin{lem}\label{lem:a bigness alternative}
  Let $l>5$ be prime, and let $k$ be an
  algebraic extension of $\Fl$. Suppose that $k$ contains a primitive
  cube root of 1.
  If $H$ is a finite subgroup of $\GL_2(k)$, acting irreducibly on $k^2$,
  then we have the following alternative.
  Either:
  \begin{enumerate}
  \item the image of $H$ in $\PGL_2(k)$ is a dihedral group, or
  \item $H$ is 2-big.
  \end{enumerate}
  \end{lem}
\begin{proof}
  Let $\bar{H}$ be the image of $H$ in $\PGL_2(k)$. By Theorem 2.47(b)
  of \cite{MR1605752}, $\bar{H}$ is either conjugate to a subgroup of
  the upper triangular matrices, or is conjugate to $\PSL_2(k'')$ or
  $\PGL_2(k'')$ for some finite extension $k''$ of $k$, or is
  isomorphic to $A_4$, $S_4$, $A_5$, or a dihedral group. 
  We saw in the proof of Lemma \ref{lem:Containing SL2 implies 2-big} that we are
  done in the case $\bar{H}$ is conjugate to 
  $\PSL_2(k'')$ or
  $\PGL_2(k'')$ for some finite extension $k''$ of $k$, since then
  $H$ is 2-big.
  The possibility that  $\bar{H}$ is conjugate to a subgroup of
  the upper triangular matrices is excluded by our assumption that
  $H$ acts irreducibly on $k^2$. If $\bar{H}$ is dihedral then we 
  are certainly done; thus we may assume $\bar{H}$  is
  isomorphic to $A_4$, $S_4$, $A_5$, whence we are done
  by Proposition \ref{prop: big image, hedral cases}.
\end{proof}

We also prove a modified version of this result tailored for working
with Galois representations.

\begin{lem}\label{lem:a bigness alternative II}Let $l>5$ be prime, and let $k$ be an
  algebraic extension of $\Fl$. Suppose that $k$ contains a primitive
  cube root of 1.
  Suppose $K$ is a number field, and $r:G_K\to\GL_2(k)$ a continuous
  absolutely irreducible representation. Then we have the following alternative.
  Either:
  \begin{enumerate}
  \item $r(G_{K(\zeta_l)})$ is 2-big, or
  \item there are field extensions $K_2/K_1/K$, with
  \begin{itemize}
  \item $K_1/K$ either trivial or cubic (and hence cyclic) Galois, and
  \item $K_2/K_1$ quadratic
  \end{itemize}
   and, after possibly replacing $k$ by its quadratic extension, a continuous character $\bar{\theta}: K_2\to k^\times$ such 
   that $r|_{G_{K_1}} \cong \Ind^{G_{K_1}}_{G_{K_2}}\bar{\theta}$.
  \end{enumerate}
  \end{lem}
\begin{proof}
  We begin by claiming that in any case where we can construct an
  extension $K_1/K$, either trivial or cubic, such that the image $I$ of
  $r(G_{K_1})$ in $\PGL_2(k)$ is isomorphic to a dihedral group $D$,
  then we are done. For then, writing $R\triangleleft D$ for the rotations
  in $D$, we have:
    $$G_{K_1}\overset{r}{\to} r(G_{K_1}) \onto I \isoto D \onto D/R
    \isoto \langle-1\rangle$$ is a nontrivial quadratic character of
    $K_1$, defining a field extension $K_2/K_1$.  Then $r(G_{K_2})$
    has cyclic image in $\PGL_2(k)$ (as $R$ is cyclic), and hence
    $r(G_{K_2})\subset\GL_2(k)$ is abelian (as any central extension
    of a cyclic group is abelian). It follows, after possibly
    replacing $k$ by its quadratic extension, that $r|_{G_{K_2}}$ is
    reducible. If $r|_{G_{K_2}}$ is indecomposable, we see easily that
    $r$ is reducible, contradicting our assumptions. Similarly, if
    $r|_{G_{K_2}}$ decomposes as a direct sum of characters and this
    decomposition is preserved under conjugation by a non-trivial
    element of $\Gal(K_2/K_1)$, then using the fact that the extension
    $K_1/K$ is either trivial or cubic, we see that $r$ is reducible.
    Thus $r|_{G_{K_2}}$ is a direct sum of characters and conjugation
    by a nontrivial element of $\Gal(K_2/K_1)$ swaps these characters.
    Thus, in this case,
     $$r|_{G_{K_1}} \cong \Ind^{G_{K_1}}_{G_{K_2}}\bar{\theta}$$
     and we have the second alternative in the statement of the
     theorem.

  Now, let $\bar{H}$ be the image of $r(G_K)$ in $\PGL_2(k)$. 
  Let $\bar{H}'$ be the image of $r(G_{K(\zeta_l)})$ in $\PGL_2(k)$. 
  Since $r(G_{K(\zeta_l)})\triangleleft G_K$, with cyclic quotient, we have that 
  $\bar{H} \triangleleft \bar{H}'$ with cyclic quotient.
  Again, by Theorem 2.47(b)
  of \cite{MR1605752}, we have the following possibilities for $\bar{H}$:
  \begin{itemize}
  \item {\sl $\bar{H}$ is conjugate to a subgroup of the upper triangular matrices.}
     This is ruled out by our assumption that $r$ is irreducible.
  \item {\sl $\bar{H}$ is isomorphic to a dihedral group.} In this case, we may take 
     $K_1=K$ and we are done by the first paragraph of the proof.
  \item {\sl $\bar{H}$ is conjugate to $\PSL_2(k'')$ or $\PGL_2(k'')$ for some
     finite extension $k''$ of $k$.} In this case, since $\bar{H}'$ is a normal
     subgroup of $\bar{H}$, then the simplicity of $\PSL_2(k'')$ tells us that
     $\bar{H}'$ is conjugate to $\PSL_2(k'')$ or $\PGL_2(k'')$ too. We see, as
     in the proof of Lemma \ref{lem:Containing SL2 implies 2-big} (with our
     $r(G_{K(\zeta_l)})$ and $\bar{H}'$ taking the roles of $H$ and $\bar{H}$
     in the proof of Lemma \ref{lem:Containing SL2 implies 2-big})
     that this means that $r(G_{K(\zeta_l)})$ is 2-big, and
     we are done.
  \item {\sl $\bar{H}$ is conjugate to $A_4$, $S_4$, or $A_5$.} In this case 
     the image of $\bar{H}'$ under the isomorphism must be a normal subgroup
     with cyclic quotient; that is, one of $K_4$, $A_4$, $S_4$, $A_5$. In any
     of the cases except for $\bar{H}'\cong K_4$, we then conclude via 
     Proposition \ref{prop: big image, hedral cases} that 
     $r(G_{K(\zeta_l)})$ is 2-big. In the remaining case, we have $i:\bar{H}\isoto A_4$
     and $i(\bar{H}')=K_4\subset A_4$. Then
     $$G_K \overset{r}{\to} r(G_K) \onto \bar{H} \isoto A_4 \onto A_4/K_4 = \Z/3\Z\isoto\langle e^{2\pi i/3}\rangle$$
     is a cubic character of $G_K$, corresponding to a cubic Galois extension $K_1$. Moreover
     the image of $r|_{G_{K_1}}$ in $\PGL_2(k)$ is isomorphic to $K_4$, which is isomorphic
     in turn to the dihedral group with 4 elements. Thus we are done by the first 
     paragraph of the proof.
  \end{itemize}
  Since we are done in every case, the lemma is proved.
\end{proof}

\section{Untwisting: the CM case.}\label{sec:untwisting CM}
\subsection{}We now prove a slight variant of Proposition 5.2.1 of
\cite{BLGG}, working over a CM base field, rather than a totally real
base field. The proof is extremely similar.

Let $F$ be a CM field with maximal totally real subfield
$F^+$. Let $M$ be a cyclic CM extension of $F^+$ of degree $m$, linearly disjoint
from $F$ over $F^+$. Suppose that $\theta :
M^\times\backslash\A_M^\times \to \bb{C}^\times$ is an algebraic
character, and that $\Pi$ is a
RAECSDC automorphic representation of $\GL_{mn}(\A_F)$ for some
$n$. Let $\iota:\Qlbar\isoto\C$ be an isomorphism.
\begin{prop}\label{prop:untwisting CM}
  Assume that there is a continuous representation
  $r:G_{F}\to\GL_n(\Qlbar)$ such that $r|_{G_{FM}}$ is irreducible
  and \[r_{l,\iota}(\Pi)\cong r\otimes
  (\Ind_{G_M}^{G_{F^+}}r_{l,\iota}(\theta))|_{G_F}.\] Then $r$ is
  automorphic. If we assume furthermore that $\Pi$ has level prime to
  $l$, then $r$ is automorphic of level prime to $l$.
\end{prop}
\begin{proof}Note that \[
  (\Ind_{G_M}^{G_{F^+}}r_{l,\iota}(\theta))|_{G_F}\cong  \Ind_{G_{FM}}^{G_{F}}(r_{l,\iota}(\theta))|_{G_{FM}}).\]
  Let $\sigma$ denote a generator of $\Gal(FM/F)$, and $\kappa$ a
  generator of $\Gal(FM/F)^\vee$. Then we have
  \begin{align*}
    r_{l,\iota}(\Pi\otimes(\kappa\circ\Art_F\circ\det))&=
    r_{l,\iota}(\Pi)\otimes r_{l,\iota}(\kappa\circ\Art_F)\\ &\cong
    r\otimes(
    r_{l,\iota}(\kappa\circ\Art_F)\otimes\Ind_{G_{FM}}^{G_F}(r_{l,\iota}(\theta)|_{G_{FM}})
    \\ & \cong
    r\otimes\Ind_{G_{FM}}^{G_F}(r_{l,\iota}(\kappa\circ\Art_F)|_{G_{FM}}\otimes
    r_{l,\iota}(\theta)|_{G_{FM}})\\ & \cong
    r\otimes\Ind_{G_{FM}}^{G_F}r_{l,\iota}(\theta)|_{G_{FM}}\\ &\cong r_{l,\iota}(\Pi),
  \end{align*}  so that
  $\Pi\otimes(\kappa\circ\Art_F\circ\det)\cong\Pi$. 

We claim that for each intermediate field ${FM}\supset N\supset F$ there
is a regular cuspidal automorphic representation $\Pi_N$ of $\GL_{n[{FM}:N]}(\bb{A}_N)$
such that \[\Pi_N\otimes(\kappa\circ\Art_N\circ\det)\cong\Pi_N\] and
$BC_{N/F}(\Pi)$ is equivalent to
\[\Pi_N\boxplus\Pi_N^\sigma\boxplus\dots\boxplus\Pi_N^{\sigma^{[N:F]-1}}\]
in the sense that for all places $v$ of $N$, the base change from
$F_{v|_F}$ to $N_v$ of $\Pi_{v|_F}$
is \[\Pi_{N,v}\boxplus\Pi_{N,v}^\sigma\boxplus\dots\boxplus\Pi_{N,v}^{\sigma^{[N:F]-1}}.\]
We prove this claim by induction on $[N:F]$. Suppose that ${FM}\supset
M_2\supset M_1\supset F$ with $M_2/M_1$ cyclic of prime degree, and
that we have already proved the result for
$M_1$. Since  \[\Pi_{M_1}\otimes(\kappa\circ\Art_{M_1}\circ\det)\cong\Pi_{M_1}\]we
see from Theorems 3.4.2 and 3.5.1 of \cite{MR1007299} (together with Lemma VII.2.6 of \cite{MR1876802}
  and the main result of \cite{MR672475}) that there is a cuspidal
  automorphic representation  $\Pi_{M_2}$ of $\GL_{n[{FM}:{M_2}]}(\bb{A}_{M_2})$
such that $BC_{{M_2}/F}(\Pi)$ is equivalent to
\[\Pi_{M_2}\boxplus\Pi_{M_2}^\sigma\boxplus\dots\boxplus\Pi_{M_2}^{\sigma^{[{M_2}:F]-1}}.\]
Since $\Pi$ is regular, $\Pi_{M_2}$ is regular. The representation
$\Pi_{M_2}\otimes(\kappa\circ\Art_{M_2}\circ\det)$ satisfies the same
properties (because $\Pi\otimes(\kappa\circ\Art_F\circ\det)\cong\Pi$),
so we see (by strong multiplicity one for isobaric representations) that we must
have \[\Pi_{M_2}\otimes(\kappa\circ\Art_{M_2}\circ\det)\cong\Pi_{M_2}^{\sigma^i}\]
for some $0\le i\le [{FM}:M_2]-1$. If $i>0$
then \[\Pi_{M_2}\boxplus\Pi_{M_2}^\sigma\boxplus\dots\boxplus\Pi_{M_2}^{\sigma^{[{M_2}:F]-1}}\]cannot
be regular (note that of course $\kappa$ is a character of finite
order), a contradiction, so in fact we must have
$i=0$. Thus \[\Pi_{M_2}\otimes(\kappa\circ\Art_{M_2}\circ\det)\cong\Pi_{M_2}\]and
the claim follows.

Let $\pi:=\Pi_{FM}$. Note that the representations $\pi^{\sigma^i}$ for
$0\le i\le m-1$ are pairwise non-isomorphic (because $\Pi$ is
regular). Note also that $\pi\otimes |\det|^{(n-nm)/2}$ is regular
algebraic (again, because $\Pi$ is regular algebraic).

Since $\Pi$ is RAECSDC, there is an algebraic character $\chi$ of
$(F^+)^\times\backslash \A_{F^+}^\times$ such that $\Pi^{c,\vee}\cong
\Pi\otimes(\chi\circ N_{F/F^+}\circ\det)$. It follows (by strong multiplicity one for
isobaric representations) that for some $0\le i\le m-1$
we have \[\pi^{c,\vee}\cong \pi^{\sigma^i}\otimes(\chi\circ
N_{{FM}/F^+}\circ\det).\] Then we have
\begin{align*}
  \pi &\cong (\pi^{c,\vee})^{c,\vee}\\&\cong
  (\pi^{\sigma^i}\otimes(\chi\circ N_{FM/F^+}\circ\det))^{c,\vee}\\&\cong   (\pi^{c,\vee})^{\sigma^i}\otimes(\chi^{-1}\circ
N_{{FM}/F^+}\circ\det))\\ &\cong  (\pi^{\sigma^i}\otimes(\chi\circ
N_{{FM}/F^+}\circ\det))^{\sigma^i}\otimes(\chi^{-1}\circ
N_{{FM}/F^+}\circ\det)) \\ &\cong \pi^{\sigma^{2i}}
\end{align*}
so that either $i=0$ or $i=m/2$. We wish to rule out the latter
possibility. Assume for the sake of contradiction that \[\pi^{c,\vee}\cong \pi^{\sigma^{m/2}}\otimes(\chi\circ
N_{{FM}/F^+}\circ\det).\]  Since $F^+$ is totally real, there is an
integer $w$ such that $\chi|\cdot|^{-w}$ has finite image. Then
$\pi\otimes|\det|^{w/2}$ has unitary central character, so is itself
unitary. Since $\pi\otimes |\det|^{(n-nm)/2}$ is regular
algebraic, we see that for places $v|\infty$ of ${FM}$ the conditions of
Lemma 7.1 of \cite{BLGHT} are satisfied for $\pi_v|\det|_v^{w/2}$, so
that
\begin{align*}
  \pi_v\boxplus\pi_v^{\sigma^{m/2}}&\cong \pi_v\boxplus\pi_v^{c,\vee}\otimes(\chi^{-1}\circ
N_{{FM}/F^+}\circ\det)\\&\cong \pi_v\boxplus
  ((\pi_v\otimes|\det|^{w/2})^{c,\vee}\otimes(|\cdot|^{w/2}\circ\det))\otimes(\chi^{-1}\circ
N_{{FM}/F^+}\circ\det)\\
  &\cong  \pi_v\boxplus
  (\pi_v\otimes|\det|^{w/2})\otimes(|\cdot|^{w/2}\circ\det))\otimes(\chi^{-1}\circ
N_{{FM}/F^+}\circ\det) \\  &\cong  \pi_v\boxplus
  ( \pi_v\otimes(\chi^{-1}|\cdot|^{w}\circ
N_{{FM}/F^+}\circ\det))
\end{align*}
which contradicts the regularity of $\Pi_{v|_F}$. Thus we have $i=0$,
so that \[\pi^{c,\vee}\cong \pi\otimes(\chi\circ
N_{{FM}/F^+}\circ\det).\] Thus $\pi\otimes |\det|^{(n-nm)/2}$ is a RAECSDC
representation, so that we have a Galois representation
$r_{l,\iota}(\pi\otimes |\det|^{(n-nm)/2})$. The condition that $BC_{{FM}/F}(\Pi)$ is equivalent to
\[\pi\boxplus\pi^\sigma\boxplus\dots\boxplus\pi^{\sigma^{m-1}}\]
translates to the fact that
  \[r_{l,\iota}(\Pi)|_{G_{FM}}\cong
  r_{l,\iota}(\pi\otimes|\det|^{(n-nm)/2})\oplus\dots\oplus
  r_{l,\iota}(\pi\otimes|\det|^{(n-nm)/2})^{\sigma^{m-1}}.\] By
  hypothesis, we also have  \[r_{l,\iota}(\Pi)|_{G_{FM}} \cong
  (r|_{G_{FM}}\otimes
  r_{l,\iota}(\theta)|_{G_{FM}})\oplus\dots\oplus(r|_{G_{FM}}\otimes r_{l,\iota}(\theta)|_{G_{FM}}^{\sigma^{m-1}}).\]Since $r|_{G_{FM}}$ is irreducible, there must be an $i$ such
that \[r|_{G_{FM}}\cong r_{l,\iota}(\pi\otimes|\det|^{(n-nm)/2})\otimes
r_{l,\iota}(\theta)|_{G_{FM}}^{\sigma^{-i}},\] so that $r|_{G_{FM}}$ is
automorphic. The result now follows from Lemma 1.4 of \cite{BLGHT}.
\end{proof}

\section{Main Results.}\label{sec:main thm}\subsection{}
Let $l>2$ be a fixed prime. Fix an isomorphism
$\iota:\Qlbar\isoto\C$. In this section, $F^+$ always denotes a
totally real field.

We begin with some preliminary results.

\begin{lemma}\label{lem: existence of potentially non-ordinary lift to
    weight 0}
Let
$\rbar:G_{F^+}\to\GL_2(\Flbar)$ be an irreducible modular
representation. Suppose that $\overline{F^+}^{\ker \ad \rbar}$ does not contain $F^+(\zeta_l)$.  There exists a finite solvable extension of totally real fields
  $L^+/F^+$ such that
  \begin{itemize}
  \item  $L^+$ is linearly disjoint from $\overline{F^+}^{\ker\rbar}(\zeta_l)$
  over $F^+$.
\item $\rbar|_{G_{L_v^+}}$ is trivial for all places $v|l$ of $L^+$.
\item For each place $v|l$, $[L_v^+:\Ql]\ge 2$.
\item There is a regular algebraic cuspidal automorphic representation $\pi$ of
  $\GL_2(\A_{L^+})$ of weight 0 such that
  \begin{itemize}
  \item $\rbar_{l,\iota}(\pi)\cong\rbar|_{G_{L^+}}$.
  \item $\pi$ is unramified at all finite places.
  \item For all places $v|l$ of $L^+$,
    $r_{l,\iota}(\pi)|_{G_{L_v^+}}$ is non-ordinary.
  \end{itemize}
  \end{itemize}
\end{lemma}

\begin{proof}
  Choose a regular algebraic cuspidal automorphic representation $\pi$
  of $\GL_2(\bb{A}_{F^+})$ of weight 0 such that $\rbar_{l,\iota}(\pi)
  \cong \rbar$. Let $r = r_{l,\iota}(\pi):G_{F^+}\to
  \GL_n(\Qlbar)$. Choose $F_1^+/F^+$ a finite solvable extension of
  totally real fields such that:
\begin{itemize}
\item $F_1^+$ is linearly disjoint from $\overline{F^+}^{\ker\rbar}(\zeta_l)$
  over $F^+$.
\item $\rbar|_{G_{F_{1,v}^+}}$ is trivial for all places $v|l$ of $F_1^+$.
\item $\rbar|_{G_{F_{1,v}^+}}$ is unramified for all finite places $v$ of $F_1^+$.
\item $[F_1^+:\Q]$ is even.
\item The base change $\pi_{F_1^+}$ of $\pi$ to $F_1^+$ is unramified
  or Steinberg at each finite place.
\item If $\pi_{F_1^+}$ is ramified at a prime $v\nmid l$ of $F_1^+$, then $\mathbf{N}v \equiv 1 \mod l$.
\end{itemize}
Let $B$ be a
quaternion algebra with centre $F_1^+$ which is ramified at
precisely the infinite places. We now introduce $l$-adic automorphic
forms on $B^{\times}$. Let $K$ be a finite extension of $\Ql$ inside
$\Qlbar$ with ring of integers $\mc{O}$ and residue field $k$, and
assume that $K$ contains the images of all embeddings $F^+_1\into\Qlbar$.
Fix a maximal order $\mc{O}_{B}$ in $B$ and for each finite place $v$
of $F_1^+$ fix an isomorphism $i_v : \mc{O}_{B,v} \isoto
M_2(\mc{O}_{F^+_{1,v}})$. For each embedding $\tau:F_1^+\into \Qlbar$ we
let $\iota\tau$ denote the real place of $F_1^+$ corresponding to the
embedding $\iota \circ \tau$. Similarly, if $\sigma : F_1^+\into
\bb{R}$ is an embedding we let $\iota^{-1} \tau$ denote the
corresponding embedding $\iota^{-1}\circ\tau$.

  For each $v|l$ let $\tau_v$ denote a smooth
  representation of $\GL_2(\mc{O}_{F^+_{1,v}})$ on a finite free
  $\mc{O}$-module $W_{\tau_v}$. Let $\tau$ denote the representation $
  \otimes_{v|l}\tau_v$ of
  $\GL_2(\mc{O}_{F_1^+,l})$ on $ W_{\tau} :=
  \otimes_{v|l}W_{\tau_v}$. 
  Suppose that $\psi: (F_1^+)^{\times}\backslash
  (\bb{A}_{F_1^+}^{\infty})^{\times}\rightarrow \mc{O}^{\times}$ is a
  continuous character so that for each prime $v|l$, the action of the
  centre $\mc{O}_{F^+_{1,v}}^{\times}$ of $\mc{O}_{B,v}^{\times}$ on
  $W_{\tau_v}$ is given by
  $\psi^{-1}|_{\mc{O}_{F^+_{1,v}}^{\times}}$. Such a character $\psi$ is
  necessarily of finite order.
  
  Let $U=\prod_v U_v \subset (B \otimes_{\bb{Q}} \bb{A}^{\infty})^{\times}$ be a
  compact open subgroup with $U_v \subset \mc{O}_{B,v}^{\times}$ for
  all $v$ and $U_v = \mc{O}_{B,v}^{\times}$ for $v |l$. We let
  $S_{0,\tau,\psi}(U,\mc{O})$ denote the space of functions
  \[ f : B^{\times}\backslash
  (B\otimes_{\bb{Q}}\bb{A}^{\infty})^{\times} \rightarrow
  W_{\tau} \] with $f(gu)=\tau(u_l)^{-1}f(g)$
  and $f(gz)=\psi(z)f(g)$ for all $u \in U$, $z \in
  (\bb{A}_{F_1^+}^{\infty})^{\times}$ and $g \in
  (B\otimes_{\bb{Q}}\bb{A}^{\infty})^{\times}$.
  Writing $U=U^l\times U_l$, we let
  $$ S_{0,\tau,\psi}(U_l,\mc{O})=\varinjlim_{U^l}S_{0,\tau,\psi}(U^l
  \times U_l,\mc{O})$$
  and we let
  $(B\otimes_{\bb{Q}}\bb{A}^{l,\infty})^{\times}$ act on this space by
  right translation.

  Let $\psi_{\bb{C}}: (F_1^+)^{\times}\backslash \bb{A}_{F_1^+}^{\times}
  \rightarrow \bb{C}^{\times}$ be the algebraic Hecke character
  defined by $\psi_{\bb{C}}(z)=\iota(\psi(z^{\infty}))$. 
Let $W_{\tau,\bb{C}}=W_{\tau}\otimes_{\mc{O},\iota}\bb{C}$.
 We have an isomorphism of $(B\otimes_{\bb{Q}}\bb{A}^{l,\infty})^{\times}$-modules
\begin{equation}
\label{eq:iso}
 S_{0,\tau,\psi}(U_l,\mc{O})\otimes_{\mc{O},\iota}\bb{C}
\isoto
\bigoplus_{\Pi}
\Hom_{\mc{O}_{B,l}^{\times}}(W_{\tau,\bb{C}}^{\vee},\Pi_l)\otimes
\Pi^{\infty,l} 
\end{equation} 
where the sum is over all automorphic
representations $\Pi$ of $(B\otimes_{\bb{Q}}\bb{A})^{\times}$ of
weight $0$ and central character $\psi_{\bb{C}}$
(see for instance the proof of Lemma 1.3 of \cite{tay-fm2}).

Let $U$ be as above and let $R$ denote a finite set of places of
$F_1^+$ containing all those places $v$ where $U_v \neq
\mc{O}_{B,v}^{\times}$. Let $\bb{T}^{R}$ denote the polynomial algebra
$\mc{O}[T_{v},S_{v}]$ where $v$ runs over all places of $F_1^+$ away
from $l$ and $R$. For such $v$ we let $T_v$ and $S_v$ act on
$S_{0,\tau,\psi}(U,\mc{O})$ via the double coset operators
\[ \left[ U i_v^{-1}\left(\begin{matrix} \varpi_v & 0 \cr 0 &
      1 \end{matrix} \right) U \right]
\mathrm{\ \ and \ \ }  \left[ U i_v^{-1} \left(\begin{matrix} \varpi_v
      & 0 \cr 0 & \varpi_v \end{matrix} \right) U \right]
\]
respectively, where $\varpi_v$ is a uniformizer in $\mc{O}_{F^+_{1,v}}$.

Let $\wt{\pi}$ denote the automorphic representation of
$(B\otimes_{\bb{Q}}\bb{A})^{\times}$ of weight
$0$ corresponding to $\pi_{F_1^+}$
under the Jacquet-Langlands correspondence. 
Choose a place $v_0$ of $F^+_1$ such that
\begin{itemize}
\item $B$ is split at $v_0$;
\item $v_0$ does not split completely in $F^+_1(\zeta_l)$;
\item $\wt{\pi}$ is unramified at $v_0$ and $\ad \rbar(\Frob_{v_0})=1$;
\item for every non-trivial root of unity $\zeta$ in a quadratic extension of $F^+_1$, $\zeta+\zeta^{-1}\not \equiv 2 \mod v_0$.
\end{itemize}
The second and third conditions imply that $H^2(G_{F^+_{1,v_0}},\ad^0
\rbar)$ is trivial and that every deformation of
$\rbar|_{G_{F^+_{1,v_0}}}$ is unramified.  Let $U= \prod_v U_v
\subset (B\otimes_{\bb{Q}}\bb{A}^{\infty})^{\times}$ be the compact
open subgroup with 
\begin{itemize}
\item $U_v = \mc{O}_{B,v}^{\times}$ for all $v|l$,  and all $v\neq v_0$ where $\wt\pi_v$ is unramified;
\item $U_{v_0}=\iota_{v_0}^{-1}\Iw_1(v_0)$ where $\Iw_1(v_0)$ is the subgroup of $\Iw(v_0)$ consisting of all elements whose reduction modulo $v_0$ is unipotent;
\item  $U_v = \iota_v^{-1} \Iw(v)$ if $v\nmid l$ and $\wt\pi_v$ is Steinberg.
\end{itemize}
The subgroup $U$ then satisfies hypothesis 3.1.2 of \cite{kis04} (see
the remarks following (2.1.2) of \cite{kis-fmc}).

Let $R$ denote the set of primes $v\nmid l$ of $F^+_1$ where $U_v \neq \mc{O}_{B,v}^\times$.
For each $v|l$ with $\wt\pi_v$ unramified, let $\tau_v$ denote the
trivial representation of $\mc{O}_{B,v}^{\times}$ on $\mc{O}$. For
each $v|l$ with $\wt\pi_v$ Steinberg, let $\tau_v$ denote the
$\mc{O}$-dual of the representation $\Ind_{\iota_v^{-1} \Iw(v)}^{\mc{O}_{B,v}^{\times}} \mc{O}$ modulo the constants. 
Let
$\chi: G_{F_1^+}^{\ab} \rightarrow \Qlbar^{\times}$ denote the
character $\epsilon \det r_{l,\iota}(\pi)|_{G_{F_1^+}}$ and let $\psi
= \chi \circ \Art_{F_1^+} : \bb{A}_{F_1^+}^{\times}/\overline{
  (F^+_{1,\infty})^{\times}_{>0}(F_1^+)^{\times}} \rightarrow
\Qlbar$. Note that $\chi$ is totally even and hence we may regard
$\psi$ as a character of
$(\bb{A}_{F_1^+}^{\infty})^{\times}/(F_1^+)^{\times} \isoto
\bb{A}_{F_1^+}^{\times}/\overline{
  (F^+_{1,\infty})^{\times}(F_1^+)^{\times}}$. Extending $K$ if
necessary, we can and do assume that $\psi$ is valued in
$\mc{O}^{\times}$. Note that for each $v|l$, $\Hom_{\mc{O}_{B,v}^\times}(\tau_v^{\vee} \otimes_{\mc{O},\iota}\bb{C}, \wt\pi_v) \neq \{0\}$, while for $v\nmid l$ we have $\wt\pi_v^{U_v}\neq\{0\}$. It follows that the subspace of $S_{0,\tau,\psi}(U_l,\mc{O})\otimes_{\mc{O},\iota}\bb{C}$ corresponding to $\wt\pi$ under \eqref{eq:iso} has non-zero intersection with $S_{0,\tau,\psi}(U,\mc{O})\otimes_{\mc{O},\iota}\bb{C}$.
 Further extending $K$ if necessary, we can and do choose a
$\bb{T}^{R}$-eigenform $f$ in
$S_{0,\tau,\psi}(U,\mc{O})$ which lies in this intersection.
The
$\bb{T}^{R}$-eigenvalues on $f$ give rise to an
$\mc{O}$-algebra homomorphism $\bb{T}^{ R}\rightarrow
\mc{O}$ and reducing this modulo $\mf{m}_{\mc{O}}$ gives rise to a
maximal ideal $\mf{m}$ of $\bb{T}^{ R}$.

By Corollary 3.1.6 of \cite{kis04}, after extending $K$, we can and do
choose a collection $\{\tau'_v\}_{v|l}$ where each $\tau'_v$ is a
cuspidal $F^+_{1,v}$-type such that the action of
$\mc{O}_{F^+_1,v}^{\times}$ on $\tau'_v$ is given by
$\psi^{-1}|_{\mc{O}_{F^+_{1,v}}^\times}$ (the trivial character)
and such that $S_{0,\tau',\psi}(U,\mc{O})_{\mf{m}} \neq \{0\}$ where $\tau'=\otimes_{v|l}\tau'_v$.

At each place $v\in R$ choose a non-trivial character $\chi_v :
\mc{O}_{F^+_{1,v}}^\times \to \mc{O}^\times$ which factors through
$k(v)^{\times}$ (where $k(v)$ is the residue field of $v$) and reduces to the
trivial character modulo $\mf{m}_{\mc{O}}$. Here we use the assumption
that $\mathbf{N}v\equiv 1 \mod l$. Let $\chi = \{\chi_v
\}_{v \in R}$. For $v \in R$ let $W_{\chi_v}$ denote the free rank 1
$\mc{O}$-module with action of $\Iw(v)$ given by the character
$\left(\begin{matrix} a & b \\ c & d \end{matrix}\right) \mapsto
\chi_v(ad^{-1})$. Let $W_{\chi}=\otimes_{v\in R} W_{\chi_v}$. Let
$S_{0,\tau',\psi,\chi}(U,\mc{O})$ denote the space of functions
\[ f : B^{\times}\backslash
(B\otimes_{\bb{Q}}\bb{A}^{\infty})^{\times} \rightarrow
W_{\tau}\otimes_{\mc{O}} W_{\chi} \] with $f(gu)=(\tau\otimes
\chi)(u_{l,R})^{-1}f(g)$ and $f(gz)=\psi(z)f(g)$ for all $u \in U$, $z
\in (\bb{A}_{F_1^+}^{\infty})^{\times}$ and $g \in
(B\otimes_{\bb{Q}}\bb{A}^{\infty})^{\times}$. Since $\overline{\chi}_v
= 1$ for each $v \in R$, we have
$S_{0,\tau',\psi,\chi}(U,\mc{O})_{\mf{m}} \neq \{0\}$ by Lemma 3.1.4
of \cite{kis04}. Extending
$\mc{O}$ we can choose a $\bb{T}^{R}$-eigenform $g \in
S_{0,\tau',\psi,\chi}(U,\mc{O})_{\mf{m}} \neq \{0\}$. Applying the
Jacquet-Langlands correspondence, $g$ gives rise to a regular
algebraic cuspidal automorphic representation $\pi_1$ of
$\GL_2(\bb{A}_{F^+_1})$ of weight 0 such that $\rbar_{l,\iota}(\pi_1)
\cong \rbar|_{G_{F^+_1}}$ and $\pi_{1,v}$ is supercuspidal for all
$v|l$, a ramified principal series representation for all $v \in R$
and unramified otherwise. The result follows easily by replacing
$F^+_1$ with an appropriate totally real solvable extension
$L^+/F^+_1$.
\end{proof}

\begin{lemma}
\label{lem: changing weight and level}
Let $F$ be an imaginary CM field with maximal totally real subfield
$F^+$. Let $\pi$ be a RAECSDC automorphic representation of
$\GL_n(\bb{A}_F)$ which is $\iota$-ordinary at all places dividing
$l$. Let $\chi : \bb{A}_{F^+}^{\times}/(F^+)^\times \to \bb{C}^\times$ be an
algebraic character with $\chi_v(-1)$ independent of $v|\infty$ and
 $\pi^c \cong \pi^{\vee}\otimes (\chi\circ \mathbf{N}_{F/F^+}\circ
 \det)$. Let $\wt{\chi}: \bb{A}_{F^+}^{\times}/(F^+)^\times \to
 \bb{C}^\times$ be the finite order algebraic character with
 $r_{l,\iota}(\wt{\chi})$ equal to the Teichm\"uller lift of $\rbar_{l,\iota}(\chi)$.
Suppose that $\rbar_{l,\iota}(\pi)$ is irreducible. 

Let $\lambda' \in (\bb{Z}^n_+)^{\Hom(F,\bb{C})}_{c}$ be a (conjugate-self-dual) weight.  Let $F'/F$ be a
finite extension.  We can find a finite solvable CM extension $L$ of
$F$, linearly disjoint from $F'$ over $F$ and such that there exists a
RAECSDC automorphic representation $\pi'$ of $\GL_n(\bb{A}_L)$ such
that:
\begin{enumerate}
\item $(\pi')^{c}\cong (\pi')^{\vee}\circ(\wt{\chi}\circ
  \mathbf{N}_{L/F^+}\circ\det)$.
\item $\pi'$ is of weight $\lambda'_L$.
\item $\pi'$ is $\iota$-ordinary at all places dividing $l$.
\item $\rbar_{l,\iota}(\pi')\cong \rbar_{l,\iota}(\pi)|_{G_L}$.
\end{enumerate}
\end{lemma}

\begin{proof}
  Using Lemma 4.1.4 of \cite{cht}, choose an algebraic character
  $\psi: \bb{A}_F^\times/F^\times \to \bb{C}^\times$ such that $\psi
  \psi^c = \chi^{-1} \circ \mathbf{N}_{F/F^+}$. After replacing $F$ by a
  solvable CM extension, linearly disjoint from $F'\overline{F}^{\ker
    \rbar_{l,\iota}(\pi)}$ over $F$, we may assume that $\psi$ is
  unramified at all finite places. (Lemma 5.1.6 of \cite{ger} shows
  that the ordinarity of $\pi$ is preserved under such a base
  change.) Then $\pi \otimes(\psi \circ \det)$ is RACSDC. Applying Lemma
  5.1.7 of \cite{ger}, we can find a solvable CM extension $L$ of $F$,
  linearly disjoint from $F'$ over $F$, and a RACSDC automorphic
  representation $\pi''$ of $\GL_n(\bb{A}_L)$ of weight $\lambda'_L$
  which is unramified at all finite places, $\iota$-ordinary at all
  places dividing $l$  and satisfies $\rbar_{l,\iota}(\pi') \cong \rbar_{l,\iota}(\pi)|_{G_L}\otimes
  \rbar_{l,\iota}(\psi)|_{G_L}$. Let $\wt{r}_{l,\iota}(\psi) : G_{F}
  \to \mc{O}_{\Qlbar}^{\times}$ be the Teichm\"uller lift of
  $\rbar_{l,\iota}(\psi)$ and let $\wt{\psi} : \bb{A}_{F}^{\times}/F^\times
  \to \bb{C}^\times$ be the algebraic character with
  $r_{l,\iota}(\wt{\psi})=\wt{r}_{l,\iota}(\psi)$. We now take
  $\pi'=\pi''\otimes (\wt{\psi} \circ\mathbf{N}_{L/F}\circ \det)$.
\end{proof}

The following proposition is the main technical result of this paper.

\begin{prop}
  \label{prop: existence of a potential ordinary lift}Let
  $\rbar:G_{F^+}\to\GL_2(\Flbar)$ be an irreducible modular
  representation. Assume further that
  \begin{enumerate}
  \item $l\ge 5$.
  \item $\rbar(G_{F^+(\zeta_l)})$ is $2$-big.
  \item $[\overline{F^+}^{\ker\ad\rbar}(\zeta_l):\overline{F^+}^{\ker\ad\rbar}]>2$.
  \end{enumerate}
  Then there exists a finite solvable extension of totally real fields
  $L^+/F^+$ such that
  \begin{itemize}
  \item $L^+$ is linearly disjoint from
    $\overline{F^+}^{\ker\rbar}(\zeta_l)$ over $F^+$.
  \item $\rbar|_{G_{L_v^+}}$ is trivial for all places $v|l$ of $L^+$.
  \item There is a regular algebraic cuspidal automorphic
    representation $\pi$ of $\GL_2(\A_{L^+})$ of weight $0$ such that
    \begin{itemize}
    \item $\rbar_{l,\iota}(\pi)\cong\rbar|_{G_{L^+}}$.
    \item $\pi$ is unramified at all finite places.
    \item For all places $v|l$ of $L^+$,
      $r_{l,\iota}(\pi)|_{G_{L_v^+}}$ is ordinary.
    \end{itemize}
  \end{itemize}
\end{prop}

\begin{proof}
  We firstly apply Lemma \ref{lem: existence of potentially non-ordinary lift to
    weight 0}, to deduce that there is a finite solvable extension of totally real fields
  $F_2^+/F^+$ such that
  \begin{itemize}
  \item  $F_2^+$ is linearly disjoint from $\overline{F^+}^{\ker\rbar}(\zeta_l)$
  over $F^+$.
\item $\rbar|_{G_{F_{2,v}^+}}$ is trivial for all places $v|l$ of
  $F_2^+$.
\item For each place $v|l$, $[F_{2,v}^+:\Ql]\ge 2$.
\item There is a regular algebraic cuspidal automorphic representation $\pi_2$ of
  $\GL_2(\A_{F_2^+})$ of weight 0 such that
  \begin{itemize}
  \item $\rbar_{l,\iota}(\pi_2)\cong\rbar|_{G_{F^+_2}}$.
  \item $\pi_2$ is unramified at all finite places.
  \item For all places $v|l$ of $F_2^+$,
    $r_{l,\iota}(\pi_2)|_{G_{F_{2,v}^+}}$ is non-ordinary.
  \end{itemize}
\end{itemize}
We now employ Lemma \ref{lem:char building} in the following setting:
\begin{itemize}
\item $F^+$ is the present $F_2^+$.
\item $l$ is as in the present setting.
\item $m=n=2$.
\item $\Favoid=F_2^+\overline{F^+}^{\ker\ad\rbar}(\zeta_l)$.
\item $T=\emptyset$.
\item $\eta=\eta'=1$.
\item $\{h_{1,\tau},h_{2,\tau}\}=\{0,1\}$ for each
  $\tau$. Furthermore, for each place $v|l$, there is at least one
  $\tau$ corresponding to $v$ with $h_{1,\tau}=0$, and at least one
  $\tau$ corresponding to $v$ with $h_{1,\tau}=1$.
\item $h'_{1,\tau}=0$ and $h'_{2,\tau}=3$ for all $\tau$.
\item $w=1$
\item $w'=3$
\end{itemize}
We obtain a quadratic CM extension $M/F_2^+$ together with two
continuous characters \[\theta,\theta':G_M\to\Zlbar^\times\]such that
\begin{enumerate}
\item $\theta$, $\theta'$ are congruent modulo $l$.
\item
  $(\rbar|_{G_{F_2^+}}\otimes\Ind_{G_M}^{G_{F_2^+}}\bar{\theta})(G_{F_2^+(\zeta_l)})$
  is big.
\item $\zeta_l\notin\overline{F_2^+}^{\ad(\rbar|_{G_{F_2^+}}\otimes\Ind_{G_M}^{G_{F_2^+}}\bar{\theta})}$.
\item $(\Ind_{G_M}^{G_{F^+_2}} \theta)\cong (\Ind_{G_M}^{G_{F^+_2}} \theta)^\vee
     \tensor \eps^{-1}.$
   \item $(\Ind_{G_M}^{G_{F^+_2}} \theta')\cong (\Ind_{G_M}^{G_{F^+_2}}
     \theta')^\vee\tensor \eps^{-3}\wtilde^{2}.$
   \item For each $v$ above $l$, the representation $(\Ind_{G_M}^{G_{F^+_2}} \theta)|_{G_{F^+_{2,v}}}$
is conjugate to a representation which breaks up as a direct sum of characters:

$$(\Ind_{G_M}^{G_{F^+_2}} \theta)|_{G_{F^+_{2,v}}} \cong 
               \chi^{(v)}_1\oplus\chi^{(v)}_2 $$
where, for each $i=1,2$, and each embedding $\tau:F^+_{2,v}\to\Qbar_l$ we have that:
$$\HT_\tau(\chi^{(v)}_i) = h_{i,\tau}.$$ In particular, the
representation $(\Ind_{G_M}^{G_{F^+_2}} \theta)|_{G_{F^+_{2,v}}}$ is
Barsotti-Tate and non-ordinary. Similarly, the representation
$(\Ind_{G_M}^{G_{F^+_2}} \theta')|_{G_{F^+_{2,v}}}$ is conjugate to a
representation which breaks up as a direct sum of characters:
$$(\Ind_{G_M}^{G_{F^+_2}} \theta')|_{G_{F^+_{2,v}}} \cong 
               \chi'^{(v)}_1\oplus\chi'^{(v)}_2$$
where, for each $i=1,2$, and each embedding $\tau:F^+_{2,v}\to\Qbar_l$ we have that:
$$\HT_\tau(\chi'^{(v)}_i) = h'_{i,\tau}.$$ Thus by the choice of the
$h'_{i,\tau}$, $\chi'^{(v)}_1|_{I_{F^+_{2,v}}}$ has finite order,
and $\chi'^{(v)}_2|_{I_{F^+_{2,v}}}$ is
a finite order character times $\epsilon^{-3}$, so that
$(\Ind_{G_M}^{G_{F^+_2}} \theta')|_{G_{F^+_{2,v}}}$ is ordinary.
\end{enumerate}
Let $F^+_3/F^+_2$ be a solvable extension of totally real
fields such that
\begin{itemize}
\item $F^+_3$ is linearly disjoint from $\overline{F^+}^{\ker\rbar}(\zeta_l)$ over
$F^+$, and
\item $(\Ind_{G_M}^{G_{F^+_2}} \theta)|_{G_{F^+_3}}$ and
  $(\Ind_{G_M}^{G_{F^+_2}} \theta')|_{G_{F^+_3}}$ are both unramified
  at all places of $F^+_3$ not lying over $l$, and crystalline at all
  places dividing $l$.
\item If $v$ is a place of $F_3^+$ lying over $l$, then
  $(\Ind_{G_M}^{G_{F^+_2}} \thetabar)|_{G_{F^+_3,v}}$ is trivial.
\item If $v$ is a place of $F^+_3$ lying over $l$, then $F^+_{3,v}$
  contains a primitive $l$-th root of unity.
\end{itemize}
Let $F_3/F_3^+$ be a quadratic CM extension which is linearly disjoint
from $M\overline{F^+}^{\ker\rbar}(\zeta_l)$ over $F^+$, and in which
all places of $F_3^+$ lying over $l$ split completely. We choose
$K\subset\Qlbar$ a finite extension of $\Ql$ with ring of integers
$\bigO$ and residue field $k$. Assume that $K$ is sufficiently large
that it contains the image of every embedding $F_3\into\Qlbar$, and
the images of $\theta$, $\theta'$ and $r_{l,\iota}(\pi_2)$. We now
regard $r_{l,\iota}(\pi_2)$, $(\Ind_{G_M}^{G_{F^+_2}}
\theta)|_{G_{F^+_3}}$ and $(\Ind_{G_M}^{G_{F^+_2}}
\theta')|_{G_{F^+_3}}$ as representations to $\GL_2(\bigO)$, and
(after conjugating if necessary) we can and do suppose that
$(\Ind_{G_M}^{G_{F^+_2}}
\bar{\theta})|_{G_{F^+_3}}=(\Ind_{G_M}^{G_{F^+_2}}
\bar{\theta}')|_{G_{F^+_3}}$. Note that there is a character
$\chi_{\pi_2}:G_{F^+_2}\to\bigO^\times$ such
that \[r_{l,\iota}(\pi_2)\cong r_{l,\iota}(\pi_2)^\vee\chi_{\pi_2}.\]

We now apply Theorem \ref{thm:existence of
  a lift with specified properties} in the following setting:
  \begin{itemize}
  \item $m=n=2$, and $l$ is as it has been throughout this section.
  \item $F=F'=F_3$.
  \item $\rbar:=\rbar|_{G_{F_3}}$.
  \item $\chi=\epsilon^{-2}\tilde{\omega}^{2}\chi_{\pi_2}|_{G_{F^+_3}}$.
  \item $\chi'=\epsilon^{-1}$.
  \item $\chi''=\epsilon^{-3}\tilde{\omega}^{2}\chi_{\pi_2}|_{G_{F^+_3}}\delta_{F_3/F_3^{+}}$.
  \item $\tilde{S}_l$ is any set of places of $F_3$ consisting of exactly
    one place above each place in $S_l$.
  \item For each place $\tv\in\tilde{S}_l$, the element
    $a_{\tv}\in(\Z^2_+)^{\Hom(F_{3,\tv},\Qlbar)}$ is given by
    $a_{\tv,1}=2$, $a_{\tv,2}=0$.
  \item $R_\tv$ is the unique ordinary component of
    $R^{\mathbf{v}_{a_\tv},cr}_{\rbar|_{G_{F_{3,\tv}}}}$ (note that
    $\rbar|_{G_{F_{3,\tv}}}$ is trivial, and
    $\bar{\epsilon}|_{G_{F_{3,\tv}}}$ is trivial, so there is a unique ordinary
    component by Lemma 3.4.3 of \cite{ger}).
  \item $r'=(\Ind_{G_M}^{G_{F^+_2}} \theta)|_{G_{F_3}}$.
  \item $r''=r_{l,\iota}(\pi_2)|_{G_{F_3}}\otimes(\Ind_{G_M}^{G_{F^+_2}} \theta')|_{G_{F_3}}$.
  \end{itemize}
We now check that the hypotheses of Theorem \ref{thm:existence of a lift with specified
    properties} hold.
  \begin{itemize}
  \item[(1),(3):] That $\rbar^c\cong\rbar^\vee\chibar|_{G_{F_3}}$, and $\rbar$
    is odd, follow directly from the definitions and the fact that
    $\rbar$, being modular, is odd (see the discussion of section \ref{subsubsec: oddness totally real}). That $(r')^c \cong (r')^\vee \chi'|_{G_{F_3}}$ follows from (4) above.
  \item[(2),(4)] That $\rbar,r',r'',\chi,\chi'$ and $\chi''$ are unramified away from $l$ and $r'$ is crystalline above $l$ follows from the choice of $F_3$ and properties of $\pi_2$.
   \item[(5)] $r''$ is automorphic of level prime to $l$ by Proposition
    5.1.3 of \cite{BLGG}. It satisfies $(r'')^c \cong (r'')^{\vee}\chi''|_{G_{F_3}}$ by construction.
     \item[(6)] $\rbar''=\rbar\otimes\rbar'$ because $\bar{\theta}=\bar{\theta}'$
    and $r_{l,\iota}(\pi_2)|_{G_{F_3}}=\rbar$. 
	By definition we have $\chi''|_{G_{F_3}} =
        \chi|_{G_{F_3}}\chi'|_{G_{F_3}}$.
  \item[(7), (8)] That $\rbar''({G_{F_3(\zeta_l)}})$ is big and
    $\bar{F}^{\ker\ad\rbar''}$ does not contain $\zeta_l$ follow from property
    (2) of Lemma \ref{lem:char building}, the assumptions on
    $\rbar$ and the choice of $F_3$.
  \item[(9)] For all places $v\in\tilde{S}_l$, if
    $\rho_v:G_{F_{3,\tv}}\to\GL_2(\bigO)$ is a lifting of
    $\rbar|_{G_{F_{3,\tv}}}$ corresponding to a closed point on
    $R_\tv$, then $\rho_v$ is ordinary (by the definition of $R_\tv$), so
    $\rho_v\sim(\Ind_{G_M}^{G_{F^+_2}} \theta')|_{G_{F_{3,\tv}}}$ by
    Lemma 3.4.3 of \cite{ger} (because both representations are
    crystalline and ordinary with the same Hodge-Tate weights, and
    trivial reduction). Similarly,
    $r'|_{G_{F_{3,\tv}}}=(\Ind_{G_M}^{G_{F^+_2}}
    \theta)|_{G_{F_{3,\tv}}}\sim r_{l,\iota}(\pi_2)|_{G_{F_{3,\tv}}}$
    by Proposition 2.3 of \cite{MR2280776} (because both
    representations are Barsotti-Tate and non-ordinary with trivial
    reduction). By the results of sections 3.3 and 3.4 of \cite{BLGG},
    it follows that $\rho_v\otimes r'|_{G_{F_{3,\tv}}}\sim
    (\Ind_{G_M}^{G_{F^+_2}} \theta')|_{G_{F_{3,\tv}}}\otimes
    r_{l,\iota}(\pi_2)|_{G_{F_{3,\tv}}}\sim
    r_{l,\iota}(\pi_2)|_{G_{F_{3,\tv}}}\otimes (\Ind_{G_M}^{G_{F^+_2}}
    \theta')|_{G_{F_{3,\tv}}} = r''|_{G_{F_{3,\tv}}}$.
  \end{itemize}
We conclude that, after possibly extending $\mc{O}$, there is a continuous lifting
$r:G_{F_3}\to\GL_2(\bigO)$ of $\rbar|_{G_{F_3}}$ such that:
\begin{itemize}
\item $r$ is unramified at all places of $F$ not dividing $l$.
\item $r|_{G_{F_{3,\tv}}}$ is crystalline and ordinary at each place
  $\tv\in\tilde{S}_l$, with Hodge-Tate weights $0$ and $3$.
\item $r^c\cong r^\vee\chi|_{G_F}$.
\item $r\otimes(\Ind_{G_M}^{G_{F_2^+}}\theta)|_{G_{F_3}}$ is
  automorphic of level prime to $l$.
\end{itemize}
Applying Proposition \ref{prop:untwisting CM} (noting that
$r|_{G_{MF_3}}$ is certainly irreducible, because $\rbar|_{G_{MF_3}}$
is irreducible), we deduce that in fact
\begin{itemize}
\item $r$ is automorphic of level prime to $l$.
\end{itemize}
Since, in addition, $r|_{G_{F_{3,\tv}}}$ is ordinary for all $\tv\in\tilde{S}_l$, it
follows from Lemma 5.2.1 of \cite{ger} 
that in fact
\begin{itemize}
\item $r$ is $\iota$-ordinarily automorphic of level prime to $l$.
\end{itemize}
By Lemma \ref{lem: changing weight and level} we can and do choose a
solvable extension $F_4^+/F_3^+$ of totally real fields together with
a RAECSDC automorphic representation $\pi_4$ of $GL_2(\bb{A}_{F_4})$
of weight 0 where $F_4=F_4^+F_3$ such that \begin{itemize}
\item $F^+_4$ is linearly disjoint from $F_3\overline{F^+}^{\ker\rbar}(\zeta_l)$ over
$F_3^+$.
\item $\pi_4$ is unramified at all finite places and $\iota$-ordinary
\item $\rbar_{l,\iota}(\pi)\cong \rbar|_{G_{F_4}}$. 
\item $r_{l,\iota}(\pi_4)^c \cong r_{l,\iota}(\pi_4)^\vee
  \epsilon^{-1}\wt{\omega}\wt{\chi}|_{G_{F_4}}$ where $\wt{\chi}$ is the
  Teichm\"uller lift of $\chibar$.
\end{itemize}
Enlarging $\bigO$ if necessary, we may assume that
$r_{l,\iota}(\pi_4)$ is valued in $\GL_2(\bigO)$.  Let $\chi_{\pi_4} =
\epsilon^{-1}\wt{\omega}\wt{\chi}|_{G_{F_4^+}}$. 

We now consider two deformation problems, one for $\rbar|_{G_{F_4^+}}$
and one for $\rbar|_{G_{F_4}}$. Let $R_{\rbar|_{G_{F^+_4}}}$ be the
universal deformation ring for ordinary Barsotti-Tate deformations of
$\rbar|_{G_{F^+_4}}$ which have determinant $\chi_{\pi_4}$ and which
are unramified outside $l$. Let $S'_l$ be the set of places of $F_4^+$
above $l$, and let $\tilde{S}'_l$ be the set of places of $F_4$ which
lie above places in $\tilde{S}_l$. We can and do extend
$\rbar|_{G_{F_4}}$ to a representation
$\rbar_4:G_{F_4^+}\to\G_2(k)$. For each place $\tv\in\tilde{S}'_l$,
let $R_v$ be the unique ordinary component of the Barsotti-Tate
lifting ring for $\rbar|_{G_{F_{4,\tv}}}$. Let $\mc{S}$ be the
deformation problem (in the sense of section \ref{subsubsec: global
  deformation
  rings}) \[(F_4/F_4^+,S'_l,\tilde{S}'_l,\bigO,\rbar_4,\chi_{\pi_4},\{R_v\}_{v
  \in S_l'}).\] Then by Proposition 3.6.3 of \cite{BLGG}, the
corresponding universal deformation ring $R_{\mathcal{S}}$ is a finite
$\bigO$-algebra. Furthermore, $R_{\rbar|_{G_{F^+_4}}}$ is a finite
$R_{\mathcal{S}}$-algebra in a natural fashion (cf. section 7.4 of
\cite{GG}). In addition, by Proposition 3.1.4 of \cite{gee061}, $\dim
R_{\rbar|_{G_{F^+_4}}}\ge 1$. Thus $R_{\rbar|_{G_{F^+_4}}}$ is a
finite $\bigO$-algebra of rank at least one (cf. the proof of Theorem
4.2.8 of \cite{MR2459302}), and so it has a $\Qlbar$-point, which
corresponds to a deformation $r_4:G_{F^+_4}\to\GL_2(\Qlbar)$ of
$\rbar|_{G_{F^+_4}}$ which is Barsotti-Tate and ordinary at each place
dividing $l$, and which is unramified outside $l$. Furthermore, by
Theorem 5.4.2 of \cite{ger}, $r_4|_{G_{F_4}}$ is automorphic, so that
$r_4$ is automorphic by Lemma 1.5 of \cite{BLGHT}, as
required.  
\end{proof}

We can prove a similar result more directly in the case that $\rbar$
is induced from a CM extension.

\begin{prop}\label{prop: potential ordinary lift in the dihedral
    case.}
  Suppose that $l\ge 3$ is a prime, and that $M/F^+$ is a quadratic
  extension, with $M$ a CM field. Let $\bar{\theta}:G_M\to\Flbar^\times$ be
  a continuous character that does not extend to $G_{F^+}$, so that
  $\rbar:=\Ind_{G_M}^{G_{F^+}}\bar{\theta}$ is an irreducible modular
representation. Then there is a finite solvable extension $L^+/F^+$
such that
  \begin{itemize}
  \item $L^+$ is linearly disjoint from
$M\overline{F^+}^{\ker\rbar}(\zeta_l)$ over $F^+$.
  \item There is an $\iota$-ordinary regular algebraic cuspidal
automorphic representation $\pi$ of $\GL_2(\A_{L^+})$ of weight $0$
and level prime to $l$ such that
    \begin{itemize}
    \item $\rbar_{l,\iota}(\pi)\cong\rbar|_{G_{L^+}}$.
    \item $\pi$ is unramified at all finite places.
    \end{itemize}
  \end{itemize}
\end{prop}

\begin{proof} We construct $\pi$ as an automorphic induction. Choose
$F_1^+/F^+$ a solvable extension of totally real fields such
that \begin{itemize}
  \item $F^+_1$ is linearly disjoint from
$M\overline{F^+}^{\ker\rbar}(\zeta_l)$ over $F^+$.
  \item Every place $v|l$ of $F^+_1$ splits in $F_1^+M$.
  \end{itemize} We now apply Lemma 4.1.6 of \cite{cht} with
  \begin{itemize}
  \item $F=F^+_1M$.
  \item $S$ the set of places of $F$ dividing $l$.
  \item $\bar{\theta}$ equal to our $\bar{\theta}|_{G_F}$.
  \item
    $\chi=\epsilon^{-1}\widetilde{(\det\rbar|_{G_{F^+_1}}\bar{\epsilon})}$,
    where a tilde denotes the Teichm\"{u}ller lift.
  \item For each place $v|l$ of $F^+_1$, write the places of $F$ lying
    over $v$ as $\tv$ and
    $\tv^c$. Then \[\psi_\tv=\widetilde{\bar{\theta}|_{G_{F_\tv}}},\]and \[\psi_{\tv^c}=\epsilon^{-1}\widetilde{\bar{\epsilon}\bar{\theta}|_{G_{F_\tv^c}}}.\]
  \end{itemize} We conclude that there is an algebraic character
  $\theta:G_{F^+_1M}\to\Zlbar^\times$ lifting $\bar{\theta}$, such
  that $\Ind_{G_{F^+_1M}}^{G_{F^+_1}}\theta$ is ordinary and
  potentially Barsotti-Tate.  Choose $F_2^+/F_1^+$ a solvable
  extension of totally real fields such that \begin{itemize}
  \item $F^+_2$ is linearly disjoint from
    $M\overline{F^+}^{\ker\rbar}(\zeta_l)$ over $F^+$.
  \item $\theta|_{G_{F^+_2M}}$ is crystalline at all places dividing
    $l$, and unramified at all other places.
  \item $F^+_2M/F^+_2$ is unramified at all finite places.
  \end{itemize} Then by Theorem 4.2 of \cite{MR1007299}, there is a
  regular algebraic cuspidal automorphic representation $\pi$ of
  $\GL_2(\A_{F^+_2})$ of weight $0$ and level prime to $l$ such
  that \[r_{l,\iota}(\pi)\cong\Ind_{G_{F^+_2M}}^{G_{F^+_2}}\theta|_{G_{F^+_2M}}.\]The
  result follows, as $\pi$ is $\iota$-ordinary by Lemma 5.2.1 of
  \cite{ger}.
\end{proof}

We now deduce the main results, essentially by combining the previous
results with those of \cite{gee061}. For the terminology of inertial
types, see section 3 of \cite{gee061} (although note the slightly
different conventions for Hodge-Tate weights in force there).

\begin{thm}
  \label{thm: main result on existence of ordinary lifts with types
    everywhere.} Let $l\ge 3$ be prime, and let $F^+$ be a totally
  real field. Let $\rbar:G_{F^+}\to\GL_2(\Flbar)$ be irreducible and
  modular. Fix a character $\psi:G_{F^+}\to\Qlbar^\times$ such that
  $\epsilon\psi$ has finite order, and $\bar{\psi}=\det\rbar$. Let
  $S$ denote a finite set of finite places of $F^+$ containing all
  places at which $\rbar$ or $\psi$ is ramified, and all places
  dividing $l$.

  For each place $v\in S$, fix an inertial type $\tau_v$ of
  $I_{F^+_v}$ on a $\Qlbar$-vector space, of determinant
  $(\psi\epsilon)|_{I_{F^+_v}}$. Assume that for each place $v\in
  S$, $v\nmid l$, $\rbar|_{G_{F^+_v}}$ has a lift of type $\tau_v$
  and determinant $\psi|_{G_{F^+_v}}$, and for each place $v\in S$,
  $v| l$, $\rbar|_{G_{F^+_v}}$ has a potentially Barsotti-Tate lift
  of type $\tau_v$ and determinant $\psi|_{G_{F^+_v}}$. For each place
  $v\in S$, we let $R_v$ denote an irreducible component of the
  corresponding lifting ring $R^{\square,\psi,\tau_v}[1/l]$ for
  (potentially Barsotti-Tate) lifts of type $\tau_v$ and determinant
  $\psi|_{G_{F^+_v}}$.

Assume further that
  \begin{itemize}
  \item Either
    \begin{itemize}
    \item
      \begin{itemize}
      \item $l\ge 5$.
  \item $\rbar(G_{F^+(\zeta_l)})$ is $2$-big.
  \item $[\overline{F^+}^{\ker\ad\rbar}(\zeta_l):\overline{F^+}^{\ker\ad\rbar}]>2$.
      \end{itemize}Or:

    \item
      \begin{itemize}
      \item There is a quadratic CM extension $M/F^+$, with $M$ not
        equal to the quadratic extension of $F^+$ in $F^+(\zeta_l)$,
        and a continuous character $\bar{\theta}:G_M\to\Flbar^\times$
        such that $\rbar=\Ind_{G_M}^{G_{F^+}}\bar{\theta}$.
      \end{itemize}
    \end{itemize}

  \end{itemize}
Then there is a continuous representation
$r:G_{F^+}\to\GL_2(\Zlbar)$ lifting $\rbar$ of determinant $\psi$ such that
\begin{itemize}
\item $r$ is modular.
\item $r$ is unramified at all places $v\notin S$.
\item For each place $v|l$ of $F^+$, $r|_{G_{F^+_v}}$ is
  potentially Barsotti-Tate of type $\tau_v$ (and indeed corresponds
  to a point of $R_v$).
\item For each place $v\in S$, $v\nmid l$, $r|_{G_{F^+_v}}$ has type $\tau_v$ (and indeed corresponds
  to a point of $R_v$).
\end{itemize}
\end{thm}

\begin{proof}We argue as in the proofs of Proposition 3.1.5 and
  Corollary 3.1.7 of \cite{gee061}. Indeed, examining the arguments of
  \emph{loc. cit.}, we see that it is enough to demonstrate that there
  is a finite solvable extension of totally real fields $F^+_2/F^+$
  such that:
  \begin{itemize}
  \item There is an $\iota$-ordinary regular algebraic cuspidal
    automorphic representation $\pi_2$ of $\GL_2(\A_{F^+_2})$ of
    weight 0 which is unramified at every finite place, with
    $\rbar_{l,\iota}(\pi_2)\cong\rbar|_{G_{F^+_2}}$, and $\det
    r_{l,\iota}(\pi_2)=\psi|_{G_{F^+_2}}$.
  \item If $w$ is a place of $F^+_2$ lying over a place $v\in S$, then
    $\tau_v|_{G_{F^+_{2,w}}}$ is trivial.
  \item $\rbar|_{G_{F^+_{2,w}}}$ is trivial for all places $w|l$ of
    $F_2^+$.
  \item $\rbar|_{G_{F^+_2(\zeta_l)}}$ is irreducible, $[F^+_2:\Q]$ is
    even, and $[F^+_2(\zeta_l):F^+_2]=[F^+(\zeta_l):F^+]$.
  \end{itemize}
To see that we can do this, note that by Propositions \ref{prop:
  existence of a potential ordinary lift} and \ref{prop: potential
  ordinary lift in the dihedral case.}, there is a finite solvable
extension of totally real fields $F_1^+/F^+$ and a regular algebraic
$\iota$-ordinary cuspidal automorphic
representation $\pi_1$ of $\GL_2(\A_{F^+_1})$ of weight 0 such that
\begin{itemize}
\item $\pi_1$ is unramified at all finite places.
\item $\rbar_{l,\iota}(\pi_2)\cong\rbar|_{G_{F^+_1}}$.
\item $F^+_1$ is linearly disjoint from
  $\overline{F^+}^{\ker\rbar}(\zeta_l)$ over $F^+$.
\end{itemize}
Choose a solvable extension of totally real fields $F^+_2/F^+_1$ such that
\begin{itemize}
\item $\psi|_{G_{F^+_2}}$ is crystalline at all places dividing
  $l$, and unramified at all other finite places.
 \item If $w$ is a place of $F^+_2$ lying over a place $v\in S$,
    then $\tau_v|_{G_{F^+_{2,w}}}$ is trivial.
    \item $F^+_2$ is linearly disjoint from
      $\overline{F^+}^{\ker\rbar}(\zeta_l)$ over $F^+$.  \item
      $\rbar|_{G_{F^+_{2,w}}}$ is trivial for all places $w|l$ of
      $F_2^+$.
\item  $[F^+_2:\Q]$
      is even.
\end{itemize}Let $\pi_2$ be the base change of $\pi_1$ to
$F^+_2$. Note that $\psi|_{G_{F^+_2}}$ and $\det r_{l,\iota}(\pi_2)$ are
crystalline characters with the same Hodge-Tate weights, and are both
unramified at all places not dividing $l$, so they differ by a finite
order totally even character which is unramified at all finite places.
Thus we may replace $\pi_2$ by a twist by a finite order character to
ensure that $\det r_{l,\iota}(\pi_2)=\psi$,
without affecting any of the other properties of $\pi_2$ noted above. The
result follows (since the hypotheses of the theorem imply that
$\rbar|_{G_{F^+(\zeta_l)}}$ is irreducible).
\end{proof}

We thank Brian Conrad for showing us the argument for case 2 in the
proof of the following lemma.

\begin{lem}
  \label{lem:existence of local pot BT lift}Suppose $l$ is an odd prime. Let $K$ be a finite
  extension of $\Ql$, and $\rbar:G_K\to\GL_2(\Flbar)$ a continuous
  representation, with \[\rbar\cong
  \begin{pmatrix}
    \psibar_1&*\\0&\psibar_2\epsilonbar^{-1}
  \end{pmatrix}\]for some characters $\psibar_1$, $\psibar_2$. Let
  $\chi:G_K\to\Zlbar^\times$ be a finite order character lifting
  $\psibar_1\psibar_2$. Then there is a continuous potentially
  Barsotti-Tate representation
  $r:G_K\to\GL_2(\Zlbar)$ of determinant $\chi\epsilon^{-1}$ lifting $\rbar$ with  \[r\cong
  \begin{pmatrix}
    \psi_1&*\\0&\psi_2\epsilon^{-1}
  \end{pmatrix}\]for some potentially unramified characters $\psi_1$, $\psi_2$
  lifting $\psibar_1$, $\psibar_2$.
\end{lem}
\begin{proof}
  We first remark that any lift of the form  \[r\cong
  \begin{pmatrix}
    \psi_1&*\\0&\psi_2\epsilon^{-1}
  \end{pmatrix}\] is automatically potentially crystalline if $\psi_1$
  and $\psi_2$ are finitely ramified
  $\psi_1\neq\psi_2$.

  \medskip{\sl Case 1: $\psibar_1 \psibar_2^{-1} \neq 1$:} Choose an
  arbitrary finite-order lift $\psi_1$ of $\psibar_1$, and set
  $\psi_2=\chi\psi_1^{-1}$. Let $E \subset \Qlbar$ be a finite
  extension of $\bb{Q}_l$ with ring of integers $\mc{O}$ and residue
  field $\bb{F}$ such that $\psi_1$ and $\psi_2$ are valued in $\mc{O}^\times$.
 By the remark above, it suffices to
  show that the natural map
  $H^{1}(G_K,\mc{O}(\psi_1 \psi_2^{-1}\epsilon)) \to
  H^1(G_K,\bb{F}(\psibar_1\psibar_2^{-1}\epsilonbar))$ is
  surjective. However, the cokernel of this map is a submodule of
  $H^2(G_K,\mc{O}(\psi_1\psi_2^{-1}\epsilon))$ which, by Tate-duality,
  is Pontryagin dual to
  $H^0(G_K,(E/\mc{O})(\psi_2\psi_1^{-1}))=\{0\}$.

  \medskip{\sl Case 2: $\psibar_1 \psibar_2^{-1} =1$:} Choose finitely ramified lifts $\psi_1$ and
  $\psi_2$ of $\psibar_1$ and $\psibar_2$ such that
  $\chi=\psi_1\psi_2$. Let $\psi=\psi_1\psi_2^{-1}$. Choose $E \subset
  \Qlbar$, a finite extension of $\bb{Q}_l$ with ring of integers
  $\mc{O}$ and residue field $\bb{F}$ such that $\psi_1$ and $\psi_2$
  are valued in $\mc{O}^\times$. Let $L$ be the line in
  $H^1(G_K,\bb{F}(\epsilonbar))$ determined $\rbar$ (an extension of
  $\psibar_2\epsilonbar^{-1}$ by $\psibar_1$) and let $H$ be the
  hyperplane in $H^1(G_K,\bb{F})$ which annihilates $L$ under the Tate
  pairing. Let $\varpi \in \mc{O}$ be a uniformizer. Let $\delta_1 : H^1(G_K,\bb F(\overline{\epsilon})) \to
  H^2(G_K,\mc{O}(\psi\epsilon))$ be the map coming from
  the exact sequence $0\to \mc{O}(\psi\epsilon)\stackrel{\varpi}{\to}\mc
  O(\psi\epsilon)\to \bb F(\overline{\epsilon})\to 0$ of
  $G_K$-modules. We need to show that $\delta_1(L)=0$.

  Let $\delta_0$ be the map
  $H^0(G_K,(E/\mc{O})(\psi^{-1})) \to H^{1}(G_K,\bb{F})$ coming from
  the exact sequence $0 \to \bb{F} \to (E/\mc{O})(\psi^{-1})
  \stackrel{\varpi}{\to} (E/\mc{O})(\psi^{-1}) \to 0$ of
  $G_K$-modules.  By Tate-duality, the condition that $L$ vanishes
  under the map $\delta_1$ is equivalent to the condition that the
  image of the map $\delta_0$ is contained in $H$.  Let $n \geq 1$ be
  the largest integer with the property that $\psi^{-1} \equiv 1 \mod
  \varpi^n$. Then we can write $\psi^{-1}(x)= 1+\varpi^n \alpha(x)$
  for some function $\alpha : G_K \to \mc{O}$. Let $\overline{\alpha}$
  denote $\alpha \mod \varpi : G_K \to \bb{F}$. Then
  $\overline{\alpha}$ is additive and the choice of $n$ ensures that
  it is non-trivial. It is straightforward to check that the image of
  the map $\delta_0$ is the line spanned by $\overline{\alpha}$. If
  $\overline{\alpha}$ is in $H$, we are done. Suppose this is not the
  case. We break the rest of the proof into two cases.

  \medskip{\sl Case 2a: $\psibar_1 \psibar_2^{-1} = 1$, and $\rbar$ is
    tr\'es ramifi\'e:} We can and do suppose that we have chosen
  $\psi_1$ and $\psi_2$ so that $\psi$ is ramified and further, that
  $\overline{\alpha}$ is ramified.  The fact that $\rbar$ is tr\'es
  ramifi\'e implies that $H$ does not contain the unramified line in
  $H^1(G_K,\bb{F})$. Thus there is a unique $\overline{x} \in
  \bb{F}^\times$ such that $\overline{\alpha}+u_{\overline{x}} \in H$
  where $u_{\overline{x}}: G_K\to \bb{F}$ is the unramified
  homomorphism sending $\Frob_K$ to $\overline{x}$. Let $y$ be an
  element of $\mc{O}^\times$ with
  $2\overline{y}=\overline{x}$. Replacing $\psi_1$ with $\psi_1$ times
  the unramified character sending $\Frob_K$ to $(1+\varpi^n y)^{-1}$
  and $\psi_2$ with $\psi_2$ times the unramified character sending
  $\Frob_K$ to $1+\varpi^ny$, we are done.

 \medskip{\sl Case 2b: $\psibar_1 \psibar_2^{-1} = 1$, and $\rbar$
    is peu ramifi\'e:} Making a ramified extension of $\mc{O}$ if
  necessary, we can and do assume that $n\geq 2$. The fact  that
  $\rbar$ is peu ramifi\'e implies that $H$ contains the unramified
  line. It follows that if we replace
  $\psi_1$ with $\psi_1$ times
  the unramified character sending $\Frob_K$ to $1+\varpi$ and
  $\psi_2$ with $\psi_2$ times the unramified character sending
  $\Frob_K$ to $(1+\varpi)^{-1}$, then we are done (as the new
  $\overline{\alpha}$ will be unramified). 
\end{proof}

Using the previous two results, we can now give our main result on the
existence of ordinary modular lifts.

\begin{thm}
  \label{cor: main result on existence of ordinary lifts.}
  Let $l\ge 3$ be prime, and let
  $F^+$ be a totally real field. Let
  $\rbar:G_{F^+}\to\GL_2(\Flbar)$ be irreducible and modular. Fix a character
  $\psi:G_{F^+}\to\Qlbar^\times$ such that $\epsilon\psi$ has finite
  order, and $\bar{\psi}=\det\rbar$. 
Assume further that
\begin{itemize}
\item For every place $v|l$ of $F^+$, $\rbar|_{G_{F^+_v}}$ is
  reducible.
\item Either
  \begin{itemize}
  \item
    \begin{itemize}
    \item $l\ge 5$.
    \item $\rbar(G_{F^+(\zeta_l)})$ is $2$-big.
    \item
      $[\overline{F^+}^{\ker\ad\rbar}(\zeta_l):\overline{F^+}^{\ker\ad\rbar}]>2$.
    \end{itemize}Or:
  \item
    \begin{itemize}
    \item There is a quadratic CM extension $M/F^+$, with $M$ not
      equal to the quadratic extension of $F^+$ in $F^+(\zeta_l)$, and
      a continuous character $\bar{\theta}:G_M\to\Flbar^\times$ such
      that $\rbar=\Ind_{G_M}^{G_{F^+}}\bar{\theta}$.
    \end{itemize}
  \end{itemize}
\end{itemize}
Then there is a continuous representation
$r:G_{F^+}\to\GL_2(\Zlbar)$ lifting $\rbar$ such that
\begin{itemize}
\item $r$ is modular.
\item For each place $v|l$ of $F^+$, $r|_{G_{F^+_v}}$ is
  potentially Barsotti-Tate and ordinary.
\item $\det r = \psi$.
\end{itemize}
\end{thm}
\begin{proof}
For each place $v|l$ of $F^+$, using Lemma \ref{lem:existence of local
  pot BT lift}, choose a potentially Barsotti-Tate lift $r^{(v)}:
G_{F^+_v} \to \GL_2(\Zlbar)$ of $\rbar|_{G_{F^+_v}}$ with
 \[r^{(v)} \cong \begin{pmatrix}
    \psi_1^{(v)}&*\\0&\psi^{(v)}_2\epsilon^{-1}
  \end{pmatrix}\] for some finitely ramified characters
  $\psi^{(v)}_1$, $\psi^{(v)}_2$ with $\psi^{(v)}_1\psi^{(v)}_2 =
  \epsilon\psi|_{G_{F_v}}$.  Let $S$ denote the set of primes of $F^+$
  not dividing $l$ at which $\rbar$ or $\psi$ is ramified.  For each
  $v \in S$, choose a lift $r^{(v)}$ of $\rbar|_{G_{F^+_v}}$ of
  determinant $\psi|_{G_{F^+_v}}$ (that this is possible follows
  easily from Lemma 3.1.4 of \cite{kis04}, which shows that $\rbar$
  has a (global) modular lift of determinant $\psi$). For each $v\in S
  \cup \{v|l\}$, let $\tau_v$ be the inertial type of $r^{(v)}$ and
  let $R_v$ denote an irreducible component of the lifting ring
  $R^{\square,\psi,\tau_v}_{\rbar|_{G_{F^+_v}}}[1/l]$ containing
  $r^{(v)}$. (The closed points of this ring correspond to lifts
  which have type $\tau_v$, determinant $\psi$ and are potentially
  Barsotti-Tate if $v|l$.) We note that if $v|l$, then $R_v$ is unique
  and moreover any other lift of $\rbar|_{G_{F^+_v}}$ corresponding
  to a closed point of $R_v[1/l]$ is ordinary. The result now follows
  from Theorem \ref{thm: main result on existence of ordinary lifts
    with types everywhere.}.
\end{proof}

Combining this result with the improvements made in \cite{MR2280776}
to the modularity lifting theorem of \cite{kis04}, we obtain the
following theorem, where in contrast to previous results in the area
we do not need to assume that $\rbar$ has a modular lift which is
ordinary at a specified set of places.

\begin{thm}\label{thm: modularity lifting theorem for BT without ord assumption}
  Let $l\ge 3$ be a prime, $F^+$ a totally real field, and
  $r:G_{F^+}\to\GL_2(\Qlbar)$ a continuous representation
  unramified outside of a finite set of primes, with determinant a
  finite order character times the inverse of the cyclotomic character. Suppose
  further that
  \begin{enumerate}
  \item $r$ is potentially Barsotti-Tate for each $v|l$.
  \item $\rbar$ is modular and irreducible.
  \item Either
    \begin{itemize}
    \item
      \begin{itemize}
      \item $l\ge 5$.
      \item $\rbar(G_{F^+(\zeta_l)})$ is $2$-big.
      \item
        $[\overline{F^+}^{\ker\ad\rbar}(\zeta_l):\overline{F^+}^{\ker\ad\rbar}]>2$.
      \end{itemize}Or:
    \item
      \begin{itemize}
      \item There is a quadratic CM extension $M/F^+$, with $M$ not
        equal to the quadratic extension of $F^+$ in $F^+(\zeta_l)$,
        and a continuous character $\bar{\theta}:G_M\to\Flbar^\times$
        such that $\rbar=\Ind_{G_M}^{G_{F^+}}\bar{\theta}$.
      \end{itemize}
    \end{itemize}
  \end{enumerate}Then $r$ is modular.
\end{thm}
\begin{proof}This follows from Theorem \ref{thm: main result on
    existence of ordinary lifts with types everywhere.}, together with
  Theorem 1.1 of \cite{MR2480560}. More precisely, let $S$ be the set
  of places of $F^+$ for which $r|_{G_{F^+_v}}$ is ramified (so
  that $S$ automatically contains all the places dividing $l$). Let
  $\psi=\det r$. After conjugating, we may assume that $r$ takes
  values in $\GL_2(\Zlbar)$. For each $v\in S$, let $\tau_v$ be the
  type of $r|_{G_{F^+_v}}$, and let $R_v$ be a component of the
  lifting ring $R^{\square,\psi,\tau_v}[1/l]$ corresponding to
  $r|_{G_{F^+_v}}$. (The closed points of this ring correspond to
  lifts which have type $\tau_v$, determinant $\psi$ and are
  potentially Barsotti-Tate if $v|l$.) Applying Theorem \ref{thm: main
    result on existence of ordinary lifts with types everywhere.}, we
  deduce that there is a modular lift $r'$ of $\rbar$ which is
  potentially Barsotti-Tate for all places $v|l$, with
  $r'|_{G_{F^+_v}}$ potentially ordinary for precisely the places
  at which $r|_{G_{F^+_v}}$ is potentially ordinary.  The
  modularity of $r$ then follows immediately from Theorem 1.1 of
  \cite{MR2480560}.
\end{proof}

We can use the results of section \ref{sec:big image GL2} to deduce
corollaries of these theorems in which the conditions on the image of
$\rbar$ are more explicit. For example, we have the following
result.

\begin{thm}
  \label{thm: main result on existence of ordinary lifts with types
    everywhere - clean hypotheses.} Let $l\ge 7$ be prime, and let
  $F^+$ be a totally real field. Let $\rbar:G_{F^+}\to\GL_2(\Flbar)$
  be irreducible and modular. Fix a character
  $\psi:G_{F^+}\to\Qlbar^\times$ such that $\epsilon\psi$ has finite
  order, and $\bar{\psi}=\det\rbar$. Let $S$ denote a finite set of
  finite places of $F^+$ containing all places at which $\rbar$ or
  $\psi$ is ramified, and all places dividing $l$. 

  For each place $v\in S$, fix an inertial type $\tau_v$ of
  $I_{F^+_v}$ on a $\Qlbar$-vector space, of determinant
  $(\psi\epsilon)|_{I_{F^+_v}}$. Assume that for each place $v\in
  S$, $v\nmid l$, $\rbar|_{G_{F^+_v}}$ has a lift of type $\tau_v$
  and determinant $\psi|_{G_{F^+_v}}$, and for each place $v\in S$,
  $v| l$, $\rbar|_{G_{F^+_v}}$ has a potentially Barsotti-Tate lift
  of type $\tau_v$ and determinant $\psi|_{G_{F^+_v}}$. For each place
  $v\in S$, we let $R_v$ denote an irreducible component of the
  corresponding lifting ring $R^{\square,\psi,\tau_v}[1/l]$ for
  (potentially Barsotti-Tate) lifts of type $\tau_v$ and determinant
  $\psi|_{G_{F^+_v}}$.

Assume further that
  \begin{itemize}

  \item $[F^+(\zeta_l):F^+]>4$.
  \item $\rbar|_{G_{F^+(\zeta_l)}}$ is irreducible.

  \end{itemize}
Then there is a continuous representation
$r:G_{F^+}\to\GL_2(\Zlbar)$ lifting $\rbar$ of determinant $\psi$ such that
\begin{itemize}
\item $r$ is modular.
\item $r$ is unramified at all places $v\notin S$.
\item For each place $v|l$ of $F^+$, $r|_{G_{F^+_v}}$ is
  potentially Barsotti-Tate of type $\tau_v$ (and indeed corresponds
  to a point of $R_v$).
\item For each place $v\in S$, $v\nmid l$, $r|_{G_{F^+_v}}$ has type $\tau_v$ (and indeed corresponds
  to a point of $R_v$).
\end{itemize}
\end{thm}
\begin{proof}
  By Lemma 4.4.2 and its proof, we see that if the projective image of
  $\rbar$ is not dihedral or $A_4$, then $\rbar(G_{F^+(\zeta_l)})$
  is 2-big. Furthermore, in these cases the image of $\ad\rbar$ is
  either of the form $\PGL_2(k)$, $\PSL_2(k)$, $S_4$ or $A_5$, so that
  $\ad\rbar(G_{F^+})$ has no cyclic quotients of order greater than
  2. By the assumption that $[F^+(\zeta_l):F^+]>4$, we deduce that
  $[\overline{F^+}^{\ker\ad\rbar}(\zeta_l):\overline{F^+}^{\ker\ad\rbar}]>2$,
  so in this case the result follows directly from Theorem \ref{thm:
    main result on existence of ordinary lifts with types
    everywhere.}.

  In the remaining cases, we see from Lemma 4.4.2 that there are
  extensions $K_2/K_1/F^+$, with $K_2/K_1$ quadratic and $K_1/F^+$
  cyclic of degree $1$ or $3$, together with a continuous character
  $\thetabar:K_2\to\Flbar^\times$ such that
  $\rbar|_{G_{K_1}}\cong\Ind_{G_{K_2}}^{G_{K_1}}\thetabar$. Then $K_1$
  is totally real (as it is an extension of $F^+$ of odd degree), and
  $K_2$ is CM (because $\rbar$, being modular, is totally odd).
  Furthermore, $\rbar|_{G_{K_1^+(\zeta_l)}}$ is irreducible because
  $\rbar|_{G_{F^+(\zeta_l)}}$ is irreducible, by assumption, and
  $K_1/F^+$ is cyclic of degree 1 or 3.  Thus we can apply Proposition
  \ref{prop: potential ordinary lift in the dihedral case.}  to
  $\rbar|_{G_{K_1}}$, and the result follows as in the proof of
  Theorem \ref{thm: main result on existence of ordinary lifts with
    types everywhere.}.
\end{proof}

From this, we immediately deduce the following versions of
Theorems \ref{cor: main result on existence of ordinary lifts.} and
\ref{thm: modularity lifting theorem for BT without ord assumption}.

\begin{thm}
  \label{thm: main existence of ord lifts theorem with ``clean'' hypotheses}Let $l\ge 7$ be prime, and let
  $F^+$ be a totally real field. Let
  $\rbar : G_{F^+}\to\GL_2(\Flbar)$ be a modular representation such
  that $\rbar|_{G_{F^+(\zeta_l)}}$ is
  irreducible. Assume that $[F^+(\zeta_l):F^+]>4$. Assume further that
  \begin{itemize}
  \item For every place $v|l$ of $F^+$, $\rbar|_{G_{F^+_v}}$ is
    reducible.
  \end{itemize}
  Then there is a continuous representation
  $r:G_{F^+}\to\GL_2(\Zlbar)$ lifting $\rbar$ such that
\begin{itemize}
\item $r$ is modular.
\item For each place $v|l$ of $F^+$, $r|_{G_{F^+_v}}$ is
  potentially Barsotti-Tate and ordinary.
\end{itemize}
\end{thm}

\begin{thm}
  \label{thm: modularity lifting theorem, clean hypotheses.} Let $l\ge 7$ be a prime, $F^+$ a totally real field, and
  $r:G_{F^+}\to\GL_2(\Qlbar)$ a continuous representation
  unramified outside of a finite set of primes, with determinant a
  finite order character times the inverse of the cyclotomic character. Suppose
  further that
  \begin{enumerate}
  \item $r$ is potentially Barsotti-Tate for each $v|l$.
  \item $\rbar$ is modular.
  \item $\rbar|_{G_{F^+(\zeta_l)}}$ is irreducible.
  \item $[F^+(\zeta_l):F^+]>4$.
\end{enumerate}
   Then $r$ is modular.
\end{thm}

We also deduce improvements such as the following to the main theorem of
\cite{gee08serrewts}, to which we refer the reader for the definitions
of the terminology used.
\begin{thm}
  Let $l\ge 7$ be prime which is unramified in a totally real field $F^+$, and let $\rbar:G_{F^+}\to\GL_2(\Flbar)$ be
  an irreducible modular representation. Assume that
  $\rbar|_{G_{F^+(\zeta_l)}}$ is irreducible.

Let $\sigma$ be a regular weight. Then $\rbar$ is modular of weight
$\sigma$ if and only if $\sigma\in W(\rbar)$.
\end{thm}
\begin{proof}
  Note that since $l\ge 7$ and $l$ is unramified in $F^+$,
  $[F^+(\zeta_l):F^+]>4$. The result now follows by replacing the use
  of Corollary 3.1.7 of \cite{gee061} with Theorem \ref{thm: main result on existence of ordinary lifts with types
    everywhere - clean hypotheses.} above.
\end{proof}

\bibliographystyle{amsalpha} 
\bibliography{barnetlambgeegeraghty}
\end{document}